\documentclass{amsproc}

\usepackage{stmaryrd}
\usepackage{amssymb}
\usepackage{graphicx}
\usepackage{amsmath}
\usepackage{mathrsfs}
\usepackage{amsxtra}
\usepackage{titletoc}
\usepackage[all,cmtip]{xy}
\usepackage{pictexwd}
\usepackage[german,english]{babel}
\usepackage{graphicx} 

\theoremstyle{definition}

\newtheorem{thm}{Theorem}[section]
\newtheorem{cor}[thm]{Corollary}
\newtheorem{lem}[thm]{Lemma}
\newtheorem{prop}[thm]{Proposition}
\theoremstyle{definition}
\newtheorem{defn}[thm]{Definition}
\theoremstyle{remark}
\newtheorem{rem}[thm]{Remark}
\numberwithin{equation}{section}
\newtheorem{eg}[thm]{Example}

\theoremstyle{remark}

\numberwithin{equation}{section}

\newcommand{\To}{\longrightarrow}

\newcommand{\AAA}{\mathbb{A}}
\newcommand{\PPP}{\mathbb P}
\newcommand{\NNN}{\mathbb N}
\newcommand{\ZZZ}{\mathbb Z}
\newcommand{\QQQ}{\mathbb Q}
\newcommand{\RRR}{\mathbb R}
\newcommand{\CCC}{\mathbb C}
\newcommand{\FFF}{\mathbb F}
\newcommand{\K}{\mathbb K}
\newcommand{\LL}{\mathbb L}
\newcommand{\F}{\mathcal{F}}
\newcommand{\too}{\rightarrow}
\newcommand{\ignore}{}
\newcommand{\recognize}{}

\newcommand{\wt}{\widetilde}
\newcommand{\wh}{\widehat}
\newcommand{\Spec}{\mathrm{Spec}}
\newcommand{\Proj}{\mathrm{Proj}}
\newcommand{\Rees}{\mathrm{Rees}}
\newcommand{\Quot}{\mathrm{Quot}}
\newcommand{\inv}{^{-1}}
\newcommand{\calo}{{\mathcal O}}
\newcommand{\mm}{{\mathit m}}
\newcommand{\isom}{\simeq}
\newcommand{\inin}{\mathrm{in}}
\newcommand{\invv}{\mathrm{inv}}
\newcommand{\ord}{\mathrm{ord}}
\newcommand{\ttop}{\mathrm{top}}
\newcommand{\Sing}{\mathrm{Sing}}
\newcommand{\Reg}{\mathrm{Reg}}
\newcommand\hint{\!\!{}^\triangleright\,}
\newcommand\challenge{\!\!{}^+\,}
\newcommand\examples{\Large Examples}

\long\def\ignore#1\recognize{}

\def\comment#1{}

\renewcommand{\citeform}[1]{{\mdseries#1}}

\begin{document}


\title{Blowups and Resolution}


\author{Herwig Hauser }
\maketitle


\hfill \textit{To the memory of Sheeram Abhyankar,} 

 \hfill \textit{with great respect.}\\

\noindent This manuscript originates from a series of lectures the author\footnote {Supported in part by the Austrian Science Fund FWF within the projects P-21461 and P-25652.} gave at the Clay Summer School on Resolution of Singularities at Obergurgl in the Tyrolean Alps in June 2012. A hundred young and ambitious students gathered for four weeks to hear and learn about resolution of singularities. Their interest and dedication became essential for the success of the school.

The reader of this article is ideally an algebraist or geometer having a rudimentary acquaintance with the main results and techniques in resolution of singularities. The purpose is to provide  quick and concrete information about specific topics in the field. As such, the article is modelled like a dictionary and not particularly suited to be read from the beginning to the end (except for those who like to read dictionaries). To facilitate the understanding of selected portions of the text without reading the whole earlier material, a certain repetition of definitions and assertions has been accepted.

Background information on the historic development and the motivation behind the various constructions can be found in the cited literature, especially in \cite{Obergurgl_Book, Ha_BAMS_1, Ha_BAMS_2, FH_BAMS, Cutkosky_Book, Kollar_Book, Lipman_Introduction}. Complete proofs of several more technical results appear in \cite{EH}.

All statements are formulated for algebraic varieties and morphism between them. They are mostly also valid, with the appropriate modifications, for schemes.  In certain cases, the respective statements for schemes are indicated separately.

Each chapter concludes with a broad selection of examples, ranging from computational exercises to suggestions for additional material which could not be covered in the text. Some more challenging problems are marked with a superscript $\,\,\challenge$. The examples should be especially useful for people planning to give a graduate course on the resolution of singularities. Occasionally the examples repeat or specialize statements which have appeared in the text and which are worth to be done personally before looking at the given proof. In the appendix, hints and answers to a selection of examples marked by a superscript $\,\, \hint$ are given.

Several results appear without proof, due to lack of time and energy of the author. Precise references are given whenever possible. The various survey articles contain complementary bibliography.

The Clay Mathematics Institute chose resolution of singularities as the topic of the 2012 summer school. It has been a particular pleasure to cooperate in this endeavour with its research director David Ellwood, whose enthusiasm and interpretation of the school largely coincided with the approach of the organizers, thus creating a wonderful working atmosphere. His sensitiveness of how to plan and realize the event has been exceptional.

The CMI director Jim Carlson and the CMI secretary Julie Feskoe supported very efficiently the preparation and realization of the school. 

Xudong Zheng provided a preliminary write-up of the lectures, Stefan Perlega and Eleonore Faber completed several missing details in preliminary drafts of the manuscript. Faber also produced the two visualizations. The discussion of the examples in the appendix was worked out by Perlega and Valerie Roitner. Anonymous referees helped substantially with their remarks and criticism to eliminate deficiencies of the exposition. Barbara Beeton from the AMS took care of the TeX-layout. All this was very helpful.

\goodbreak


\tableofcontents
\goodbreak



\section{Lecture I: A First Example of Resolution}

\noindent Let $X$ be the zeroset in affine three space $\AAA^3$ of the polynomial 
\[
f = 27 x^2y^3z^2 + (x^2 + y^3 - z^2)^3
\]
over a field $\K$ of characteristic different from $2$ and $3$. This is an algebraic surface, called {\em Camelia}, with possibly singular points and curves, and certain symmetries. For instance, the origin $0$ is singular on $X$, and $X$ is symmetric with respect to the automorphisms of $\AAA^3$ given by replacing $x$ by $-x$ or $z$ by $-z$, and also by replacing $y$ by $-y$ while interchanging $x$ with $z$. Sending $x$, $y$ and $z$ to $t^3x$ , $t^2y$ and $t^3z$ for $t\in \K$ also preserves $X$. See figure 1 for a plot of the real points of $X$. The intersections of $X$ with the three coordinate hyperplanes of $\AAA^3$ are given as the zerosets of the equations\goodbreak
\[
x=(y^3 - z^2)^3 = 0,\]
\[ 
y=(x^2 - z^2)^3 = 0,\]
\[
z=(x^2 + y^3)^3 = 0.
\]
These intersections are plane curves: two perpendicular cusps lying in the $xy$- and $yz$-plane, respectively the union of the two diagonals in the $xz$-plane. The singular locus $\Sing(X)$ of $X$ is given as the zeroset of the partial derivatives of $f$ inside $X$. This yields for $\Sing(X)$ the additional equations 
\[
x\cdot [9y^3z^2 + (x^2 + y^3 - z^2)^2]=0,
\]
\[
y^2\cdot[9x^2z^2 + (x^2 + y^3 - z^2)^2]=0,
\]
\[
z\cdot[9x^2y^3 - (x^2 + y^3 - z^2)^2]=0.
\]
Combining these equations with $f=0$, it results that the singular locus of $X$ has six irreducible components, defined respectively by 
\[
x=y^3-z^2=0,
\]
\[
z=x^2+y^3=0,
\]
\[
y=x+z=0,
\]
\[
y=x-z=0,
\]
\[
x^2-y^3=x+\sqrt{-1}\cdot z=0,
\]
\[
x^2-y^3=x-\sqrt{-1}\cdot z=0.
\]
The first four components of $\Sing(X)$ coincide with the four curves given by the three coordinate hyperplane sections of $X$. The last two components are plane cusps in the hyperplanes given by $x\pm\sqrt{-1}\cdot z=0$. At points $a\neq 0$ on the first two singular components of $\Sing(X)$,  the intersections of $X$ with a plane through $a$ and transversal to the component are again cuspidal curves.


\begin{figure}[h]
\centering
\includegraphics[scale=0.2]{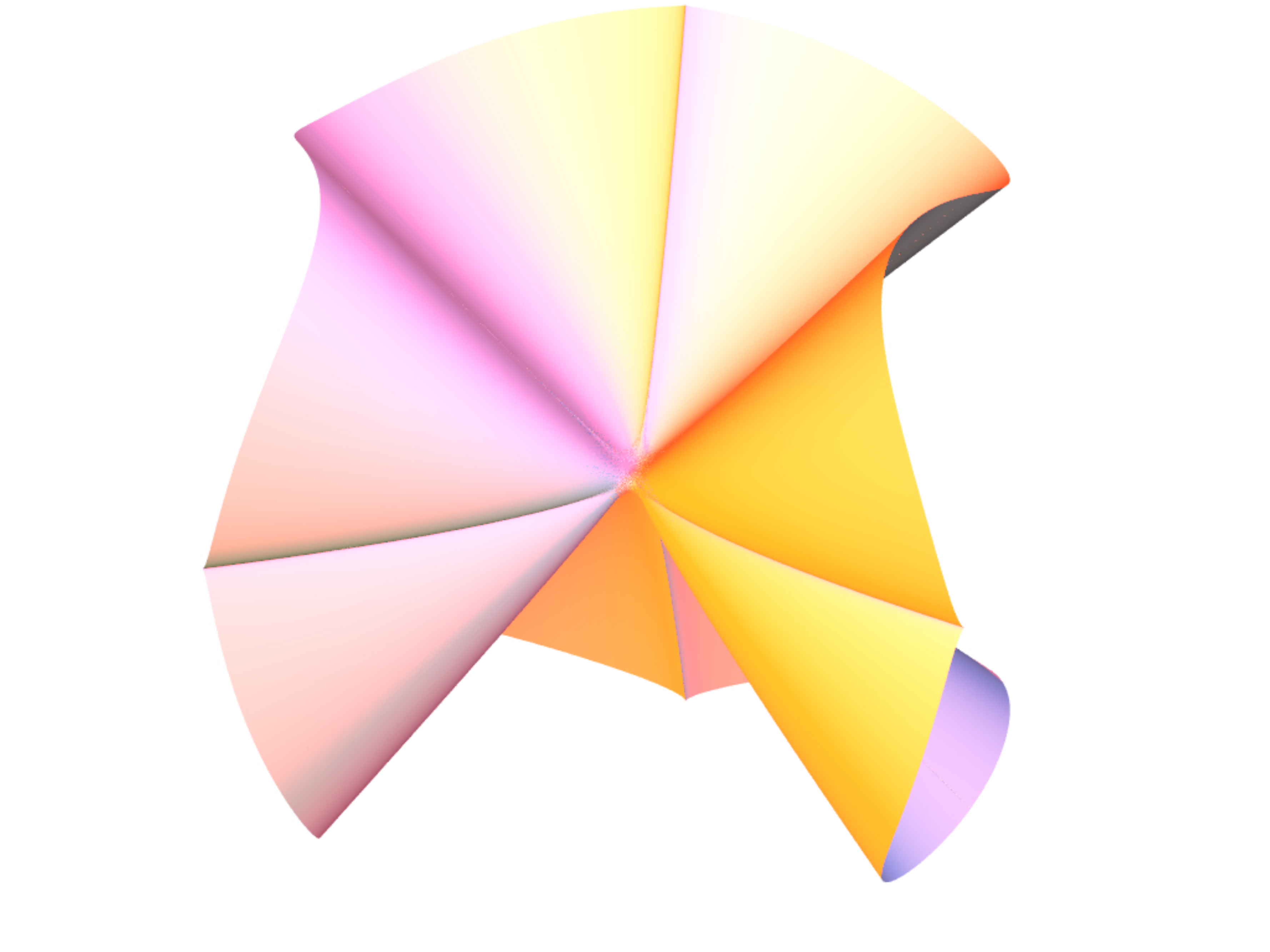}
\vskip-.5cm\caption{The surface \textit{Camelia}: $27 x^2y^3z^2 + (x^2 + y^3 - z^2)^3=0$.}
\end{figure}

\goodbreak

Consider now the surface $Y$ in $\AAA^4$ which is given as the cartesian product $C\times C$ of the plane cusp $C: x^2-y^3=0$ in $\AAA^2$ with itself. It is defined by the equations $x^2-y^3=z^2-w^3=0$. The singular locus $\Sing(Y)$ is the union of the two cusps $C_1=C\times 0$ and $C_2=0\times C$ defined by $x^2-y^3 = z=w=0$, respectively $x=y=z^2-w^3=0$. The surface $Y$ admits the parametrization 
\[
\gamma:\AAA^2\too\AAA^4, (s,t)\mapsto (s^3,s^2,t^3,t^2).
\]
The image of $\gamma$ is $Y$. The composition of $\gamma$ with the linear projection 
\[
\pi: \AAA^4 \too \AAA^3, (x, y, z,w) \mapsto (x, -y + w, z)
\]
yields the map
\[
\delta=\pi\circ \gamma:\AAA^2\too\AAA^3, (s,t)\mapsto (s^3,-s^2+t^2,t^3).
\]
Replacing in the polynomial $f$ of $X$ the variables $x$, $y$, $z$ by $s^3$, $-s^2+t^2$, $t^3$ gives $0$. This shows that the image of $\delta$ lies in $X$. As $X$ is irreducible of dimension $2$ and $\delta$ has  rank $2$ outside $0$ the image of $\delta$ is dense in (and actually equal to) $X$. Therefore the image of $Y$ under $\pi$ is dense in $X$: \comment{[actually, $\pi(Y)$ should equal $X$]} This interprets $X$ as a contraction of $Y$ by means of the projection $\pi$ from $\AAA^4$ to $\AAA^3$. The two surfaces $X$ and $Y$ are not isomorphic because, for instance, their singular loci have a different number of components. The simple geometry of $Y$ as a cartesian product of two plane curves is scrambled up when projecting it down to $X$.

The point blowup of $Y$ in the origin produces a surface $Y_1$ whose singular locus has two components. They map to the two components $C_1$ and $C_2$ of $\Sing(Y)$ and are regular and transversal to each other. The blowup $X_1$ of $X$ at $0$ will still be the image of $Y_1$ under a suitable projection. The four singular components of $\Sing(X)$ will become regular in $X_1$ and will either meet pairwise transversally or not at all. The two regular components of $\Sing(X)$ will remain regular in $X_1$ but will no longer meet each other, cf.÷ figure 2.


\begin{figure}[h]
\centering
\includegraphics[scale=0.2]{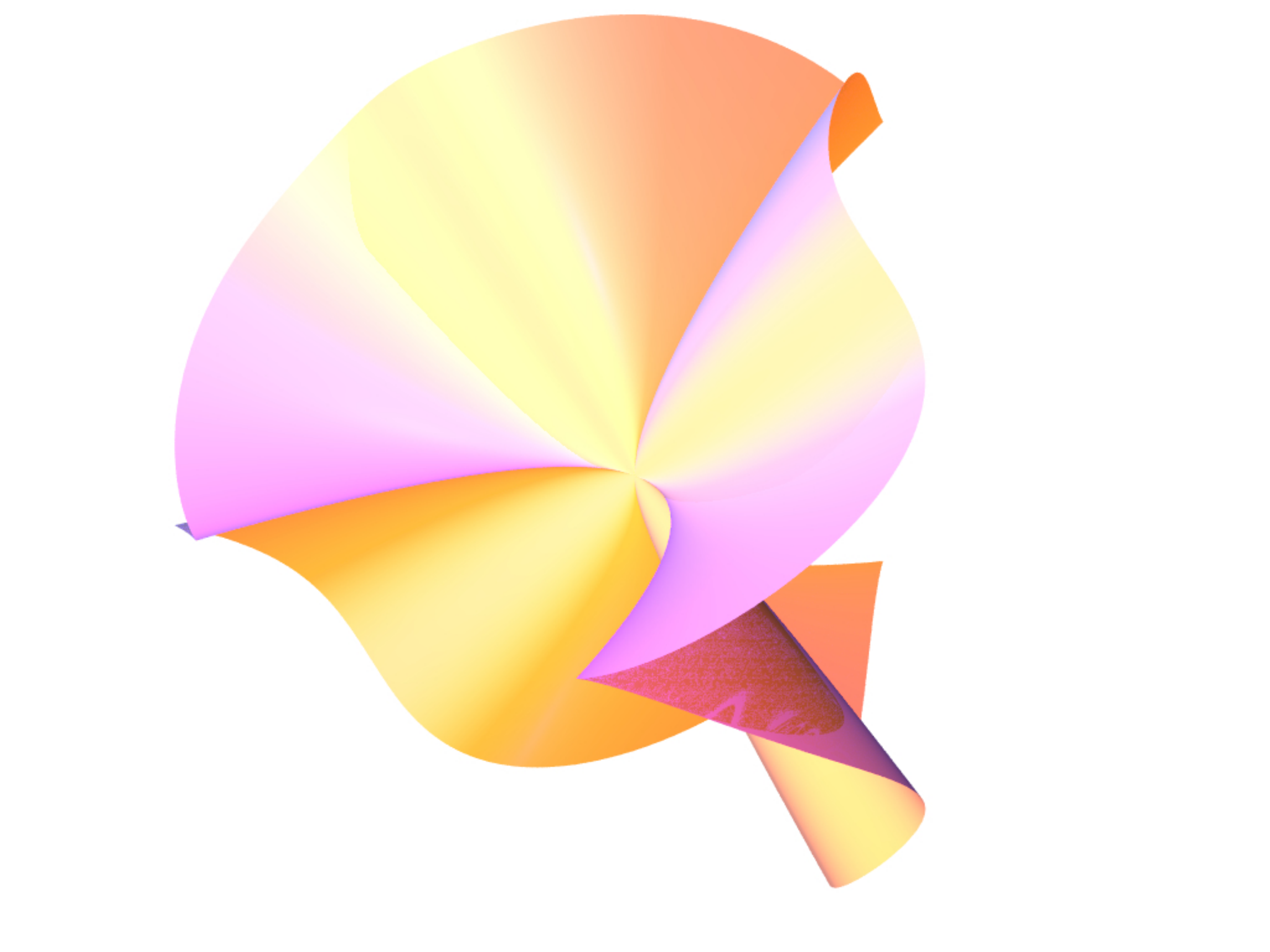}
\vskip-.5cm\caption{The surface $X_1$ obtained from \textit{Camelia} by a point blowup.}
\end{figure}


The point blowup of $Y_1$ in the intersection point of the two curves of $\Sing(Y_1)$ separates the two curves and yields a surface $Y_2$ whose singular locus consists of two disjoint regular curves. Blowing up these separately yields a regular surface $Y_3$ and thus resolves the singularities of $Y$. The resolution of the singularities of $X$ is more complicated, see the examples below. \comment{[Determine the resolution of $X$]} 

\goodbreak

\bigskip
\noindent\textit{\examples}


\begin{eg} \label{1.1}$\hint$ Show that the surface $X$ defined in $\AAA^3$ by $27x^2y^3z^2+(x^2+y^3-z^2)^3=0$ is the image of the cartesian product $Y$ of the cusp $C: x^3-y^2=0$ with itself under the projection from $\AAA^4$ to $\AAA^3$ given by $(x,y,z,w)\mapsto (x, -y + w, z)$. \end{eg}


\begin{eg} \label{1.2}$\hint$  Find additional symmetries of $X$ aside from those mentioned in the text.   
\end{eg}


\begin{eg} \label{1.3}$\hint$  Produce a real visualization of the surface obtained from \textit{Camelia }by replacing in the equation $z$ by $\sqrt{-1}\cdot z$.  Determine the components of the singular locus. \end{eg}


\begin{eg} \label{1.4}$\hint$ Consider at the point $a=(0, 1,1)$ of $X$ the plane $P:2y+3z=5$ through $a$. It is transversal to the component of the singular locus of $X$ passing through $a$ (i.e., this component is regular at $a$ and its tangent line at $a$ does not lie in $P$). Determine the singularity of $X\cap P$ at $a$. The normal vector $(0,2,3)$ to $P$ is the tangent vector at $t=1$ of the parametrization $(0,t^2,t^3)$ of the component of $\Sing(X)$ defined by $x=y^3-z^2$.
\end{eg}


\begin{eg} \label{1.5}  The point $a$ on $X$ with coordinates $(1, 1,\sqrt{-1})$ is a singular point of $X$, and a non-singular point of the curve of $\Sing(X)$ passing through it. A plane transversal to $\Sing(X)$ at $a$ is given e.g.÷ by 
\[
P: 3x+2y+3\sqrt{-1}\cdot z=9.
\]
Determine the singularity of $X\cap P$ at $a$. The normal vector $(3,2,3\sqrt{-1})$ to $P$ is the tangent vector at $s=1$ of the parametrization $(s^3,s^2, \sqrt{-1}\cdot s^3)$ of the  curve of $\Sing(X)$ through $a$ defined by $x^2-y^3=x-\sqrt{-1}\cdot z=0$.
\end{eg}


\begin{eg} \label{1.6}$\hint$ 
Blow up $X$ and $Y$ in the origin, describe exactly the geometry of the transforms $X_1$ and $Y_1$ and produce visualizations of $X$ and $X_1$ over $\RRR$ in all coordinate charts. For $Y_1$ the equations are in the $x$-chart $1- xy^3=z^3x - w^2=0$ and in the $y$-chart $x^2 -y=z^3y - w^2 =0$.
\end{eg}


\begin{eg} \label{1.7} 
Blow up $X_1$ and $Y_1$ along their singular loci. Compute a resolution of $X_1$ and compare it with the resolution of $Y_1$.
\end{eg}

\goodbreak

\section{Lecture II: Varieties and Schemes}

\noindent The following summary of basic concepts of algebraic geometry shall clarify the terminology used in later sections. Detailed definitions are available in  \cite{Mumford, Shafarevich, Hartshorne, Eisenbud_Harris, Liu, Kemper, Goertz_Wedhorn, EGA2, Zariski_Samuel, Nagata, Matsumura, Atiyah_MacDonald}.\\


\noindent\textit{\Large Varieties}


\begin{defn} Write $\AAA^n=\AAA^n_\K$ for the \textit{affine $n$-space} over some field $\K$. The points of $\AAA^n$ are identified with $n$-tuples $a=(a_1,\ldots, a_n)$ of elements $a_i$ of $\K$. The space $\AAA^n$ is equipped with the \textit{Zariski topology}: the closed sets are the \textit{algebraic subsets} of $\AAA^n$, i.e.,  the zerosets $V(I)=\{a\in\AAA^n,\, f(a)=0$ for all $f\in I\}$ of ideals $I$ of $\K[x_1, \dots, x_n]$.  

To $\AAA^n$ one associates its \textit{coordinate ring}, given as the polynomial ring $\K[\AAA^n]=\K[x_1, \dots, x_n]$ in $n$ variables over $\K$.  \end{defn} 


\begin{defn} A \textit{polynomial map} $f: \AAA^n \too \AAA^m$ between affine spaces is given by a vector $f = (f_1, \dots, f_m)$ of polynomials $f_i=f_i(x) \in \K[x_1, \dots, x_n]$. It induces a $\K$-algebra homomorphism $f^*:\K[y_1,\ldots,y_m]\too\K[x_1,\ldots,x_n]$ sending $y_i$ to $f_i$ and a polynomial $h=h(y)$ to $h\circ f=h(f(x))$. A \textit{rational map} $f: \AAA^n \too \AAA^m$ is given by a vector $f = (f_1, \dots, f_m)$ of elements $f_i \in \K(x_1, \dots, x_n)=\Quot(\K[x_1, \dots, x_n])$ in the quotient field of $\K[x_1, \dots, x_n]$. Albeit the terminology, a rational map need not define a set-theoretic map on whole affine space $\AAA^n$. It does this only on the open subset $U$ which is the complement of the union of the zerosets of the denominators of the elements $f_i$. The induced map $f_{\vert U}: U\to \AAA^m$ is then called a \textit{regular map} on $U$. \end{defn}


\begin{defn} An \textit{affine (algebraic) variety} $X$ is a subset of an affine space $\AAA^n$ which is defined as the zeroset $V(I)$ of a radical ideal $I$ of $\K[x_1,\ldots,x_n]$ and equipped with the topology induced by the Zariski topology of $\AAA^n$. It is thus a closed subset of $\AAA^n$. In this text, the ideal $I$ need not be prime, hence $X$ is not required to be irreducible. The ideal $I_X$ of $\K[x_1,\ldots,x_n]$ of all polynomials  $f$ vanishing on $X$ is the largest ideal such that $X=V(I_X)$. If the field $\K$ is algebraically closed and $X=V(I)$ is defined by the radical ideal $I$, the ideal $I_X$ coincides with $I$ by Hilbert's Nullstellensatz.

The \textit{(affine) coordinate ring of $X$} is the factor ring $\K[X]= \K[x_1,\ldots,x_n]/I_X$. As $I_X$ is radical, $\K[X]$ is reduced, i.e., has no nilpotent elements. If $X$ is irreducible, $I$ is a prime ideal and $\K[X]$ is an integral domain. To each point $a=(a_1,\ldots,a_n)$ of $X$ one associates the maximal ideal $\mm_{X,a}$ of $\K[X]$ generated by the residue classes of the polynomials $x_1-a_1,\ldots, x_n-a_n$. If the field $\K$ is algebraically closed, this defines, by Hilbert's Nullstellensatz, a bijection of the points of $X$ and the maximal ideals of $\K[X]$. 

The \textit{function field} of an irreducible variety $X$ is the quotient field $\K(X)=\Quot(\K[X])$. 

The \textit{local ring} of an affine variety $X$ at a point $a$ is the localization $\calo_{X,a}=\K[X]_{\mm_{X,a}}$ of $\K[X]$ at the maximal ideal $\mm_{X,a}$ of $\K[X]$ associated to $a$ in $X$. It is isomorphic to the factor ring $\calo_{\AAA^n,a}/I_a$ of the local ring $\calo_{\AAA^n,a}$ of $\AAA^n$ at $a$ by the ideal $I_a$ generated by the ideal $I$ defining $X$ in $\calo_{\AAA^n,a}$. If $Z\subset X$ is an irreducible subvariety defined by the prime ideal $p_Z$, the local ring $\calo_{X,Z}$ of $\calo_X$ along $Z$ is defined as the localization $\K[X]_{p_Z}$. The local ring $\calo_{X,a}$ gives rise to the \textit{germ} $(X,a)$ of $X$ at $a$, see below. By abuse of notation, the (unique) maximal ideal $\mm_{X,a}\cdot \calo_{X,a}$ of $\calo_{X,a}$ is also denoted by $\mm_{X,a}$. The factor ring $\kappa_a=\calo_{X,a}/\mm_{X,a}$ is called the \textit{residue field} of $X$ at $a$.\end{defn}


\begin{defn} A \textit{principal open subset} of an affine variety $X$ is the complement in $X$ of the zeroset $V(g)$ of a single non-zero divisor $g$ of $\K[X]$. Principal open subsets are thus dense in $X$. They form a basis of the Zariski-topology.  A \textit{quasi-affine variety} is an open subset of an affine variety. A \textit{(closed) subvariety} of an affine variety $X$ is a subset $Y$ of $X$ which is defined as the zeroset $Y=V(J)$ of an ideal $J$ of $\K[X]$. It is thus a closed subset of $X$.
\comment{[Define Zariski closures of varieties and extensions of ideals defined in principal open sets.]}\end{defn}


\begin{defn} \label{220} Write $\PPP^n=\PPP^n_\K$ for the \textit{projective $n$-space} over some field $\K$. The points of $\PPP^n$ are identified with equivalence classes of $(n+1)$-tuples $a=(a_0,\ldots, a_n)$ of elements $a_i$ of $\K$, where $a\sim b$ if $a=\lambda \cdot b$ for some non-zero $\lambda\in\K$. Points are given by their \textit{projective coordinates} $(a_0:\ldots:a_n)$, with $a_i\in \K$ and not all zero. Projective space $\PPP^n$ is equipped with the Zariski topology whose closed sets are the algebraic subsets of $\PPP^n$, i.e., the zerosets $V(I)=\{a\in\PPP^n, f(a_0,\ldots,a_n)=0$ for all $f\in I\}$ of homogeneous ideals $I$ of $\K[x_0, \dots, x_n]$ different from the ideal generated by $x_0,\ldots,x_n$. As $I$ is generated by homogeneous polynomials, the definition of $V(I)$ does not depend on the choice of the affine representatives $(a_0,\ldots,a_n)\in\AAA^{n+1}$ of the points $a$. Projective space is covered by the affine open subsets $U_i\isom\AAA^n$ formed by the points whose $i$-th 
projective coordinate does not vanish ($i=0,\ldots, n$). 

The \textit{homogeneous coordinate ring} of $\PPP^n$ is the graded polynomial ring $\K[\PPP^n]=\K[x_0, \dots, x_n]$ in $n+1$ variables over $\K$, the grading being given by the degree. \end{defn} 


\begin{defn} A \textit{projective algebraic variety} $X$ is a subset of projective space $\PPP^n$ defined as the zeroset $V(I)$ of a homogeneous radical ideal $I$ of $\K[x_0,\ldots,x_n]$ and equipped with the topology induced by the Zariski topology of $\PPP^n$. It is thus a closed algebraic subset of $\PPP^n$. The ideal $I_X$ of $\K[x_0,\ldots,x_n]$ of all homogeneous polynomials $f$ which vanish at all (affine representatives of) points of $X$ is the largest ideal such that $X=V(I_X)$. If the field $\K$ is algebraically closed and $X=V(I)$ is defined by the radical ideal $I$, the ideal $I_X$ coincides with $I$ by Hilbert's Nullstellensatz.

The \textit{homogeneous coordinate ring of $X$} is the graded factor ring $\K[X]= \K[x_0,\ldots,x_n]/I_X$ equipped with the grading given by degree. \comment{[The \textit{function field} of an irreducible projective variety $X$ is the degree zero part of the quotient field $\K(X)=\Quot(\K[X])$?]}\end{defn}


\begin{defn}  A \textit{principal open subset} of a projective variety $X$ is the complement of the zeroset $V(g)$ of a single homogeneous non-zero divisor $g$ of $\K[X]$.  A \textit{quasi-projective variety} is an open subset of a projective variety. A \textit{(closed) subvariety} of a projective variety $X$ is a subset $Y$ of $X$ which is defined as the zeroset $Y=V(J)$ of a homogeneous ideal $J$ of $\K[X]$. It is thus a closed algebraic subset of $X$.\end{defn}


\begin{rem} Abstract algebraic varieties are obtained by gluing affine algebraic varieties along principal open subsets, cf.÷ \cite{Mumford} I, \S3, \S4, \cite{Shafarevich} V, \S3. This allows to develop the category of algebraic varieties with the usual constructions therein. All subsequent definitions could be formulated for abstract algebraic varieties, but will only be developed in the affine or quasi-affine case to keep things simple.\end{rem}


\begin{defn}  Let $X$ and $Y$ be two affine or quasi-affine algebraic varieties $X\subset \AAA^n$ and $Y\subset\AAA^m$. A \textit{regular map} or \textit{(regular) morphism} from $X$ to $Y$ is a map $f:X\too \AAA^m$ sending $X$ into $Y$ with components rational functions $f_i\in\K(x_1,\ldots,x_n)$ whose denominators do not vanish on $X$. If $X$ and $Y$ are affine varieties, a morphism induces a $\K$-algebra homomorphism $f^*:\K[Y]\too \K[X]$ between the coordinate rings, which, in turn, determines $f$. Over algebraically closed fields, a morphism between affine varieties is the restriction to $X$ of a polynomial map $f:\AAA^n\too \AAA^m$ sending $X$ into $Y$, i.e., such that $f^*(I_Y)\subset I_X$, \cite{Hartshorne}, chap.÷ I, Thm.÷ 3.2, p.÷ 17.   
\end{defn}


\begin{defn}  A \textit{rational map} $f:X\too Y$ between affine varieties is a morphism $f:U\too Y$ on some dense open subset $U$ of $X$. One then says that $f$ is \textit{defined} on $U$. Albeit the terminology, it need not induce a set-theoretic map on whole $X$. For irreducible varieties, a rational map is given by a $\K$-algebra homomorphism $\alpha_f: \K(Y)\too \K(X)$ of the function fields. A \textit{birational map} $f:X\too Y$ is a regular map $f:U\too Y$ on some dense open subset $U$ of $X$ such that $V=f(U)$ is open in $Y$ and such that $f_{\vert U}:U\too V$ is a regular isomorphism, i.e., admits an inverse morphism. In this case $U$ and $V$ are called \textit{biregularly isomorphic}, and $X$ and $Y$ are \textit{birationally isomorphic}. For irreducible varieties, a birational map is given by an isomorphism $\alpha_f: \K(Y)\too \K(X)$ of the function fields. A \textit{birational morphism} $f:X\too Y$ is a birational map which is defined on whole $X$, i.e., a morphism $f:X\too Y$ which admits a 
rational inverse map defined on a dense open subset $V$ of $Y$.
\end{defn}


\begin{defn}  A morphism  $f:X\too Y$ between algebraic varieties is called \textit{separated} if the diagonal $\Delta\subset X\times_YX$ is closed in the fibre product $X\times_YX=\{(a,b)\in X\times X,\, f(a)=f(b)\}$.
A morphism $f:X\too Y$ between varieties is \textit{proper} if it is separated and universally closed, i.e., if for any variety $Z$ and morphism $Z\too Y$ the induced morphism $g:X\times_Y Z\too Z$ is closed. Closed immersions and compositions of proper morphisms are proper. \comment{[For schemes, the morphism $f:X\too Y$ has to be in addition of finite type \cite{EGA} II, 5.4.1]}\end{defn}


\begin{defn} The \textit{germ of a variety $X$ at a point $a$}, denoted by $(X,a)$, is the equivalence class of open neighbourhoods $U$ of $a$ in $X$ where two neighbourhoods of $a$ are said to be equivalent if they coincide on a (possibly smaller) neighbourhood of $a$. To a germ $(X,a)$ one associates the local ring $\calo_{X,a}$ of $X$ at $a$, i.e., the localization $\calo_{X,a}=\K[X]_{\mm_{X,a}}$ of the coordinate ring $\K[X]$ of $X$ at the maximal ideal $\mm_{X,a}$ of $\K[X]$ defining $a$ in $X$. \end{defn}


\comment{[Localization is an exact functor, i.e., $S^{-1}R$ is a flat $R$-module for any multiplicatively closed subset $S$ of $R$ \cite{Atiyah_MacDonald} Cor.~3.6,~p.~40, .]}


\begin{defn} Let $X$ and $Y$ be two varieties, and let $a$ and $b$ be points of $X$ and $Y$. The \textit{germ of a morphism} $f:(X,a)\too (Y,b)$  at $a$ is the equivalence class of a morphism $\wt f:U\too Y$ defined on an open neighbourhood $U$ of $a$ in $X$ and sending $a$ to $b$; here, two morphisms defined on neighbourhoods of $a$ in $X$ are said to be equivalent if they coincide on a (possibly smaller) neighbourhood of $a$. The morphism $\wt f:U\too Y$ is called a \textit{representative} of $f$ on $U$. Equivalently, the germ of a morphism is given by a local $\K$-algebra homomorphism $\alpha_f=f^*:\calo_{Y,b}\to \calo_{X,a}$. \end{defn}


\begin{defn} The \textit{tangent space} ${\rm T}_aX$ to a variety $X$ at a point $a$ is defined as the $\K$-vector space 
\[
{\rm T}_aX= (\mm_{X,a}/\mm_{X,a}^2)^*=\rm{Hom}(\mm_{X,a}/\mm_{X,a}^2,\K)
\]
with $\mm_{X,a}\subset\K[X]$ the maximal ideal of $a$. Equivalently, one may take the maximal ideal of the local ring $\calo_{X,a}$. The tangent space of $X$ at $a$ thus only depends on the germ of $X$ at $a$. The \textit{tangent map} ${\rm T}_af: {\rm T}_aX\too {\rm T}_bY$ of a morphism $f:X\too Y$ sending a point $a\in X$ to $b\in Y$ or of a germ $f:(X,a)\too (Y,b)$ is defined as the linear map induced naturally by $f^*:\calo_{Y,b}\too \calo_{X,a}$.\end{defn}


\begin{defn} The \textit{embedding dimension} $\rm{embdim}_aX$ of a variety $X$ at a point $a$ is defined as the $\K$-dimension of ${\rm T}_aX$.\end{defn}

\bigskip
\noindent\textit{\Large Schemes}


\begin{defn} An \textit{affine scheme $X$} is a commutative ring $R$ with $1$, called the \textit{coordinate ring} or \textit{ring of global sections} of $X$. The set $\Spec(R)$, also denoted by $X$ and called the \textit{spectrum} of $R$, is defined as the set of prime ideals of $R$. Here, $R$ does not count as a prime ideal, but $0$ does if it is prime, i.e., if $R$ is an integral domain. A \textit{point} of $X$ is an element of $\Spec(R)$. In this way, $R$ is the underlying algebraic structure of an affine scheme, whereas $X=\Spec(R)$ is the associated geometric object. To be more precise, one would have to define $X$ as the pair consisting of the coordinate ring $R$ and the spectrum $\Spec(R)$.

The spectrum is equipped with the \textit{Zariski topology}: the closed sets $V(I)$ are formed by the prime ideals containing a given ideal $I$ of $R$. It also comes with a sheaf of rings $\calo_X$, the \textit{structure sheaf of $X$}, whose stalks at points of $X$ are the localizations of $R$ at the respective prime ideals \cite{Mumford, Hartshorne, Shafarevich}. For affine schemes, the sheaf $\calo_X$ is completely determined by the ring $R$. In particular, to define affine schemes it is not mandatory to introduce sheaves or locally ringed spaces. 

An affine scheme is \textit{of finite type over some field $\K$} if its coordinate ring $R$ is a finitely generated $\K$-algebra, i.e., a factor ring $\K[x_1,\ldots,x_n]/I$ of a polynomial ring by some ideal $I$. If $I$ is radical, the scheme is called \textit{reduced}. Over algebraically closed fields, affine varieties can be identified with reduced affine schemes of finite type over $\K$ \cite{Hartshorne}, chap.÷ II, Prop.÷ 2.6, p.÷ 78.\end{defn}


\begin{defn} As a scheme, \textit{affine $n$-space} over some field $\K$ is defined as the scheme $\AAA^n=\AAA^n_\K$ given by the polynomial ring $\K[x_1, \dots, x_n]$ over $\K$ in $n$ variables. Its underlying topological space $\AAA^n=\Spec(\K[x_1, \dots, x_n])$ consists of the prime ideals of $\K[x_1, \dots, x_n]$. The points of $\AAA^n$ when considered as a scheme therefore correspond to the irreducible subvarieties of $\AAA^n$ when considered as an affine variety. A point is \textit{closed} if the respective prime ideal is a maximal ideal. Over algebraically closed fields, the closed points of $\AAA^n$ as a scheme correspond, by Hilbert's Nullstellensatz, to the points of $\AAA^n$ as a variety.\end{defn} 


\begin{defn} Let $X$ be an affine scheme of coordinate ring $R$. A \textit{closed subscheme} $Y$ of $X$ is a factor ring $S=R/I$ of $R$ by some ideal $I$ of $R$, together with the canonical homomorphism $R\too R/I$. The set $Y=\Spec(R/I)$ of prime ideals of $R/I$ can be identified with the closed subset $V(I)$ of $X=\Spec(R)$ of prime ideals of $R$ containing $I$. The homomorphism $R\too R/I$ induces an injective continuous map $Y\too X$ between the underlying topological spaces. In this way, the Zariski topology of $Y$ as a scheme coincides with the topology induced by the Zariski topology of $X$.

A \textit{point} $a$ of $X$ is a prime ideal $I$ of $R$, considered as an element of $\Spec(R)$. To a point one may associate a closed subscheme of $X$ via the spectrum of the factor ring $R/I$. The point $a$ is called \textit{closed} if $I$ is a maximal ideal of $R$. 

A \textit{principal open subset} of $X$ is an affine scheme $U$ defined by the ring of fractions $R_g$ of $R$ with respect to the multiplicatively closed set $\{1,g,g^2,\ldots\}$ for some non-zero divisor $g\in R$. The natural ring homomorphism $R\too R_g$ sending $h$ to $h/1$ interprets $U=\Spec(R_g)$ as an open subset of $X=\Spec(R)$. Principal open subsets form a basis of the Zariski topology of $X$. Arbitary open subsets of $X$ need not admit an interpretation as affine schemes.\end{defn}


\begin{rem} Abstract schemes are obtained by gluing affine schemes along principal open subsets \cite{Mumford} II, \S1, \S2, \cite{Hartshorne} II.2, \cite{Shafarevich} V, \S3. Here, two principal open subsets will be identified or patched together if their respective coordinate rings are isomorphic. This works out properly because the passage from a ring $R$ to its ring of fractions $R_g$ satisfies two key algebraic properties: An element $h$ of $R$ is determined by its images in $R_g$, for all $g$, and, given elements $h_g$ in the rings $R_g$ whose images in $R_{gg'}$ coincide for all $g$ and $g'$, there exists an element $h$ in $R$ with image $h_g$ in $R_g$, for all $g$. This allows in particular to equip arbitrary open subsets of (affine) schemes $X$ with a natural structure of a scheme, which will then be called \textit{open subscheme} of $X$.

The gluing of affine schemes along principal open sets can also be formulated on the sheaf-theoretic level, even though, again, this is not mandatory. For local considerations, one can mostly restrict to the case of affine schemes.   \end{rem}


\begin{defn} A \textit{morphism $f:X\too Y$} between affine schemes $X=\Spec(R)$ and $Y=\Spec(S)$ is a (unitary) ring homomorphism $\alpha=\alpha_f:S\too R$. It induces a continuous map $\Spec(R)\too \Spec(S)$ between the underlying topological spaces by sending a prime ideal $I$ of $R$ to the prime ideal $\alpha\inv(I)$ of $S$. Morphisms between arbitary schemes are defined by choosing coverings by affine schemes and defining the morphism locally subject to the obvious conditions on the overlaps of the patches. A \textit{rational map} $f:X\too Y$ between schemes is a morphism $f:U\too Y$ defined on a dense open subscheme $U$ of $X$. It need not be defined on whole $X$. A rational map is \textit{birational} if the morphism $f:U\too Y$ admits on a dense open subscheme $V$ of $Y$ an inverse morphism $V\too U$, i.e., if $f_{\vert U}:U\too V$ is an isomorphism. A \textit{birational morphism $f:X\too Y$} is a morphism $f:X\too Y$ which is also a birational map, i.e., induces an isomorphism $U\too V$ of dense open 
subschemes. In contrast to birational maps, a birational morphism is defined on whole $X$, while its inverse map is only defined on a dense open subset of $Y$.
\end{defn}


\begin{defn} Let $R=\bigoplus_{i=0}^\infty R_i$ be a graded ring. The set $X=\Proj(R)$ of homogeneous prime ideals of $R$ not containing the \textit{irrelevant ideal} $M=\oplus_{i\geq 1}^\infty R_i$ is equipped with a topology, the \textit{Zariski topology}, and with a sheaf of rings $\calo_X$, the \textit{structure sheaf} of $X$ \cite{Mumford, Hartshorne, Shafarevich}. It thus becomes a scheme. Typically, $R$ is generated as an $R_0$-algebra by the homogeneous elements $g\in R_1$ of degree $1$. An open covering of $X$ is then given by the affine schemes $X_g = \Spec(R^\circ_g)$, where $R^\circ_g$ denotes, for any non-zero divisor $g\in R_1$, the subring of elements of degree $0$ in the ring of fractions $R_g$. \comment {[Does $\Proj(R)=\Spec(R)$ hold for trivially graded rings?]}\end{defn} 


\comment{[Add affine covering of $\Proj(R)$ defined by the rings $R_g^\circ$ of degree zero elements of the ring of fractions $R_g$, with $g$ varying in $R_1$.]}


\begin{rem} A graded ring homomorphism $S\too R$ induces a morphism of schemes $\Proj(R)\too \Proj(S)$.\end{rem}


\begin{defn} Equip the polynomial ring $\K[x_0,\ldots , x_n]$ with the natural grading given by the degree. The scheme $\Proj(\K[x_0,\ldots , x_n])$ is called $n$-dimensio\-nal \textit{projective space} over $\K$, denoted by $\PPP^n=\PPP^n_\K$. It is is covered by the affine open subschemes $U_i\isom\AAA^n$ which are defined through the principal open sets associated to the rings of fractions $\K[x_0,\ldots , x_n]_{x_i}$ ($i=0,\ldots, n$).\end{defn}


\begin{defn} A morphism $f:X\too Y$ is called \textit{projective} if it factors, for some $k$, into a closed embedding $X\hookrightarrow Y\times \PPP^k$ followed by the projection $Y\times\PPP^k\too Y$ onto the first factor \cite{Hartshorne} II.4, p.÷ 103.\end{defn}

\bigskip
\noindent\textit{\Large Formal germs}


\begin{defn} Let $R$ be a ring and $I$ an ideal of $R$. The powers $I^k$ of $I$ induce natural homomorphisms $R/I^{k+1}\too R/I^k$. The \textit{$I$-adic completion} of $R$ is the inverse limit $\wh R=\varprojlim R/I^k$, together with the canonical homomorphism $R\too \wh R$. If $R$ is a local ring with maximal ideal $\mm$, the $\mm$-adic completion $\wh R$ is called the \textit{completion} of $R$. For an $R$-module $M$, one defines the \textit{$I$-adic completion} $\wh M$ of $M$ as the inverse limit $\wh M=\varprojlim M/I^k\cdot M$.\end{defn}


\begin{lem} \label{191} Let $R$ be a noetherian ring with prime ideal $I$ and $I$-adic completion $\wh R$.  

\begin{enumerate}


\item If $J$ is another ideal of $R$ with $(I\cap J)$-adic completion $\wh J$, then $\wh{J}=J\cdot \wh{R}$ and $\wh{R/J}\isom \wh{R}/\wh{J}$.
 

\item If $J_1$, $J_2$ are two ideals of $R$ and $J=J_1\cdot J_2$, then $\wh{J} = \wh{J_1}\cdot \wh{J_2}$.


\item If $R$ is a local ring with maximal ideal $\mm$, the completion $\wh R$ is a noetherian local ring with maximal ideal $\wh \mm=\mm\cdot \wh R$, and $\wh{\mm}\cap R=\mm$.


\item If $J$ is an arbitrary ideal of a local ring $R$, then  $\wh J\cap R=\bigcap_{i\geq0}(J+\mm^i)$. 


\item Passing to the $I$-adic completion defines an exact functor on finitely generated $R$-modules.


\item $\wh R$ is a faithfully flat $R$-algebra.


\item If $M$ a finitely generated $R$-module, then $\wh M=M\otimes_R\wh R$.


\comment{[(8) The completion of a radical ideal is again radical, but the completion of a prime ideal need not be prime, cf.÷ \cite{Zariski_Samuel}?]}


\end{enumerate}
\end{lem}


\begin{proof}{} (1) By \cite{Zariski_Samuel}~VIII,~Thm.~6,~Cor.~2,~p.~258, it suffices to verify that $J$ is closed in the $(I\cap J)$-adic topology. The closure $\overline{J}$ of $J$ equals $\overline{J}=\bigcap_{i\geq0}(J+(I\cap J)^i)=J$ by \cite{Zariski_Samuel}~VIII,~Lemma 1,~p.~253.

(2) By definition, $\wh{J}=J_1\cdot J_2\cdot \wh R=J_1\cdot \wh R\cdot J_2\cdot \wh R=\wh J_1\cdot \wh J_2$.

(3), (4) The ring $\wh R$ is noetherian and local by \cite{Atiyah_MacDonald}~Prop.~10.26,~p.~113,  and Prop.~10.16,~p.~109. The ideal $\wh I\cap R$ equals the closure $\overline{I}$ of $I$ in the $\mm$-adic topology. Thus, $\wh I\cap R=\bigcap_{i\geq0}(I+\mm^i)$ and $\wh \mm\cap R=\mm$. 

(5) -- (7) \cite{Atiyah_MacDonald} Prop.~10.12, p.~108, Prop.~10.14, p.~109, Prop.~10.14, p.~109, and \cite{Matsumura}, Thm.~8.14, p.~62.
\end{proof}


\begin{defn} Let $X$ be a variety and let $a$ be a point of $X$. The local ring $\calo_{X,a}$ of $X$ at $a$ is equipped with the $\mm_{X,a}$-adic topology whose basis of neighbourhoods of $0$ is given by the powers $\mm_{X,a}^k$ of the maximal ideal $\mm_{X,a}$ of $\calo_{X,a}$. The induced completion $\wh\calo_{X,a}$ of $\calo_{X,a}$ is called the \textit{complete local ring of $X$ at $a$} \cite{Nagata} II, \cite{Zariski_Samuel} VIII, \S1, \S2. The scheme defined by $\wh\calo_{X,a}$ is called the \textit{formal neighbourhood} or \textit{formal germ} of $X$ at $a$, denoted by $(\wh X,a)$. If $X=\AAA^n_\K$ is the affine $n$-space over $\K$, the complete local ring $\wh \calo_{\AAA^n,a}$ is isomorphic as a local $\K$-algebra to the formal power series ring $\K[[x_1,\ldots,x_n]]$ in $n$ variables over $\K$. The natural ring homomorphism $\calo_{X,a}\too \wh\calo_{X,a}$ defines a morphism $(\wh X,a)\too (X,a)$ in the category of schemes from the formal neighbourhood to the germ of $X$ at $a$. A \textit{formal 
subvariety} $(\wh Y,a)$ of $(\wh X,a)$, also called a \textit{formal local subvariety} of $X$ at $a$, is (the scheme defined by) a factor ring $\wh\calo_{X,a}/I$ for an ideal $I$ of $\wh\calo_{X,a}$. 

As $\mm_{X,a}/\mm_{X,a}^2\isom \wh \mm_{X,a}/\wh \mm_{X,a}^2$, one defines the tangent space ${\rm T}_a(\wh X,a)$ of the formal germ as the tangent space ${\rm T}_a(X)$ of the variety at $a$.

A map $f:(\wh X,a)\too (\wh Y,b)$ between two formal germs is defined as a local algebra homomorphism $\alpha_f:\wh\calo_{Y,b}\too \wh\calo_{X,a}$. It is 
also called a \textit{formal map} between $X$ and $Y$ at $a$. It induces in a natural way a linear map, the \textit{tangent map}, ${\rm T}_af: {\rm T}_a X\too {\rm T}_bY$ between the tangent spaces.

Two varieties $X$ and $Y$ are \textit{formally isomorphic} at points $a\in X$, respectively $b\in Y$, if the complete local rings $\wh\calo_{X,a}$ and $\wh\calo_{Y,b}$ are isomorphic. If $X\subset \AAA^n$ and $Y\subset \AAA^n$ are subvarieties of the same affine space $\AAA^n$ over $\K$, and $a=b=0$, this is equivalent to saying that there is a local algebra automorphism of $\calo_{\AAA^n,0}\isom \K[[x_1,\ldots,x_n]]$ sending the completed ideals $\wh I$ and $\wh J$ of $X$ and $Y$ onto each other. Such an automorphism is also called a \textit{formal coordinate change} of $\AAA^n$ at $0$. It is given by a vector of $n$ formal power series without constant term whose Jacobian matrix of partial derivatives is invertible when evaluated at $0$. \end{defn}


\begin{rem} The inverse function theorem does not hold for algebraic varieties and regular maps between them, but it holds in the category of formal germs and formal  maps: A formal map $f:(\wh X,a)\too (\wh Y,b)$ is an isomorphism if and only if its tangent map ${\rm T}_af: {\rm T}_aX\too {\rm T}_bY$ is a linear isomorphism.
\end{rem}


\begin{defn} {} A morphism $f:X\too Y$ between varieties is called \textit{\'etale} if for all $a\in X$ the induced maps of formal germs $f:(\wh X,a)\too (\wh Y,f(a))$ are isomorphisms, or, equivalently, if all tangent maps ${\rm T}_af$ of $f$ are isomorphisms, \cite{Hartshorne}, chap.÷ III, ex.÷ 10.3 and 10.4, p.÷ 275. \end{defn}


\begin{defn} A morphism $f:X\too Y$ between varieties is called \textit{smooth} if for all $a\in X$ the induced maps of formal germs $f:(\wh X,a)\too (\wh Y,f(a))$ are submersions, i.e., if the tangent maps ${\rm T}_af$ of $f$ are surjective,  \cite{Hartshorne}, chap.÷ III, Prop.÷ 10.4, p.÷ 270. \end{defn}


\begin{rem} The category of formal germs and formal maps admits the usual concepts and constructions as e.g.÷ the decomposition of a variety in irreducible components, intersections of germs, inverse images, or fibre products. Similarly, when working of $\Bbb R$, $\Bbb C$ or any complete valued field $\K$, one can develop, based on rings of convergent power series, the category of analytic varieties and analytic spaces, respectively of their germs, and analytic maps between them \cite{DeJong_Pfister}.
\end{rem}




\comment{[Add: Krull-dimension, graded rings, Taylor expansion]\\}
\goodbreak


\bigskip
\noindent\textit{\examples}


\begin{eg} \label{2.32} Compare the algebraic varieties $X$ satisfying $(X,a) \isom (\AAA^d,0)$ for all points $a \in X$ (where $\isom$ stands for biregularly isomorphic germs) with those where the isomorphism is just formal, $(\wh X,a) \isom (\wh \AAA^d,0)$. Varieties with the first property are called {\it plain}, cf.÷ def.÷ \ref{515}. \end{eg}


\begin{eg} \label{2.33} It can be shown that any complex regular (def.÷ \ref{147}) and rational surface (i.e., a surface which is birationally isomorphic to the affine plane $\AAA^2$ over $\CCC$) is plain,  \cite{BHSV} Prop.÷ 3.2. Show directly that $X$ defined in $\AAA^3$ by $x - (x^2 + z^2)y=0$ is plain by exhibiting a local isomorphism (i.e., isomorphism of germs) of $X$ at $0$ with $\AAA^2$ at $0$. Is there an algorithm to construct such a local isomorphism for any complex regular and rational surface? \end{eg}


\begin{eg} \label{2.34}$\hint$ Let $X$ be an algebraic variety and $a \in X$ be a point. What is the difference between the concept of regular system of parameters (def.÷ \ref{146}) in $\calo_{X,a}$ and $\widehat \calo_{X,a}$?\end{eg}


\begin{eg} \label{2.35} The formal neighbourhood $(\wh \AAA^n, a)$ of affine space $\AAA^n$ at a point $a=(a_1,\ldots,a_n)$ is given by the formal power series ring $\wh\calo_{\AAA^n,a}\isom\K[[x_1-a_1,\ldots, x_n-a_n]]$. If $X\subset \AAA^n$ is an affine algebraic variety defined by the ideal $I$ of $\K[x_1,\ldots,x_n]$, the formal neighbourhood $(\wh X,a)$ is given by the factor ring $\wh\calo_{X,a}=\wh\calo_{\AAA^n,a}/\wh I\isom\K[[x_1-a_1,\ldots, x_n-a_n]]/\wh I$, where $\wh I= I\cdot \wh\calo_{\AAA^n,a}$ denotes the extension of $I$ to $\wh \calo_{\AAA^n,a}$.\end{eg}


\begin{eg} \label{2.36} The map $f: \AAA^1 \too \AAA^2$ given by $t \mapsto  (t^2,t^3)$ is a regular morphism which induces a birational isomorphism onto the curve $X$ in $\AAA^2$ defined by $x^3=y^2$. The inverse $(x,y)\mapsto y/x$ is a rational map on $X$ and regular on $X\setminus \{0\}$. \end{eg}


\begin{eg} \label{2.37} The map $f: \AAA^1 \too \AAA^2$ given by $t \mapsto  (t^2-1, t(t^2-1))$ is a regular morphism which induces a birational isomorphism onto the curve $X$ in $\AAA^2$ defined by $x^2+x^3=y^2$. The inverse $(x,y)\mapsto y/x$ is a rational map on $X$ and regular on $X\setminus \{0\}$. The germ $(X,0)$ of $X$ at $0$ is not isomorphic to the germ $(Y,0)$ of the union $Y$ of the two diagonals in $\AAA^2$ defined by $x^2=y^2$. The formal germs $(\wh X,0)$ and $(\wh Y,0)$ are isomorphic via the map $(x,y)\mapsto (x\sqrt{1+x},y)$.\end{eg}


\begin{eg} \label{2.38}$\hint$ The map $f: \AAA^2 \too \AAA^2$ given by $(x, y) \mapsto (xy, y)$ is a birational morphism with inverse the rational map $(x, y) \mapsto (\frac{x}{y}, y)$. The inverse  defines a regular morphism outside the $x$-axis.  \end{eg}


\begin{eg} \label{2.39} The maps $\varphi_{ij}: \AAA^{n+1} \too \AAA^{n+1}$ given by
\[
(x_0,\ldots,x_n) \mapsto (\frac{x_0}{x_j},\ldots,\frac{x_{i-1}}{x_j}, 1,\frac{x_{i+1}}{x_j},\ldots,\frac{x_n}{x_j}).
\]
are birational maps for each $i,j=0,\ldots,n$. They are the transition maps between the affine charts $U_j=\PPP^n\setminus V(x_j)\isom \AAA^n$ of projective space $\PPP^n$.\end{eg}


\begin{eg} \label{2.40}$\hint$ The maps $\varphi_{ij}: \AAA^n \too \AAA^n$ given by 
\[
(x_1,\ldots,x_n) \mapsto \Big(\frac{x_1}{x_j},\ldots,\frac{x_{i-1}}{x_j},\frac{1}{x_j},\frac{x_{i+1}}{x_j},\ldots,\frac{x_{j-1}}{x_j},x_ix_j,\frac{x_{j+1}}{x_j},\ldots,\frac{x_n}{x_j}\Big)
\]
are birational maps for $i,j=1,\ldots,n$. They are the transition maps between the affine charts of the blowup $\wt\AAA^n\subset \AAA^n\times\PPP^{n-1}$ of $\AAA^n$ at the origin (cf.÷ Lecture IV) .\end{eg}


\begin{eg} \label{2.41}$\hint$ Assume that the characteristic of the ground field is different from $2$. The elliptic curve $X$ defined in $\AAA^2$ by $y^2=x^3-x$ is formally isomorphic at each point $a$ of $X$ to $(\wh \AAA^1,0)$, whereas the germs $(X,a)$ are not biregularly isomorphic to $(\AAA^1,0)$. \end{eg}


\begin{eg} \label{2.42}$\hint$ The projection $(x,y)\mapsto x$ from the hyperbola $X$ defined in $\AAA^2$ by $xy=1$ to the $x$-axis $\AAA^1=\AAA^1\times\{0\}\subset\AAA^2$ has open image $\AAA^1\setminus\{0\}$. In particular, it is not proper.\end{eg}


\begin{eg} \label{2.43}$\hint$ The curves $X$ and $Y$ defined in $\AAA^2$ by $x^3=y^2$, respectively $x^5=y^2$ are not formally isomorphic to each other at $0$, whereas the curve $Z$ defined in $\AAA^2$ by $x^3+x^5=y^2$ is formally isomorphic to $X$ at $0$.\end{eg}


\section{Lecture III: Singularities}

\noindent Let $X$ be an affine algebraic variety defined over a field $\K$. Analog concepts to the ones given below can be defined for abstract varieties and schemes.


\begin{defn} \label{145} A noetherian local ring $R$ with maximal ideal $\mm$ is called a \textit{regular ring} if $\mm$ can be generated by $n$ elements, where $n$ is the Krull dimension of $R$. The number $n$ is then given as the vector space dimension of $\mm/\mm^2$ over the residue field $R/\mm$. A noetherian local ring $R$ is regular if and only if its completion $\wh R$ is regular, \cite{Atiyah_MacDonald} p.÷ 123.\end{defn}


\begin{defn} \label{146} Let $R$ be a noetherian regular local ring with maximal ideal $\mm$.  A minimal set of generators $x_1, \dots, x_n$ of $\mm$ is called a \textit{regular system of parameters} or \textit{local coordinate system} of $R$. \end{defn}


\begin{rem} A regular system of parameters of a local ring $R$ is also a regular system of parameters of its completion $\wh R$, but not conversely, cf.÷ ex.÷ \ref{3.27}.\end{rem}


\begin{defn} \label{147} A point $a$ of $X$ is a \textit{regular} or \textit{non-singular point} of $X$ if the local ring $\calo_{X, a}$ of $X$ at $a$ is a regular ring. Equivalently, the $m_{X,a}$-adic completion $\wh\calo_{X, a}$ of $\calo_{X, a}$ is isomorphic to the completion $\wh\calo_{\AAA^d, 0}\isom\K[[x_1,\ldots,x_d]]$ of the local ring of some affine space $\AAA^d$ at $0$. Otherwise $a$ is called a \textit{singular point} or a \textit{singularity} of $X$. The set of singular points of $X$ is denoted by $\Sing(X)$, its complement by $\Reg(X)$. The variety $X$ is \textit{regular} or \textit{non-singular} if all its points are regular. Over perfect fields regular varieties are the same as \textit{smooth} varieties, \cite{Cutkosky_Book, Liu}.\end{defn}

           
\begin{prop} A subvariety $X$ of a regular variety $W$ is regular at $a$ if and only if there are local coordinates $x_1,\ldots,x_n$ of $W$ at $a$ so that $X$ can be defined locally at $a$ by $x_1= \dots=x_k=0$ where $k$ is the codimension of $X$ in $W$ at $a$.\end{prop}


\begin{proof} \cite{DeJong_Pfister} Cor.÷ 4.3.20, p.÷ 155.\end{proof}


\begin{rem} Affine and projective space are regular at each of their points. The germ of a variety at a regular point need not be biregularly isomorphic to the germ of an affine space at $0$. However, it is formally isomorphic, i.e., after passage to the formal germ, cf.÷ ex.÷ \ref{3.28}.\end{rem}


\begin{prop} \label{468} Assume that the field $\K$ is perfect. Let $X$ be a hypersurface defined in a regular variety $W$ by the square-free equation $f=0$. The point $a\in X$ is singular if and only if all partial derivatives $\partial_{x_1}f,\ldots,  \partial_{x_n}f$ of vanish at $a$.\end{prop}


\begin{proof} \cite{Zariski_Simple} Thm.÷ 7, \cite{Hartshorne} I.5, \cite{DeJong_Pfister} 4.3.\end{proof}


\begin{defn} The characterization of singularities by the vanishing of the partial derivatives as in Prop.÷ \ref{468} is known as the \textit{Jacobian criterion for smoothness}.\end{defn}


\begin{rem} A similar statement holds for irreducible varieties which are not hypersurfaces, using instead of the partial derivatives the $k\times k$-minors  of the Jacobian matrix of a system of equations of $X$ in $W$, with $k$ the codimension of $X$ in $W$ at $a$ \cite{DeJong_Pfister}, \cite{Cutkosky_Book} pp.÷ 7-8, \cite{Liu} pp.÷ 128 and 142. The choice of the system of defining polynomials of the variety is significant: The ideal generated by them has to be radical, otherwise the concept of singularity has to be developed scheme-theoretically, allowing non-reduced schemes. Over non-perfect fields, the criterion of the proposition does not hold \cite{Zariski_Simple}.\end{rem}


\begin{cor} The singular locus $\Sing(X)$ of $X$ is closed in $X$.\end{cor}


\begin{defn} \label{515} A point $a$ of $X$ is a \textit{plain point} of $X$ if the local ring $\calo_{X, a}$ is isomorphic to the local ring $\calo_{\AAA^d, 0}\isom\K[x_1,\ldots,x_d]_{(x_1,\ldots,x_d)}$ of some affine space $\AAA^d$ at $0$ (with $d$ the dimension of $X$ at $a$). Equivalently, there exists an open neighbourhood $U$ of $a$ in $X$ which is biregularly isomorphic to an open subset $V$ of some affine space $\AAA^d$. The variety $X$ is \textit{plain} if all its points are plain. \end{defn}


\begin{thm} (Bodn\'ar-Hauser-Schicho-Villamayor) The blowup of a plain variety defined over an infinite field along a regular center is again plain.\end{thm}


\begin{proof} \cite{BHSV} Thm.÷ 4.3.\end{proof}


\begin{defn} An irreducible variety $X$ is \textit{rational} if it has a dense open subset $U$ which is biregularly isomorphic to a dense open subset $V$ of some affine space $\AAA^d$. Equivalently, the function field $\K(X)=\Quot(\K[X])$ is isomorphic to the field of rational functions $K(x_1,\ldots,x_d)$ of $\AAA^d$.\end{defn}


\begin{rem} A plain complex variety is smooth and rational. The converse is true for curves and surfaces, and unknown in arbitrary dimension \cite{BHSV}.\end{rem}


\begin{defn} A point $a$ is a \textit{normal crossings point} of $X$ if the formal neighbourhood $(\wh X,a)$ is isomorphic to the formal neighbourhood $(\wh Y,0)$ of a union $Y$ of coordinate subspaces of $\AAA^n$ at $0$. The point $a$ is a \textit{simple normal crossings point} of $X$ if it is a normal crossings point and all components of $X$ passing through $a$ are regular at $a$. The variety $X$ has \textit{normal crossings}, respectively \textit{simple normal crossings}, if the property holds at all of its points. \comment{[Should we require that the components are globally regular?]} \end{defn}


\begin{rem} In the case of schemes, a normal crossings scheme may be non-reduced in which case the components of the scheme $Y$ are equipped with multiplicities. Equivalently, $Y$ is defined locally at $0$ in $\AAA^n$ up to a formal isomorphism by a monomial ideal of $\K[x_1,\ldots,x_n]$. \end{rem}


\begin{prop} A point $a$ is a normal crossings point of a subvariety $X$ of a regular ambient variety $W$ if and only if there exists a regular system of parameters $x_1, \dots, x_n$ of $\calo_{W,a}$ so that the germ $(X, a)$ is defined in $(W,a)$ by a radical monomial ideal in $x_1, \dots, x_n$. This is equivalent to saying that each component of $(X, a)$ is defined by a subset of the coordinates. \end{prop}


\begin{rem} In the case of schemes, the monomial ideal need not be radical.\end{rem}


\comment{[A point $a$ of a variety $X$ is an \textit{algebraic normal crossings point} if the local ring $\calo_{X,a}$ is isomorphic to the local ring $\calo_{N,0}$ of a union $N$ of coordinate subspaces of $\AAA^n$ at $0$. An elliptic curve is regular, hence all its points are simple normal crossings points but none of them is an algebraic normal crossings point. An algebraic normal crossings point of $X$ is a simple normal crossings point of $X$, but not conversely.]}


\begin{defn} Two subvarieties $X$ and $Y$ of a regular ambient variety $W$ meet \textit{transversally} at a point $a$ of $W$ if they are regular at $a$ and if the union $X\cup Y$ has normal crossings at $a$.\end{defn}


\begin{rem} The definition differs from the corresponding notion in differential geometry, where it is required that the tangent spaces of the two varieties at intersection points sum up to the tangent space of the ambient variety at the respective point. In the present text, an inclusion $Y\subset X$ of two regular varieties is considered as being transversal, and also any two coordinate subspaces in $\AAA^n$ meet transversally. In the case of schemes, the union $X\cup Y$ has to be defined by the product of the defining ideals, not their intersection.\end{rem}


\begin{defn} A variety $X$ is a \textit{cartesian product}, if there exist positive dimensional varieties $Y$ and $Z$ such that $X$ is biregularly isomorphic to $Y\times Z$. Analogous definitions hold for germs $(X,a)$ and formal germs $(\wh X,a)$ (the cartesian product of formal germs has to be taken in the category of complete local rings).\end{defn}


\begin{defn} A variety $X$ is \textit{(formally) a cylinder over a subvariety $Z\subset X$} at a point $a$ of $Z$ if the formal neighborhood $(\wh X,a)$ is isomorphic to a cartesian product $(\wh Y,a)\times (\wh Z,a)$ for some (positive-dimensional) subvariety $Y$ of $X$ which is regular at $a$. One also says that $X$ is \textit{trivial} or a \textit{cylinder along $Y$ at $a$ with transversal section $Z$}.\end{defn}


\begin{defn} Let $X$ and $F$ be varieties, and let $0$ be a distinguished point on $F$. The set $S$ of points $a$ of $X$ where the formal neighborhood $(\wh X,a)$ is isomorphic to $(\wh F,0)$ is  called the \textit{triviality locus of $X$ of singularity type $(\wh F,0)$}.\end{defn}


\begin{thm} (Ephraim, Hauser-M\"uller)  For any complex variety $F$ and point $0$ on $F$, the triviality locus $S$ of $X$ of singularity type $(\wh F,0)$ is locally closed and regular in $X$,  and $X$ is a cylinder along $S$. Any subvariety $Z$ of $X$ such that $(\wh X, a)\isom (\wh S, a)\times (\wh Z,a)$ is unique up to formal isomorphism at $a$. \end{thm}


\begin{proof} \cite{Ephraim} Thm.÷ 2.1, \cite{HM_Trivial_Locus} Thm.÷ 1. \end{proof}


\begin{cor} The singular locus $\Sing(X)$ and the non-normal crossings locus of a variety $X$ are closed subvarieties.\end{cor}


\begin{proof} \cite{Mumford} III.4, Prop.÷ 3, p.÷ 170, \cite{Bodnar_Tests}. \end{proof}

\comment{[By the theorem, the respective complements are locally closed and regular. But this does not suffice to show that they are open, so this is not a corollary of the theorem.]}


\begin{thm} (Hauser-M\"uller)  Let $X$, $Y$ and $Z$ be complex varieties with points $a$, $b$ and $c$ on them, respectively. The formal germs of $X\times Z$ at  $(a,c)$ and $Y\times Z$ at $(b,c)$ are isomorphic if and only if $(\wh X,a)$ and $(\wh Y,b)$ are isomorphic.\end{thm}


\begin{proof} \cite{HM_Cancellation}, Thm.÷ 1. \end{proof}
\goodbreak


\bigskip
\noindent\textit{\examples}


\begin{eg} \label{3.27}$\hint$ The element $x\sqrt{1+x}$ is a regular parameter of the completion $\K[[x]]$ of  the local ring $K[x]_{(x)}$ which does not stem from a regular parameter of $K[x]_{(x)}$. \end{eg}


\begin{eg} \label{3.28}$\hint$ The elements $y^2-x^3-x$ and $y$ form a regular parameter system of $\calo_{\AAA^2,0}$ but the zeroset $X$ of $y^2=x^3-x$ is not locally at $0$ isomorphic to $\AAA^1$. However, $(\wh X,0)$ is formally isomorphic to $(\wh \AAA^1,0)$. \end{eg}


\begin{eg} \label{3.29} The set of plain points of a variety $X$ is Zariski open (cf.÷ def.÷ \ref{515}).\end{eg}


\begin{eg} \label{3.30} Let $X$ be the plane cubic curve defined by $x^3 + x^2 -y^2 =0$ in $\AAA^2$. The origin $0 \in X$ is a normal crossings point, but $X$ is not locally at $0$ biregularly isomorphic to the union of the two diagonals of $\AAA^2$ defined by $x=\pm y$.  \end{eg}


\begin{eg} \label{3.31}$\hint$ Let $X$ be the complex surface in $\AAA^3_\CCC$ defined by $x^2 - y^2z = 0$, the \textit{Whitney umbrella} or \textit{pinch point singularity}. The singular locus is the $z$-axis. The origin is not a normal crossings point of $X$. For $a \neq 0$ a point on the $z$-axis, $a$ is a normal crossings point but not a simple normal crossings point of $X$. The formal neighbourhoods of $X$ at points $a\neq 0$ on the $z$-axis are isomorphic to each other, since they are isomorphic to the formal neighbourhood at $0$ of the union of two transversal planes in $\AAA^3$.   \end{eg}


\begin{eg} \label{3.32} Prove that the non-normal crossings locus of a variety is closed. Try to find equations for it \cite{Bodnar_Tests}. \end{eg}


\begin{eg} \label{3.33}$\challenge$ Find a local invariant that measures reasonably the distance of a point $a$ of $X$ from being a normal crossings point. \end{eg}


\begin{eg} \label{3.34}$\hint$ Do the varieties defined by the following equations have normal crossings, respectively simple normal crossings, at the origin? Vary the ground field. 

(a) $x^2+y^2=0$, 

(b) $x^2-y^2=0$, 

(c) $x^2+y^2+z^2=0$,

(d) $x^2+y^2+z^2+w^2=0$,

(e) $xy(x-y)=0$,

(f) $xy(x^2-y)=0$,

(g) $(x-y)z(z-x)=0$.\end{eg} 


\begin{eg} \label{3.35}$\hint$ Visualize the zeroset of $(x-y^2)(x-z)z=0$ in $\AAA^3_\RRR$. \end{eg}


\begin{eg} \label{3.36}$\challenge$ Formulate and then prove the theorem of local analytic triviality in positive characteristic, cf.÷ Thm.÷ 1 of \cite{HM_Trivial_Locus}.\end{eg}


\begin{eg} \label{3.37} Find a coordinate free description of normal crossings singularities. Find an algorithm that tests for normal crossings \cite{Faber_Thesis, Faber_NC}. \end{eg}


\begin{eg} \label{3.38}$\hint$ The singular locus of a cartesian product $X\times Y$ is the union of $\Sing(X)\times Y$ and $X\times \Sing(Y)$.\end{eg}


\begin{eg} \label{3.39}$\hint$ Call a finite union $X=\bigcup X_i$ of regular varieties {\it mikado} if all possible intersections of the components $X_i$ are (scheme-theoretically) regular. The intersections are defined by the sums of the ideals, not taking their radical. Find the simplest example of a variety  which is not mikado but for which all pairwise intersections are non-singular.  \end{eg}


\begin{eg} \label{3.40} Interpret the family given by taking the germs of a variety $X$ at varying points  $a\in X$ as the germs of the fibers of a morphism of varieties, equipped with a section.\end{eg}


\section{Lecture IV: Blowups}

\noindent  Blowups, also known as monoidal transformations, can be introduced in several ways. The respective equivalences will be proven in the second half of this section. All varietes are reduced but not necessarily irreducible, and subvarieties are closed if not mentioned differently. Schemes will be noetherian but not necessarily of finite type over a field. To ease the exposition they are often assumed to be affine, i.e., of the form $X=\Spec(R)$ for some ring $R$. Points of varieties are closed, points of schemes can also be non-closed. The coordinate ring of an affine variety $X$ is denoted by $\K[X]$ and  the structure sheaf of a scheme $X$ by $\calo_X$, with local rings $\calo_{X,a}$ at points $a\in X$.  

References providing additional material on blowups are, among many others, \cite{Hironaka_Annals}, chap.÷ III, and \cite{Eisenbud_Harris}, chap.÷ IV.2.


\begin{defn} A subvariety $Z$ of a variety $X$ is called a \textit{hypersurface in $X$} if the codimension of $Z$ in $X$ at any point $a$ of $Z$ is $1$,
\[
\dim_aZ = \dim_aX -1.
\]
\end{defn}


\begin{rem} In the case where $X$ is non-singular and irreducible, a hypersurface is locally defined at any point $a$ by a single non-trivial equation, i.e., an equation given by a non-zero and non-invertible element $h$ of $\calo_{X, a}$. \comment{[which is also a non-zero divisor since $\calo_{X, a}$ is an integral domain.]} This need not be the case for singular varieties, see ex.÷ \ref{4.21}. Hypersurfaces are a particular case of \textit{effective Weil divisors} \cite{Hartshorne} chap.÷ II, Rmk.÷ 6.17.1, p.÷ 145.\end{rem}


\begin{defn} A subvariety $Z$ of an irreducible variety $X$ is called a \textit{Cartier divisor} in $X$ at a point $a\in Z$ if $Z$ can be defined locally at $a$ by a single equation $h=0$ for some non-zero element $h\in \mathcal{O}_{X, a}$. If $X$ is not assumed to be irreducible, $h$ is required to be a non-zero divisor of $\mathcal{O}_{X, a}$. This excludes the possibility that $Z$ is a component or a union of components of $X$. The subvariety $Z$ is called a \textit{Cartier divisor} in $X$ if it is a Cartier divisor at each of its points. The empty subvariety is considered as a Cartier divisor. A (non-empty) Cartier divisor is a hypersurface in $X$, but not conversely, and its complement $X\setminus Z$ is dense in $X$. \comment{[A hypersurface $Z$ is Cartier in $X$ if and only if $Z$ is rare in $X$, i.e., if $X\setminus Z$ is dense in $X$? The direction ''Cartier implies $X\setminus Z$ dense in $X$'' is shown in \cite{Eisenbud_Harris} Lemma IV-19, p.÷ 167, and should be easy, since $Z$ cannot be a 
component of $X$. The inverse implication should be false, since the line in the cone is rare.]} Cartier divisors are, in a certain sense, the largest closed and properly contained subvarieties of $X$. If $Z$ is Cartier in $X$, the ideal $I$ defining $Z$ in $X$ is called \textit{locally principal} \cite{Hartshorne} chap.÷ II, Prop.÷ 6.13, p.÷ 144, \cite{Eisenbud_Harris} III.2.5, p.÷ 117. \end{defn}


\begin{defn}\label{blowup1} (Blowup via universal property) Let $Z$ be a (closed) subvariety of a variety $X$. A variety $\wt X$ together with a morphism $\pi: \wt {X} \too X$ is called a \textit{blowup of $X$ with center $Z$ of $X$}, or a \textit{blowup of $X$ along $Z$}, if the inverse image $E =\pi\inv(Z)$ of $Z$ is a Cartier divisor in $\wt X$ and $\pi$ is universal with respect to this property: For any morphism $\tau: X' \too X$ such that $\tau^{-1}(Z)$ is a Cartier divisor in $X'$, there exists a unique morphism $\sigma: X' \too \wt {X}$ so that $\tau$ factors through $\sigma$, say $\tau = \pi\circ \sigma$, 
%
%
\[
\xymatrix@R=2pc@C=3pc{
X'  \ar@{-->}[r]^{ \exists_1\, \sigma}  \ar[rd]^\tau & \wt X \ar[d]^{\pi} \\
 & X }
 \]
The morphism $\pi$ is also called the {\it blowup map}. The subvariety $E$ of $\wt X$ is a Cartier divisor, in particular a hypersurface, and called the \textit{exceptional divisor} or \textit{exceptional locus} of the blowup. One says that $\pi$ \textit{contracts $E$ to $Z$}.
\end{defn}

\begin {rem} By the universal property, a blowup of $X$ along $Z$, if it exists, is unique up to unique isomorphism. It is therefore called \textit{the} blowup of $X$ along $Z$. If $Z$ is already a Cartier divisor in $X$, then $\wt X=X$ and $\pi$ is the identity by the universal property. In particular, this is the case when $X$ is non-singular and $Z$ is a hypersurface in $X$. \end{rem}


\begin{defn} The \textit{Rees algebra} of an ideal $I$ of a commutative ring $R$ is the graded $R$-algebra
\[
\Rees(I) = \bigoplus_{i = 0}^{\infty}I^i = \bigoplus_{i =
0}^{\infty}I^i\cdot t^i\subset R[t],
\]
where $I^i$ denotes the $i$-fold power of $I$, with $I^0$ set equal to $R$. The variable $t$ is given degree $1$ so that the elements of $I^i\cdot t^i$ have degree $i$. Write $\wt R$ for $\Rees(I)$ when $I$ is clear from the context. The Rees algebra of $I$ is generated by elements of degree $1$, and $R$ embeds naturally into $\wt R$ by sending an element $g$ of $R$ to the degree $0$ element $g\cdot t^0$. If $I$ is finitely generated by elements $g_1,\ldots,g_k\in R$, then $\wt R=R[g_1t,\ldots,g_kt]$.  
The Rees algebras of the zero-ideal $I=0$ and of the whole ring $I=R$ equal $R$, respectively $R[t]$. If $I$ is a principal ideal generated by a non-zero divisor of $R$, the schemes $\Proj(\wt R)$ and $\Spec(R)$ are isomorphic. The Rees algebras of an ideal $I$ and its $k$-th power $I^k$ are isomorphic as graded $R$-algebras, for any $k\geq 1$.\end{defn}


\begin{defn} (Blowup via Rees algebra) Let $X=\Spec(R)$ be an affine scheme and let $Z=\Spec(R/I)$ be a closed subscheme of $X$ defined by an ideal $I$ of $R$. Denote by $\wt R=\Rees(I)$ the Rees algebra of $I$ over $R$, equipped with the induced grading. The \textit{blowup of $X$ along $Z$} is the scheme $\wt {X} =\Proj(\wt R)$ together with the morphism $\pi: \wt X \too X$ given by the natural ring homomorphism $R\too\wt R$. The subscheme $E=\pi\inv(Z)$ of $\wt X$ is called the \textit{exceptional divisor} of the blowup.\end{defn}


\begin{defn} (Blowup via secants)  Let $W = \AAA^n$ be affine space over $\K$ (taken as a variety) and let $p$ be a fixed point of $\AAA^n$. Equip $\AAA^n$ with a vector space structure by identifying it with its tangent space ${\mathrm T}_p\AAA^n$. \comment{[Otherwise it is not clear what a line in $\AAA^n$ is.]} For a point $a \in \AAA^n$ different from $p$, denote by $g(a)$  the secant line in $\AAA^n$ through $p$ and $a$, considered as an element of projective space $\PPP^{n - 1}=\PPP({\mathrm T}_p\AAA^n)$. The morphism
\[
\gamma: \Bbb{A}^n \setminus \{p\}\too\PPP^{n - 1}, a \mapsto g(a),
\]
is well defined. The Zariski closure $\wt {X}$ of the graph $\Gamma$ of $\gamma$ inside $\Bbb{A}^n \times \PPP^{n - 1}$ together with the restriction $\pi: \wt {X} \too X$ of the projection map $\Bbb{A}^n \times \PPP^{n-1} \too \AAA^n$ is the \textit{point blowup of $\AAA^n$ with center $p$}.\end{defn}

\comment{[The secant definition of point blowups can be done for $X$ any smooth variety, see Griffiths-Harris.]} 


\begin{defn}\label{blowup4} (Blowup via closure of graph) Let $X$ be an affine variety with coordinate ring $\K[X]$ and let $Z = \textrm{V}(I)$ be a subvariety of $X$ defined by an ideal $I$ of $\K[X]$ generated by elements $g_1, \dots, g_k$. The morphism
\[
\gamma: X\setminus Z \too \PPP^{k-1}, a \mapsto (g_1(a): \dots: g_k(a)),
\]
is well defined. The Zariski closure $\wt {X}$ of the graph $\Gamma$ of $\gamma$
inside $X \times \PPP^{k-1}$ together with the restriction $\pi: \wt {X} \too X$
of the projection map $X \times \PPP^{k-1}\too X$ is the \textit{blowup of $X$ along  
$Z$}. It does not depend, up to isomorphism over $X$, on the choice of the generators $g_i$ of $I$.\end{defn}


\begin{defn}\label{blowup5} (Blowup via equations) Let $X$ be an affine variety with coordinate ring $\K[X]$ and let $Z = \textrm{V}(I)$ be a subvariety of $X$ defined by an ideal $I$ of $\K[X]$ generated by elements $g_1, \dots, g_k$. Assume that $g_1, \dots, g_k$ form a regular sequence in $\K[X]$.\comment{[It seems sufficient to require that the relations between $g_1, \dots, g_k$ are linear modulo the trivial relations.]} Let $(u_1:\ldots :u_k)$ be projective coordinates on $\PPP^{k-1}$. The sub\-variety $\wt {X}$ of $X \times \PPP^{k-1}$ defined by the equations 
\[
u_i\cdot g_j - u_j\cdot g_i =0, \hskip 1cm i,j = 1, \dots, k, 
\]
together with the restriction $\pi: \wt {X} \too X$ of the projection map $X \times \PPP^{k-1}\too X$ is the \textit{blowup of $X$ along $Z$}. It does not depend, up to isomorphism over $X$, on the choice of the generators $g_i$ of $I$.\end{defn}


\begin{rem} This is a special case of the preceding definition of blowup as the closure of a graph. If $g_1, \dots, g_k$ do not form a regular sequence, the subvariety $\wt {X}$ of $X \times \PPP^{k-1}$ may require more equations, see ex.÷ \ref{4.43}. \end{rem}


\begin{defn}\label{blowup6} (Blowup via affine charts) Let $W = \Bbb{A}^n$ be affine space over $\K$ with chosen coordinates $x_1,\ldots, x_n$. Let $Z\subset W$ be a coordinate subspace defined by equations $x_j=0$ for $j$ in some subset $J\subset \{1, \dots n\}$. Set $U_j=\AAA^n$ for $j\in J$, and glue two affine charts $U_j$ and $U_\ell$ via the transition maps 
%
\[
\hskip 3cm x_i\mapsto x_i/x_j,\hskip .4cm \textrm{ if } i \in J\setminus \{j,\ell\},
\]
\[
x_j\mapsto 1/x_\ell,
\]
\[
x_\ell\mapsto x_jx_\ell,
\]
\[
\hskip 1.85cm x_i\mapsto x_i,\hskip 1cm \textrm{ if } i\not \in J.
\]
%
This yields a variety $\wt W$. Define a morphism $\pi:\wt W\too W$ by the chart expressions $\pi_j:\AAA^n \too\AAA^n$ for $j\in J$ as follows,
\[
x_i \mapsto x_i,
\hskip 1.3cm \textrm{ if } i \not\in J\setminus \{j\}, 
\]
\[
x_i \mapsto x_ix_j, \hskip 1cm \textrm{ if } i \in J\setminus \{j\}.
\]
The variety $\wt W$ together with the morphism $\pi:\wt W\too W$  is the \textit{blowup of $W=\AAA^n$ along $Z$}. The map $\pi_j:\AAA^n \too \AAA^n$ is called the \textit{$j$-th affine chart} of the blowup map, or the \textit{$x_j$-chart}. The maps $\pi_j$ depend on the choice of coordinates in $\AAA^n$, whereas the blowup map $\pi:\wt W\too W$ only depends on $W$ and $Z$.\end{defn}


\begin{defn}\label{blowup7} (Blowup via ring extensions)  Let $I$ be a non-zero ideal in a noetherian integral domain $R$, generated by non-zero elements $g_1, \dots, g_k$ of $R$. The \textit{blowup of $R$ in $I$} is given by the ring extensions 
\[
R \hookrightarrow R_j = R\left[\frac{g_1}{g_j}, \dots,
\frac{g_k}{g_j}\right], \hskip 1cm  j = 1, \dots, k,
\]
inside the rings $R_{g_j}=R[\frac{1}{g_j}]$, where the rings $R_j$ are glued pairwise via the natural inclusions $R_{g_j}, R_{g_\ell}\subset R_{g_jg_{\ell}}$, for $j,\ell=1,\ldots, k$. The blowup does not depend, up to isomorphism over $X$, on the choice of the generators $g_i$ of $I$.\end{defn}


\begin{rem} The seven concepts of blowup given in the preceding definitions are all essentially equivalent. It will be convenient to prove the equivalence in the language of schemes, though the substance of the proofs is inherent to varieties. \end{rem}



\begin{defn} (Local blowup via germs)  Let $\pi: X'\too X$ be the blowup of $X$ along $Z$, and let $a'$ be a point of $X'$ mapping to the point $a\in X$. The \textit{local blowup of $X$ along $Z$ at $a'$} is the morphism of germs $\pi: (X',a')\too (X,a)$. It is given by the dual homomorphism of local rings $\pi^*:\calo_{X,a}\too \calo_{X',a'}$. This terminology is also used for the completions of the local rings giving rise to a morphism of formal neighborhoods $\pi: (\wh X',a')\too (\wh X,a)$ with dual homomorphism $\pi^*:\wh \calo_{X,a}\too \wh\calo_{X',a'}$. 
\end{defn}


\begin{defn} (Local blowup via localizations of rings)  Let $R=(R,\mm)$ be a local ring and let $I$ be an ideal of $R$ with Rees algebra $\wt R=\bigoplus_{i=0}^\infty I^i$. Any non-zero element $g$ of $I$ defines a homogeneous element of degree $1$ in $\wt R$, also denoted by $g$. The ring of quotients $\wt R_g$ inherits from $\wt R$ the structure of a graded ring. The set of degree $0$ elements of $\wt R_g$ forms a ring, denoted by $R[Ig\inv]$, and consists of fractions $f/g^\ell$ with $f\in I^\ell$ and $\ell\in\NNN$.
The localization $R[Ig^{-1}]_p$ of $R[Ig^{-1}]$ at any prime ideal $p$ of $R[Ig^{-1}]$ containing the maximal ideal $\mm$ of $R$ together with the natural ring homomorphism $\alpha: R\too R[Ig^{-1}]_p$ is called the \textit{local blowup} of $R$ with center $I$ associated to $g$ and $p$. \end{defn}


\begin{rem} The same definition can be made for non-local rings $R$, giving rise to a local blowup $R_q\too R_q[Ig^{-1}]_p$ with respect to the localization $R_q$ of $R$ at the prime ideal $q=p\cap R$ of $R$.\end{rem}


\begin{thm} \comment{[$\Proj(\wt R)$ versus universal property]} Let $X=\Spec(R)$ be an affine scheme and $Z$ a closed subscheme defined by an ideal $I$ of $R$. The blowup $\pi:\wt X\too X$ of $X$ along $Z$ when defined as $\Proj(\wt R)$ of the Rees algebra $\wt R$ of $I$ satisfies the universal property of blowups.\end{thm} 


\begin{proof} \comment{[Cf.÷ \cite{Eisenbud_Harris} IV, Prop.÷ 18]} (a) It suffices to prove the universal property on the affine charts of an open covering of $\Proj(\wt R)$, since on overlaps the local patches will agree by their uniqueness. \comment{[This requires a special case of the base change property of blowups, namely compatability with localization, which, however, is immediate from the construction.]} This in turn reduces the proof to the local situation.

(b) Let $\beta:R\to S$ be a homomorphism of local rings such that $\beta(I)\cdot S$ is a principal ideal of $S$ generated by a non-zero divisor, for some ideal $I$ of $R$. It suffices to show that there is a unique homomorphism of local rings $\gamma:R'\to S$ such that $R'$ equals a localization $R'=R[Ig^{-1}]_p$ for some non-zero element $g$ of $I$ and a prime ideal $p$ of $R[Ig^{-1}]$ containing the maximal ideal of $R$, and such that the diagram\\
\centerline{
\xymatrix{
R \ar[d]^\beta \ar[r]^\alpha & R' \ar[ld]^\gamma\\
S}
}
commutes, where $\alpha:R\too R'$ denotes the local blowup of $R$ with center $I$ specified by the choice of $g$ and $p$. The proof of the local statement goes in two steps.

(c) There exists an element $f\in I$ such that $\beta(f)$ generates $\beta(I)\cdot S$: Let $h\in S$ be a non-zero divisor generating $\beta(I)\cdot S$. Write $h=\sum_{i=1}^n \beta(f_i)s_i$ with elements $f_1,\ldots,f_n\in I$ and $s_1,\ldots,s_n\in S$. Write $\beta(f_i)=t_ih$ with elements $t_i\in S$. Thus, $h=\big(\sum_{i=1}^n s_it_i\big)h$. Since $h$ is a non-zero divisor, the sum $\sum_{i=1}^n s_it_i$ equals $1$. In particular,  there is an index $i$ for which $t_i$ does not belong to the maximal ideal of $S$. This implies that this $t_i$ is invertible in $S$, so that $h=t_i^{-1}\beta(f_i)$. In particular, for this $i$, the element $\beta(f_i)$ generates $\beta(I)\cdot S$.

(d) Let $f\in I$ be as in (c). By assumption, the image $\beta(f)$ is a non-zero divisor in $S$.  For every $\ell\geq0$ and every $h_\ell\in I^\ell$, there is an element $a_\ell\in S$ such that $\beta(h_\ell)=a_\ell\beta(f)^\ell$. Since $\beta(f)$ is a non-zero divisor, $a_\ell$ is unique with this property. For an arbitrary element $\sum_{\ell=0}^n h_\ell/g^\ell$ of $R[Ig^{-1}]$ with $h_\ell\in I^\ell$ set $\delta(\sum_{\ell=0}^n h_\ell/g^\ell)=\sum_{\ell=0}^n a_\ell$. This defines a ring homomorphism $\delta:R[Ig^{-1}]\too S$ that restricts to $\beta$ on $R$. By definition, $\delta$ is unique with this property. Let $p\subset R[Ig\inv]$ be the inverse image under $\delta$ of the maximal ideal of $S$. This is a prime ideal of $R[Ig\inv]$ which contains the maximal ideal of $R$. By the universal property of localization, $\delta$ induces a homomorphism of local rings $\gamma: R[Ig^{-1}]_p\to S$ that restricts to $\beta$ on $R$, i.e., satisfies $\gamma\circ\alpha=\beta$. By construction, $\gamma$ is unique. \end{proof}


\begin{thm} \comment{[$\Proj(\wt R)$ versus affine charts and ring extensions]} Let $X=\Spec(R)$ be an affine scheme and let $Z$ be a closed subscheme defined by an ideal $I$ of $R$. The blowup of $X$ along $Z$ defined by $\Proj(\wt R)$ can be covered by affine charts as described in def.÷ \ref{blowup5}. \end{thm}


\begin{proof} For $g\in I$ denote by $R[Ig\inv]\subset R_g$ the subring of the ring of quotients $R_g$ generated by homogeneous elements of degree $0$ of the form $h/g^\ell$ with $h\in I^\ell$ and $\ell\in\NNN$. This gives an injective ring homomorphism $R[Ig\inv]\too \wt R_g$. Let now $g_1,\ldots , g_k$ be generators of $I$. Then $X=\Proj(\wt R)$ is covered by the principal open sets $\Spec(R[I g_i\inv])=\Spec(R[g_1/g_i,\ldots ,g_k/g_i])$. The chart expression $\Spec(R[g_1/g_i,\ldots ,g_k/g_i])\too\Spec(R)$ of the blowup map $\pi:\wt X\too X$ follows now by computation. \comment{[Computation to be inserted later]}\end{proof}


\begin{thm} \comment{[$\Proj(\wt R)$ vs closure of graph and equations]} Let $X$ be an affine variety with coordinate ring $R=\K[X]$ and let $Z$ be a closed subvariety defined by the ideal $I$ of $R$ with generators $g_1,\ldots, g_k$. The blowup of $X$ along $Z$ defined by $\Proj(\wt R)$ of the Rees algebra $\wt R$ of $R$ equals the closure of the graph $\Gamma$ of $\gamma: X\setminus Z \too\PPP^{k-1}, a \mapsto (g_1(a): \dots: g_k(a))$.  If $g_1,\ldots,g_k$ form a regular sequence, the closure is defined as a subvariety of $X\times\PPP^{k-1}$ by equations as indicated in def.÷ \ref{blowup5}.\end{thm}


\begin{proof} (a) Let $U_j\subset \PPP^{k-1}$ be the affine chart given by $u_j\neq 0$, with isomorphism $U_j\isom\AAA^{k-1}, (u_1:\ldots:u_k)\mapsto (u_1/u_j,\ldots ,u_k/u_j)$. The $j$-th chart expression of $\gamma$ equals $a\mapsto (g_1(a)/g_j(a),\ldots,g_k(a)/g_j(a))$ and is defined on the principal open set $g_j\neq 0$ of $X$. The closure of the graph of $\gamma$ in $U_j$ is therefore given by $\Spec(R[g_1/g_i,\ldots ,g_k/g_i])$. The preceding theorem then establishes the required equality.

(b) If $g_1,\ldots,g_k$ form a regular sequence, their only linear relations over $R$ are the trivial ones, so that
\[
R[Ig_i\inv]=R[g_1/g_i,\ldots g_k/g_i]\isom R[t_1,\ldots , t_k]/ (g_it_j-g_j, j=1,\ldots,k).
\]
The $i$-th chart expression of $\pi:\wt X\too X$ is then given by the ring inclusion $R\too R[t_1,\ldots , t_k]/ (g_it_j-g_j, j=1,\ldots,k)$.  \end{proof}


\noindent\textit{\examples}


\comment{[Varieties $X$ for which all hypersurfaces can be defined locally by a single equation are those for which all local rings are unique factorization domains.]}


\begin{eg} \label{4.21} For the cone $X$ in $\AAA^3$ of equation $x^2 + y^2 = z^2$, the line $Y$ in $\AAA^3$ defined by $x = y - z = 0$ is a hypersurface of $X$ at each point. It is a Cartier divisor at any point $a \in Y \setminus \{0\}$ but it is not a Cartier divisor at $0$. \comment{[What about characteristic 2? Prove that $Y$ is not Cartier at 0.]} The double line $Y'$ in $\AAA^3$ defined by $x^2 = y - z = 0$ is a Cartier divisor of $X$ since it can be defined in $X$ by $y - z = 0$.  The subvariety $Y''$ in $\AAA^3$ consisting of the two lines defined by $x=y^2-z^2=0$ is a Cartier divisor of $X$ since it can be defined in $X$ by $x=0$.\end{eg}


\begin{eg} \label{4.22} For the surface $X: x^2y-z^2 =0$ in $\AAA^3$, the subvariety $Y$ defined by $y^2-xz=x^3-yz=0$ is the singular curve parametrized by $t \mapsto (t^3 , t^4 , t^5)$. It is everywhere a Cartier divisor except at $0$: there, the local ring of $Y$ in $X$ is $\K[x,y,z]_{(x,y,z)}/(x^2y-z^2,y^2-xz,x^3-yz)$, which defines a singular cubic. It is of codimension $1$ in $X$ but not a complete intersection (i.e., cannot be defined in $\AAA^3$, as the codimension would suggest, by two but only by three equations). Thus it is not a Cartier divisor. At any other point $a=(t^3,t^4,t^5)$, $t \neq 0$ of $Y$, one has $\calo_{Y,a}=S/(x^2y-z^2,y^2/x-z,x^3-yz)=S/(x^2y-z^2,y^2/x-z)$ with $S=\K[x,y,z]_{(x-t^3,y-t^4,z-t^5)}$ the localization of $\K[x,y,z]$ at $a$. Hence $Y$ is a Cartier divisor there. \end{eg}


\begin{eg} \label{4.23} Let $Z$ be one of the axes of the cross $X:xy=0$ in $\AAA^2$, e.g. the $x$-axis. At any point $a$ on the $y$-axis except the origin, $Z$ is Cartier: one has $\calo_{X,a}= \K[x,y]_{(x,y-a)}/(xy) =\K[x,y]_{(x,y-a)}/(x)$ and the element $h = y$ defining $Z$ is a unit in this local ring. At the origin $0$ of $\AAA^2$ the
local ring of $X$ is $\K[x,y]_{(x,y)}/(xy)$ and $Z$ is locally defined by $h = y$ which is a zero-divisor in $\calo_{X,0}$. Thus $Z$ is not Cartier in $X$ at $0$.\end{eg}


\begin{eg} \label{4.24} Let $Z$ be the origin of $X=\AAA^2$. Then $Z$ is not a hypersurface and hence not Cartier in $X$. \end{eg}


\begin{eg} \label{4.25}$\hint$ The (reduced) origin $Z=\{0\}$ in $X = V (xy, x^2)\subset\AAA^2$ is not a Cartier divisor in $X$.
\end{eg}


\begin{eg} \label{4.26}$\hint$ Are $Z =V(x^2)$ in $\AAA^1$ and $Z =V(x^2y)$ in $X =\AAA^2$ Cartier divisors? \end{eg}


\begin{eg} \label{4.27}$\hint$ Let $Z=V(x^2,y)$ be the origin of $\AAA^2$ with non-reduced  structure given by the ideal $(x^2,y)$. Then $Z$ is not a Cartier divisor in $X=V(xy, x^2)\subset \AAA^2$. \end{eg}



\begin{eg} \label{4.28} Let $X = \AAA^1$ be the affine line with coordinate ring $R=\K[x]$ and let $Z$ be the origin of $\AAA^1$. The Rees algebra $\wt R = \K[x, xt] \subset \K[x, t]$ with respect to the ideal defining $Z$ is isomorphic, as a graded ring, to a polynomial ring $\K[u, v]$ in two variables with $\deg u = 0$ and $\deg v = 1$. Therefore, the point blowup $\wt X$ of $X$ is isomorphic to $\AAA^1 \times \PPP^0 = \AAA^1$, and $\pi : \wt X \too X$ is the identity.  

More generally, let $X = \AAA^n$ be $n$-dimensional affine space with coordinate ring $R=\K[x_1,\ldots,x_n]$ and let $Z\subset X$ be a hypersurface defined by some non-zero $g \in \K[x_1,\ldots,x_n]$. The Rees algebra $\wt R = \K[x_1,\ldots,x_n,gt] \subset \K[x_1,\ldots,x_n,t]$ with respect the ideal defining $Z$ is isomorphic, as a graded ring, to a polynomial ring $\K[u_1,\ldots,u_n,v]$ in $n+1$ variables with $\deg u_i = 0$ and $\deg v = 1$. Therefore the blowup $\wt X$ of $X$ along $Z$ is isomorphic to $\AAA^n \times \PPP^0 = \AAA^n$, and $\pi : \wt X \too X$ is the identity. \end{eg}


\begin{eg} \label{4.29} Let $X = \Spec(R)$ be an affine scheme and $Z\subset X$ be defined by some non-zero-divisor $g \subset R$. The Rees algebra $\wt R = R[gt] \subset R[t]$ with respect to the ideal defining $Z$ is isomorphic, as a graded ring, to $R[v]$ with $\deg r = 0$ for $r \in R$ and $\deg v = 1$. Therefore the blowup $\wt X$ of $X$ along $Z$ is isomorphic to $X\times \PPP^0 =X$, and $\pi:\wt X \too X$ is the identity.
\end{eg}


\begin{eg} \label{4.30} Take $X=V(xy) \subset \AAA^2$ with coordinate ring $R=K[x,y]/(xy )$. Let $Z$ be defined in $X$ by $y=0$. The Rees algebra of $R$ with respect to the ideal defining $Z$ equals $\wt R = R[yt] \cong K[x, y, u]/(xy, xu )$ with $\deg x = \deg y = 0$ and $\deg u = 1$. \end{eg}


\begin{eg} \label{4.31} Let $X = \Spec(R)$ be an affine scheme and let $Z\subset X$ be defined by some zero-divisor $g \neq 0$ in $R$, say $h \cdot g = 0$ for some non-zero $h\in R$. The Rees algebra $\wt R=R[gt] \subset R[t]$ with respect to the ideal defining $Z$ is isomorphic, as a graded ring, to $R[v]/(h\cdot v)$ with $\deg r=0$ for $r\in R$ and $\deg v=1$. Therefore, the blowup $\wt X$ of $X$ along $Z$ equals the closed subvariety of $X\times \PPP^0 =X$ defined by $h\cdot v=0$, and $\pi:\wt X \too X$ is the inclusion map. \comment{[?]} \end{eg}


\begin{eg} \label{4.32} Take in the situation of the preceding example $X=V(xy)\subset \AAA^2$ and $g=y$, $h=x$. Then $R=\K[x,y]/(xy)$ and $\wt R=R[yt]\isom \K[x,y,u]/(xy, xu)$ with $\deg\, x=\deg\, y=0$ and $\deg\, u =1$.\end{eg}


\comment{[? Take $X=\Spec (\K[x]/(x^5 ))$ and let $Z$ be defined in $X$ by $g=x$. The Rees algebra of $R=\K[x]/(x^5 )$ with respect to the ideal defining $Z$ equals $\wt R=\K[x,u]/(x^5, u^5 )$.]}


\begin{eg} \label{4.33} Let $X = \Spec(R)$ be an affine scheme and let $Z\subset X$ be defined by some nilpotent element $g \neq 0$ in $R$, say $g^k = 0$ for some $k\geq 1$. The Rees algebra $\wt R=R[gt] \subset R[t]$ with respect to the ideal defining $Z$ is isomorphic, as a graded ring, to $R[v]/(v^k )$ with $\deg r=0$ for $r \in R$ and $\deg v = 1$. The blowup $\wt X$ of $X$ is the closed subvariety of $X \times \PPP^0=X$ defined by $v^k=0$. \end{eg}


\begin{eg} \label{4.34} Let $X = \AAA^2$ and let $Z$ be defined in $X$ by $I = (x,y)$. The Rees algebra $\wt R = \K[x,y,xt,yt] \subset \K[x,y,t]$ of $R=\K[x,y]$ with respect to the ideal defining $Z$ is isomorphic, as a graded ring, to the factor ring $\K[x,y,u,v]/(xv - yu)$, with $\deg x = \deg y = 0$ and $\deg v = \deg u = 1$. It follows that the blowup $\wt X$ of $X$ along $Z$ embeds naturally as the closed and regular subvariety of $\AAA^2 \times \PPP^1$ defined by $xv - yu = 0$, the morphism $\pi : \wt X\too X$ being given by the restriction to $\wt X$ of the first projection $\AAA^2 \times \PPP^1 \too \AAA^2$. \end{eg}


\begin{eg} \label{4.35} Let $X = \AAA^2$ and let $Z$ be defined in $X$ by $I = (x, y^2)$. The Rees algebra $\wt R = \K[x, y, xt, y^2t] \subset \K[x, y, t]$ of $R=\K[x,y]$ with respect to the ideal defining $Z$ is isomorphic, as a graded ring, to the factor ring $\K[x,y,u,v]/(xv - y^2u )$, with $\deg x = \deg y = 0$ and $\deg u = \deg v = 1$. It follows that the blowup $\wt X$ of $X$ along $Z$ embeds naturally as the closed and singular subvariety of $\AAA^2 \times \PPP^1$ defined by $xv - y^2u = 0$, the morphism $\pi: \wt X \too X$ being given by the restriction to $\wt X$ of the first projection $\AAA^2 \times \PPP^1\too \AAA^2$. \end{eg}


\begin{eg} \label{4.36} Let $X = \AAA^3$ and let $Z$ be defined in $X$ by $I = (xy,z)$. The Rees algebra $\wt R = \K[x,y,z,xyt,zt] \subset \K[x,y,z,t]$ of $R=\K[x,y,z]$ with respect to the ideal defining $Z$ is isomorphic, as a graded ring, to $\K[u,v,w,r,s]/(uvs - wr)$, with $\deg u = \deg v = \deg w = 0$ and $\deg r = \deg s = 1$. It follows that the blowup $\wt X$ of $X$ along $Z$ embeds naturally as the closed and singular subvariety of $\AAA^3 \times \PPP^1$ defined by $uvs - wr = 0$, the morphism $\pi : \wt X \too X$ being given by the restriction to $\wt X$ of the first projection $\AAA^3 \times \PPP^1 \too \AAA^3$.
\end{eg}

\begin{eg} \label{4.37} Let $X = \AAA^1_\ZZZ$ be the affine line over the integers and let $Z$ be defined in $X$ by $I = (x,p )$ for a prime $p \in \ZZZ$. The Rees algebra $\wt R = \ZZZ[x,xt,pt] \subset \ZZZ[x,t]$ of $R= \ZZZ[x]$ with respect to the ideal defining $Z$ is isomorphic, as a graded ring, to $\ZZZ[u, v, w]/ (xw - pv)$, with $\deg u = 0$ and $\deg v = \deg w = 1$. It follows that the blowup $\wt X$ of $X$ along $Z$ embeds naturally as the closed and regular subvariety of $\AAA^1_\ZZZ \times \PPP^1_\ZZZ$ defined by $xw - pv$, the morphism $\pi: \wt X \too X$ being given by the restriction to $\wt X$ of the first projection $\AAA^1_\ZZZ \times \PPP^1_\ZZZ\too \AAA^1_\ZZZ$.
\end{eg}


\begin{eg} \label{4.38}$\hint$ Let $X = \AAA^2_{\ZZZ_{20}}$ be the affine plane over the ring $\ZZZ_{20} = \ZZZ/20\ZZZ$, and let $Z$ be defined in $X$ by $I = (x,2y)$. The Rees algebra $\wt R = \ZZZ_{20}[x,y,xt,2yt] \subset \ZZZ_{20}[x,y,t]$ of $R= \ZZZ_{20}[x]$ with respect to $I$ is isomorphic, as a graded ring, to $\ZZZ_{20}[u, v, w, z]/(xz - 2vw)$, with $\deg u = \deg v = 0$ and $\deg w = \deg z = 1$. \end{eg}

\begin{eg} \label{4.39} Let $R = \ZZZ_6$ with $\ZZZ_6 = \ZZZ/6\ZZZ$, and let $Z$ be defined in $X$ by $I = (2)$. The Rees algebra $\wt R = \ZZZ_6[2t] \subset \ZZZ_6[t]$ of $R $ with respect to $I$ is isomorphic, as a graded ring, to $\ZZZ_6\oplus u\cdot \ZZZ_3[u]$ with $\deg u = 1$.\end{eg}


\begin{eg} \label{4.40} Let $R = \K[[x, y]]$ be a formal power series ring in two variables $x$ and $y$, and let  $Z$ be defined in $X$ by $I$ be the ideal generated by $e^x-1$ and $\ln(y+1)$. The Rees algebra $\wt R$ of $R$ with respect to $I$ equals $R[(e^x -1)t, \ln(y+1)t] = \K[[x, y]][(e^x -1)t, \ln(y+1)t] \cong \K[[x, y]][u, v]/(\ln(y+1)u - (e^x -1)v)$ with $\deg  u= \deg v=1$.\end{eg}


\begin{eg} \label{4.41} Let $X=\AAA^n$. The point blowup $\wt X \subset X \times \PPP^{n-1}$ of $X$ at the origin is defined by the ideal $(u_i x_j - u_j x_i, i,j=1, \ldots, n)$ in $\K[x_1, \ldots, x_n, u_1, \ldots, u_n]$. This is a graded ring, where $\deg x_i=0$ and $\deg u_i=1$ for all $i=1, \ldots, n$. The ring
\[
\K[x_1, \ldots, x_n, u_1, \ldots, u_n]/ (u_i x_j - u_j x_i, i,j=1, \ldots, n)
\]
is isomorphic, as a graded ring, to $\wt R=\K[x_1, \ldots, x_n, x_1t, \ldots, x_nt] \subset \K[x_1, \ldots, x_n, t]$ with $\deg x_i=0$, $\deg t=1$. The ring $\wt R$ is the Rees-algebra of the ideal $(x_1, \ldots, x_n)$ of $\K[x_1, \ldots, x_n]$. Thus $\wt X$ is isomorphic to $\Proj\big(\bigoplus_{d \geq 0} (x_1, \ldots, x_n)^d\big)$. \end{eg}


\begin{eg} \label{4.42} Let $1\leq k\leq n$. The $i$-th affine chart of the blowup $\wt \AAA^n$ of $\AAA^n$ in the ideal $I=(x_1,\ldots,x_k)$ is isomorphic to $\AAA^n$, for $i=1,\ldots,k$, via the ring isomorphism
\[
\K[x_1,\ldots,x_n]\isom 
\K[x_1,\ldots,x_n,t_1,\ldots,\wh t_i,\ldots t_k]/ (x_it_j-x_j, j=1,\ldots,k, j\neq i),
\]
where $x_j\mapsto t_j$ for $j=1,\ldots,k$, $j\neq i$, respectively $x_j\mapsto x_j$ for $j=k+1,\ldots, n$ and $j=i$. The inverse map is given by $x_j\mapsto x_i x_j$ for $j=1,\ldots,k$, $j\neq i$, respectively $x_j\mapsto x_j$ for $j=k+1,\ldots,n$ and $j=i$, respectively $t_j\mapsto x_j$ for $j=1,\ldots,k$, $j\neq i$.\end{eg} 

\begin{eg} \label{4.43} Let $X=\AAA^2$ be the affine plane and let $g_1=x^2, g_2=xy, g_3=y^3$ generate the ideal $I \subset \K[x,y]$. The $g_i$ do not form a regular sequence. The subvariety of $\AAA^2 \times \PPP^2$ defined by the equations $g_i u_j - g_j u_i=0$ is singular, but the blowup of $\AAA^2$ in $I$ is regular.\end{eg}


\ignore

\begin{eg} \label{}  The Zariski-closure of the map
\[
\gamma: U=\AAA^n \setminus \{0\} \too \PPP^{n-1}: (a_1, \ldots, a_n) \mapsto (a_1: \ldots : a_n)
\]
is the point blowup of $X$ at the origin. In terms of rings the above map reads as
\[
\calo_U^n \too \calo_U: (f_1, \ldots, f_n) \mapsto \sum f_i x_i,
\]
that is, we have a map of the open subset $U$ of $X$ into $\Proj^{n-1}_U$. Let $Y$ be the variety obtained by gluing the affine charts $R_i=R[\frac{x_1}{x_i}, \ldots, \frac{x_n}{x_i},x_i]$. Then $Y$ is the closure of the graph of $\alpha_{(x_1, \ldots, x_n)}$: the open subset $(\Spec R_i)_{x_i} \subset Y$ is the graph of the map
\[
\alpha_{(x_1, \ldots, x_n)}|_{(\Spec R_i)_{x_i}}: (\Spec R_i)_{x_i} \too (\Proj_R^{n-1})_{x_i}=\Spec R[\frac{x_1}{x_i}, \ldots, \frac{x_n}{x_i}].
\]
Since the $R_i$ are dense in $Y$, the assertion follows.
\end{eg}

\recognize


\begin{eg} \label{4.44} Compute the following blowups:

(a) $\AAA^2$ in the center $(x,y)(x,y^2)$,

(b) $\AAA^3$ in the centers $(x,yz)$ and $(x,yz)(x,y)(x,z)$,

(c) $\AAA^3$ in the center $(x^2+y^2-1,z)$. 

(d) The plane curve $x^2=y$ in the origin.

Use affine charts and ring extensions to determine at which points the resulting varieties are regular or singular. \end{eg}


\begin{eg} \label{4.45} Blow up $\AAA^3$ in $0$ and compute the inverse image of $x^2+y^2=z^2$.  
\end{eg}

\begin{eg} \label{4.46}$\hint$ Blow up $\AAA^2$ in the point $(0,1)$. What is the inverse image of  the lines $x+y=0$ and $x+y=1$? \end{eg}


\begin{eg} \label{4.47}$\hint$ Compute the two chart transition maps for the blowup of $\AAA^3$ along the $z$-axis.  \end{eg}


\begin{eg} \label{4.48}$\hint$ Blow up the cone $X$ defined by $x^2+y^2=z^2$ in one of its lines.  \end{eg}


\begin{eg} \label{4.49}$\hint$ Blow up $\AAA^3$ in the circle $x^2+(y+2)^2-1=z=0$ and in the elliptic curve $y^2-x^3-x=z=0$. \end{eg}


\begin{eg} \label{4.50} Interpret the blowup of $\AAA^2$ in the ideal $(x,y^2)(x,y)$ as a composition of blowups in regular centers. \end{eg}


\begin{eg} \label{4.51} Show that the blowup of $\AAA^n$ along a coordinate subspace $Z$ equals the cartesian product of the point-blowup in a transversal subspace $V$ of $\AAA^n$ of complementary dimension (with respect to $Z$) with the identity map on $Z$. 
\end{eg}


\begin{eg} \label{4.52}$\hint$ Show that the ideals $(x_1,\ldots,x_n)$ and $(x_1,\ldots,x_n)^m$ define the same blowup of $\AAA^n$ when taken as center. \comment{[{\it Hint}: Think before you start to compute, or, better, don't compute.]} \end{eg}


\begin{eg} \label{4.53} Let $E$ be a normal crossings subvariety of $\AAA^n$ and let $Z$ be a subvariety of $\AAA^n$ such that $E\cup Z$ also has normal crossings (the union being defined by the product of ideals). Show that the inverse image of $E$ under the blowup of $\AAA^n$ along $Z$ has again normal crossings. Show by an example that the assumption on $Z$ cannot be dropped in general. \end{eg}


\begin{eg} \label{4.54} Draw a real picture of the blowup of $\AAA^2$ at the origin.\end{eg}


\begin{eg} \label{4.55} Show that the blowup $\widetilde \AAA^n$ of $\AAA^n$ with center a point is non-singular.\end{eg}


\begin{eg} \label{4.56} Describe the geometric construction via secants of the blowup $\wt\AAA^3$ of $\AAA^3$ with center $0$. What is $\pi^{-1}(0)$? For a chosen affine chart $W'$ of $\widetilde\AAA^3$, consider all cylinders $Y$ over a circle in $W'$ centered at the origin and parallel to a coordinate axis. What is the image of $Y$ under $\pi$ in $\AAA^3$?\end{eg}


\begin{eg} \label{4.57}$\hint$ Compute the blowup of the Whitney umbrella $X=V(x^2-y^2z) \subset \AAA^3$ with center one of the three coordinate axes, respectively the origin $0$. \end{eg}


\begin{eg} \label{4.58} Determine the locus of points of the Whitney umbrella $X=V(x^2-y^2z)$ where the singularities are normal crossings, respectively simple normal crossings. Blow up the complements of these loci and compare with the preceding example \cite{Kollar_Book} ex.÷ 3.6.1, p.÷ 123, \cite{BDSMP} Thm.÷ 3.4.\end{eg}


\begin{eg} \label{4.59}  Let $Z$ be a regular center in $\AAA^n$, with induced blowup $\pi:\wt\AAA^n\too  \AAA^n$, and let $a'$ be a point of $\wt\AAA^n$ mapping to a point $a\in Z$. Show that it is possible to choose local formal coordinates at $a$, i.e., a regular parameter system of $\wh \calo_{\AAA^n,a}$, so that the center is a coordinate subspace, and so that $a'$ is the origin of one of the affine charts of $\wt\AAA^n$. Is this also possible with a regular parameter system of the local ring $\calo_{\AAA^n,a}$?\end{eg}


\begin{eg} \label{4.60}  Let $X=V(f)$ be a hypersurface in $\AAA^n$, defined by $f\in\K[x_1,\ldots,x_n]$,  and let $Z$ be a regular closed subvariety which is contained in the locus $S$ of points of $X$ where $f$ attains its maximal order. Let $\pi:\widetilde{\AAA}^n\to\AAA^n$ be the blowup along $Z$ and let $X^s=V(f^s)$ be the strict transform of $X$, defined as the Zariski closure of $\pi\inv(X\setminus Z)$ in $\widetilde{\AAA}^n$. Show that for points $a\in Z$ and $a'\in E=\pi^{-1}(Z)$ with $\pi(a')=a$ the inequality $\ord_{a'}f'\leq\ord_a f$ holds. Do the same for subvarieties of $\AAA^n$ defined by arbitrary ideals. \end{eg}


\begin{eg} \label{4.61} Find an example of a variety $X$ for which the dimension of the singular locus increases under the blowup of a closed regular center $Z$ that is contained in the top locus of $X$  \cite{Ha_Obstacles} ex.÷ 9.\end{eg}


\begin{eg} \label{4.62} Let $\pi:(\wt\AAA^n,a')\too (\AAA^n,a)$ be the local blowup of $\AAA^n$ with center the point $a$, considered at a point $a'\in E$. Let $x_1,\ldots ,x_n$ be given local coordinates at $a$. Determine the coordinate changes in $(\AAA^n,a)$ which make the chart expressions of $\pi$ monomial (i.e., each component of a chart expression is a monomial in the coordinates).\end{eg}


\begin{eg} \label{4.63} Let $x_1,\ldots ,x_n$ be local coordinates at $a$ so that the local blowup $\pi:(\wt\AAA^n,a')\too (\AAA^n,a)$ is monomial with respect to them. Determine the formal automorphisms of $(\wt\AAA^n,a')$ which commute with the local blowup.\end{eg}


\begin{eg} \label{4.64}$\hint$ Compute the blowup of $X=\Spec(\ZZZ[x])$ in the ideals $(x,p)$ and $(px,pq)$ where $p$ and $q$ are primes.\end{eg}

\begin{eg} \label{4.65} {} Let $R= \K[x, y, z]$ be the polynomial
ring in three variables and let $I$ be the ideal $(x, yz )$ of $R$. The blowup of $R$ along $I$  corresponds to the ring extensions:
\[
R  \hookrightarrow R\left[\frac{yz}{x}\right] \cong \K[s, t, u, v]/(sv - tu), 
\]
\[
R \hookrightarrow R\left[\frac{x}{yz}\right] \cong \K[s, t, u, v]/(s - tuv).
\]
The first ring extension defines a singular variety while the second one defines a non-singular one. Let $W = \AAA^3$ and $Z = V(I) \subset W$ be the affine space and the subvariety defined by $I$. The blowup of $R$ along $I$ coincides with the blowup $W'$ of $W$ along $Z$.

(a) Compute the chart expressions of the blowup maps.

(b) Determine the exceptional divisor. 

(c) Apply one more blowup to $W'$ to get a non-singular variety $W''$. 

(d) Express the composition of the two blowups as a single blowup in a properly chosen ideal.

(e) Blow up $\AAA^3$ along the three coordinate axes. Show that  the resulting variety is non-singular. 

(f) Show the same for the blowup of $\AAA^n$ along the $n$ coordinate axes.\end{eg}


\begin{eg} \label{4.66} The zeroset in $\AAA^3$ of the non-reduced ideal $(x,yz)(x,y)(x,z) = (x^3, x^2y,x^2z, xyz,y^2z)$ is the union of the $y$- and the $z$-axis. Taken as center, the resulting blowup of $\AAA^3$ equals the composition of two blowups: The first blowup has center the ideal $(x,yz)$ in $\AAA^3$, giving a three-fold $W_1$ in a regular four-dimensional ambient variety with one singular point of local equation $xy=zw$. The second blowup is the point blowup of $W_1$ with center this singular point \cite{Ha_Excellent} Prop.÷ 3.5, \cite{Faber_Westra, Levine}.\end{eg}


\begin{eg} \label{4.67}$\challenge$ Consider the blowup $\wt \AAA^n$ of $\AAA^n$ in a monomial ideal $I$ of $\K[x_,\ldots,x_n]$. Show that $\wt \AAA^n$ may be singular. What types of singularities will occur? Find a natural saturation procedure $I\rightsquigarrow \overline I$ so that the blowup of $\AAA^n$ in $\overline I$ is regular and equal to a (natural) resolution of the singularities of $\wt \AAA^n$ \cite{Faber_Westra}. \end{eg}


\begin{eg} \label{4.68} The \textit{Nash modification} of a subvariety $X$ of $\AAA^n$ is the closure of the graph of the map which associates to each non-singular point its tangent space, taken as an element of the Grassmanian of $d$-dimensional linear subspaces of $\K^n$, where $d=\dim(X)$. For a hypersurface $X$ defined in $\AAA^n$ by $f=0$, the Nash modification coincides with the blowup of $X$ in the Jacobian ideal of $f$ generated by the partial derivatives of $f$.\end{eg}

\ignore 
\recognize

\section{Lecture V: Properties of Blowup}


\begin{prop} \label{basechange} Let $\pi: \wt {X} \too X$ be the blowup of $X$ along a subvariety $Z$, and let $\varphi: Y\too X$ be a morphism, the base change. Denote by $p:\wt {X} \times_X Y\too Y$ the projection from the fibre product to the second factor. Let $S = \varphi^{-1}(Z)\subset Y$ be the inverse image of $Z$ under $\varphi$, and let $\wt {Y}$ be the Zariski closure of $p^{-1}(Y\setminus S)$ in $\wt {X} \times_X Y$. The restriction $\tau: \wt Y\too Y$ of $p$ to $\wt Y$ equals the blowup of $Y$ along $S$. 


\centerline{
\xymatrix{
F~ \ar@{^{(}->}[r] \ar[rd] & \wt{Y}~ \ar@{^{(}->}[r] \ar[rd]^\tau & \wt{X}\times_X Y \ar[r]^q \ar[d]^p & \wt{X} \ar[d]^\pi \\
& S~ \ar@{^{(}->}[r] & Y \ar[r]^\varphi & X}
}
\end{prop}


\begin{proof} \comment{[cf.÷ \cite{Eisenbud_Harris} Prop.÷ IV-21, p.÷ 168]} The assertion is best proven via the universal property of blowups. To show that $F=\tau^{-1}(S)$ is a Cartier divisor in $\wt Y$, consider the projection $q:\wt X\times_X Y\to \wt X$ onto the first factor. By the commutativity of the diagram, 
\[
F=p^{-1}(S)=p^{-1}\circ\varphi^{-1}(Z)=q^{-1}\circ\pi^{-1}(Z)=q^{-1}(E).
\]
As $E$ is a Cartier divisor in $\wt X$ and $q$ is a projection, $F$ is locally defined by a principal ideal. The associated primes of $\wt Y$ are the associated primes of $Y$ not containing the ideal of $S$. Thus, the local defining equation of $E$ in $\wt X$ cannot pull back to a zero divisor on $\wt Y$. This proves that $F$ is a Cartier divisor. 

To show that $\tau:\wt Y\to Y$ fulfills the universal property, let $\psi:Y'\to Y$ be a morphism such that $\psi^{-1}(S)$ is a Cartier divisor in $Y'$. This results in the following diagram.


\centerline{
\xymatrix{
&& \wt X \ar[rd]^\pi & \\
Y' \ar@/^1pc/[rru]^\rho \ar[r]^\sigma \ar@/_1pc/[rrd]^\psi & \wt X\times_X Y \ar[ru]^q \ar[rd]^p && X \\
&& Y \ar[ru]^\varphi}
}


Since $\psi^{-1}(S)=\psi\inv(\varphi\inv(Z))=(\varphi\circ\psi)^{-1}(Z)$, there exists by the universal property of the blowup $\pi: \wt X\too X$ a unique map $\rho:Y'\to \wt X$ such that $\varphi\circ\psi=\pi\circ\rho$. By the universal property of fibre products, there exists a unique map $\sigma:Y'\to\wt X\times_X Y$ such that $q\circ \sigma=\rho$ and $p\circ\sigma=\psi$.

It remains to show that $\sigma(Y')$ lies in $\wt Y\subset \wt X\times_X Y$. Since $\psi^{-1}(S)$ is a Cartier divisor in $Y'$, its complement $Y'\setminus \psi^{-1}(S)$ is dense in $Y'$. From 
\[
Y'\setminus \psi^{-1}(S)=\psi^{-1}(Y\setminus S)=(p\circ\sigma)\inv (Y\setminus S)=\sigma\inv(p\inv(Y\setminus S))
\]
follows that $\sigma(Y'\setminus \psi^{-1}(S))\subset p\inv(Y\setminus S)$.  But $\wt Y$ is the closure of $p\inv(Y\setminus S)$ in $\wt X\times_X Y$, so that $\sigma(Y')\subset \wt Y$ as required.\end{proof}


\begin{cor}
(a) Let $\pi: X' \too X$ be the blowup of $X$ along a subvariety $Z$, and let $Y$ be a closed subvariety of $X$. Denote by $Y'$ the Zariski closure of $\pi\inv(Y\setminus Z)$ in $X'$, i.e., the \textit {strict transform} of $Y$ under $\pi$, cf.÷ def.÷ \ref{stricttransform}. The restriction $\tau:Y'\too Y$ of $\pi$ to $Y'$ is the blowup of $Y$ along $Y\cap Z$. In particular, if $Z\subset Y$, then $\tau$ is the blowup $\wt Y$ of $Y$ along $Z$.


(b) Let $U \subset X$ be an open subvariety, and let $Z \subset X$ be a closed
subvariety, so that $U \cap Z$ is closed in $U$. Let $\pi: X' \too X$ be the blowup of $X$ along $Z$. The blowup of $U$ along $U \cap Z$ equals the restriction of $\pi$ to $U'=\pi\inv(U)$.


(c) Let $a\in X$ be a point. Write $(X, a)$ for the germ of $X$ at $a$, and $(\wh X, a)$ for the formal neighbourhood. There are natural maps
\[
(X, a) \too X \hskip .5cm \textrm{and} \hskip .5cm (\wh X, a) \too X 
\]
corresponding to the localization and completion homomorphisms $\calo_X \too \calo_{X, a}\too \wh\calo_{X, a}$. Take a point $a'$ above $a$ in the blowup $X'$ of $X$ along a subvariety $Z$ containing $a$. This gives local blowups of germs and formal neighborhoods 
\[
\pi_{a'}: (X', a') \too(X, a), \hskip 1cm 
\wh \pi_{a'}: (\wh X', a') \too (\wh X, a).
\]
The blowup of a local ring is not local in general; to get a local blowup one needs to localize also on $X'$.


(d) If $X_1\too X$ is an isomorphism between varieties sending a subvariety $Z_1$ to $Z$, the blowup $X_1'$ of $X_1$ along $Z_1$ is canonically isomorphic to the blowup $X'$ of $X$ along $Z$. This also holds for local isomorphisms. 


(e) If $X= Z\times Y$ is a cartesian product of two varieties, and $a$ is a given point of $Y$, the blowup $\pi:X'\too X$ of $X$ along $Z\times \{a\}$ is isomorphic to the cartesian product $\mathrm{Id}_Z\times \tau:Z\times Y'\too Z\times Y$ of the identity on $Z$ with the blowup $\tau: Y'\too Y$ of $Y$ in $a$. 
\end{cor}


\begin{prop} \label{199} Let $\pi:X'\too X$ be the blowup of $X$ along a regular subvariety $Z$ with exceptional divisor $E$. Let $Y$ be a subvariety of $X$, and $Y^*=\pi\inv(Y)$ its preimage under $\pi$. Let $Y'$ be the Zariski closure of $\pi\inv(Y\setminus Z)$ in $X'$. If $Y$ is transversal to $Z$, i.e., $Y\cup Z$ has normal crossings at all points of the intersection $Y\cap Z$, also $Y^*$ has normal crossings at all points of $Y^*\cap E$. In particular, if $Y$ is regular and transversal to $Z$, also $Y'$ is regular and transversal to $E$.\end{prop}

\begin{proof} Having normal crossings is defined locally at each point through the completions of local rings. The assertion is proven by a computation in local coordinates for which the blowup is monomial, cf.÷ Prop.÷ \ref{111} below.\end{proof}


\begin{prop} \label{111} Let $W$ be a regular variety of dimension $n$ with a regular subvariety $Z$ of codimension $k$. Let $\pi:W'\too W$ denote the blowup of $W$ along $Z$, with exceptional divisor $E$. Let $V$ be a regular hypersurface in $W$ containing $Z$, let $D$ be a (not necessarily reduced) normal crossings divisor in $W$ having normal crossings with $V$. Let $a$ be a point of $V\cap Z$ and let $a'\in E$ be a point lying above $a$. There exist local coordinates $x_1,\ldots, x_n$ of $W$ at $a$ such that

(1) $a$ has components $a=(0,\ldots,0)$.

 (2) $V$ is defined in $W$ by $x_{n-k+1}=0$.

(3) $Z$ is defined in $W$ by $x_{n-k+1}=\ldots =x_n=0$.

(4) $D\cap V$ is defined in $V$ locally at $a$ by a monomial $x_1^{q_{1}}\cdots x_n^{q_n}$, for some $q=(q_1,\ldots,q_n)\in \NNN^n$ with $q_{n-k+1}=0$.

(5) The point $a'$ lies in the $x_n$-chart of $W'$. The chart expression of $\pi$ in the $x_n$-chart is of the form 
\[
\hskip .9cmx_i\mapsto x_i \hskip 2.4cm \mathrm{for} \hskip .1cm i\leq n-k \hskip .1cm\mathrm{and} \hskip .1cm i=n,
\]
\[
\hskip 1.2cm x_i\mapsto x_ix_n \hskip 2cm\mathrm{for} \hskip .1cm  n-k+1\leq i\leq n-1.
\]

(6) In the induced coordinates of the $x_n$-chart, the point $a'$ has components $a'=(0,\ldots,0,a'_{n-k+2},\ldots, a'_{n-k+d},0,\ldots,0)$ with non-zero entries $a_j'\in\K$ for $n-k+2\leq j\leq n-k+d$, where $d$ is the number of components of $D$ whose strict transforms do not pass through $a'$. 

(7) The strict transform (def.÷ \ref{stricttransform}) $V^s$ of $V$ in $W'$ is given in the induced coordinates locally at $a'$ by $x_{n-k+1}=0$.

(8) The local coordinate change $\varphi$ in $W$ at $a$ given by $\varphi(x_i)= x_i+a_i'\cdot x_n$ makes the local blowup $\pi:(W',a')\too (W,a)$ monomial. It preserves the defining ideals of $Z$ and $V$ in $W$.

(9) If condition (4) is not imposed, the coordinates $x_1,\ldots,x_n$ at $a$ can be chosen with (1) to (3) and so that $a'$ is the origin of the $x_n$-chart.

\end{prop}


\begin{proof} \cite{Ha_Wild}. \comment{[To be inserted later]}\end{proof}


\begin{thm} Any projective birational morphism $\pi: X'\too X$ is a blowup of $X$ in an ideal $I$.\end{thm}

\begin{proof} \cite{Hartshorne}, chap.÷ II,  Thm.÷ 7.17.\end{proof}


\comment{[\noindent ADDENDA:\\
Eleonore, chap.÷ 4\\
Normalization
Blowup is a projective space over X and exceptional divisor very ample \cite{Goertz_Wedhorn}
Prop.÷ 13.96, p.÷ 411.\\
If $X$ is integral and $Z$ is non-empty, then $X'$ is integral and $\pi:X'\too X$ is birational, proper, surjective \cite{Goertz_Wedhorn} Cor.÷ 13.97.\\
Flatness, see Liu, Prop.÷ 1.12, p.÷ 322\\
Characterize projective morphisms which are blowups by Grauert's theorem. Include projectivized normal bundle as defining the exeptional locus.\\
Composition of blowups is a blowup \cite{ Bodnar_Centers}. \\
Weak factorization theorem [Morelli, W\l odarczyk, Abramovich, Karu, Matsuki].\\
Weighted blowup, toric modifications, Nash modifications [Spivakovsky, Gonzales-Sprinberg], F-Blowups [Yasuda], alterations, Dulac, Atiyah.]}


\bigskip
\noindent\textit{\examples}
 

\begin{eg} \label{5.6} Let $X$ be a regular subvariety of $\AAA^n$ and $Z$ a regular closed subvariety which is transversal to $X$. Show that the blowup $X'$ of $X$ along $Z$ is again a regular variety (this is a special case of Prop.÷ \ref{199}).\end{eg}


\begin{eg} \label{5.7}$\challenge$ Prove that plain varieties remain plain under blowup in regular centers \cite{BHSV} Thm.÷  4.3.\end{eg}


\begin{eg} \label{5.8} $\challenge$ Is any rational and regular variety plain?\end{eg}


\begin{eg} \label{5.9} Consider the blowup $\pi:W'\to W$ of a regular variety $W$ along a closed subvariety $Z$. Show that, for any chosen definition of blowup, the exceptional divisor $E=\pi^{-1}(Z)$ is a hypersurface in $W'$.\end{eg}


\begin{eg} \label{5.10} The composition of two blowups $W''\too W'$ and $W'\too W$ is a blowup of $W$ in a suitable center \cite{Bodnar_Centers}. \end{eg}


\begin{eg} \label{5.11} A \textit{fractional ideal} $I$ over an integral domain $R$ is an $R$-sub\-module of $\Quot(R)$ such that $rI \subset R$ for some non-zero element $r\in R$.  Blowups can be defined via $\Proj$ also for centers which are fractional ideals \cite{EGA2}. Let $I$ and $J$ be two (ordinary) non-zero ideals of $R$. The blowup of $R$ along $I$ is isomorphic to the blowup of $R$ along $J$ if and only if there exist positive integers $k$, $\ell$ and fractional ideals $K$, $L$ over $R$ such that $JK = I^k$ and $IL = J^\ell$ \cite{Moody} Cor.÷ 2.\end{eg}




\begin{eg} \label{5.12}$\hint$ Determine the equations in $X \times \PPP^{2}$ of the blowup of $X=\AAA^3$ along the image $Z$ of the monomial curve $(t^3,t^4,t^5)$ of equations $g_1=y^2-xz$, $g_2=yz-x^3$, $g_3=z^2-x^2y$.\end{eg}


\begin{eg} \label{5.13}$\hint$ The blowup of the cone $X=V(x^2-yz)$ in $\AAA^3$ along the $z$-axis $Z=V(x,y)$ is an isomorphism locally at all points outside $0$, but not globally on $X$.
\end{eg}


\begin{eg} \label{5.14}$\hint$ Blow up the non-reduced point $X=V(x^2)$ in $\AAA^2$ in the (reduced) origin $Z=0$. \end{eg}


\begin{eg} \label{5.15}$\hint$ Blow up the subscheme $X=V(x^2, xy)$ of $\AAA^2$ in the reduced origin.
\end{eg}


\begin{eg} \label{5.16} Blow up the subvariety $X=V(xz,yz)$ of $\AAA^3$ first in the origin, then in the $x$-axis, and determine the points where the resulting morphisms are local isomorphisms.
\end{eg}


\begin{eg} \label{5.17}$\hint$ The blowups of $\AAA^3$ along the union of the $x$- with the $y$-axis, respectively along the cusp, with ideals $(xy,z)$ and $(x^3-y^2,z)$, are singular.
\end{eg}

\ignore 
\recognize

\section{Lecture VI: Transforms of Ideals and Varieties under Blowup}

\noindent Throughout this section, $\pi: W' \too W$ denotes the blowup of a variety $W$ along a closed subvariety $Z$, with exceptional divisor $E=\pi\inv(Z)$ defined by the principal ideal $I_E$ of $\calo_{W'}$. Let $\pi^*:\calo_W\too \calo_{W'}$ be the dual homomorphism of $\pi$. Let $X \subset W$ be a closed subvariety, and let $I$ be an ideal on $W$. Several of the subsequent definitions and results can be extended to the case of arbitrary birational morphisms, taking for $Z$ the complement of the open subset of $W$ where the inverse map of $\pi$ is defined.


\begin{defn} \label{totaltransform}The inverse image $X^*=\pi^{-1}(X)$ of $X$ and the extension $I^*=\pi^*(I)=I\cdot \calo_{W'}$ of $I$ are called the \textit{total transform} of $X$ and $I$ under $\pi$. For $f\in \calo_W$, denote by $f^*$ its image $\pi^*(f)$ in $\calo_{W'}$. The ideal $I^*$ is generated by all transforms $f^*$ for $f$ varying in $I$. If $I$ is the ideal defining $X$ in $W$, the ideal $I^*$ defines $X^*$ in $W'$. In particular, the total transform of the center $Z$ equals the exceptional divisor $E$, and $W^*=W'$. In the category of schemes, the total transform will in general be non-reduced when considered as a subscheme of $W'$.\end{defn}


\begin{defn}\label{stricttransform} The Zariski closure of $\pi^{-1}(X\setminus Z)$ in $W'$ is called the \textit{strict transform} of $X$ under $\pi$ and denoted by $X^s$, also known as the \textit{proper} or \textit{birational transform}. The strict transform is a closed subvariety of the total transform $X^*$. The difference $X^*\setminus X^s$ is contained in the exceptional locus $E$. If $Z\subset X$ is contained in $X$, the strict transform $X^s$ of $X$ in $W'$ equals the blowup $\wt X$ of $X$ along $Z$, cf.÷ Prop.÷ \ref{basechange} and its corollary. If the center $Z$ coincides with $X$, the strict transform $X^s$ is empty. 

Let $I_U$ denote the restriction of an ideal $I$ to the open set $U=W\setminus Z$.  Set $U'=\pi^{-1}(U)\subset W'$ and let $\tau:U'\to U$ be the restriction of $\pi$ to $U'$. The \textit{strict transform} $I^s$ of the ideal $I$ is defined as $\tau^*(I_U)\cap \calo_{W'}$. If the ideal $I$ defines $X$ in $W$, the ideal $I^s$ defines $X^s$ in $W'$. It equals the union of colon ideals 
\[
I^s =\bigcup_{i \geq 0}\, (I^* : I_E^i).
\]
Let $h$ be an element of $\calo_{W',a'}$ defining $E$ locally in $W'$ at a point $a'$. Then, locally at $a'$,
\[
I^s = (f^s,\, f\in I),
\]
where the strict transform $f^s$ of $f$ is defined at $a'$ and up to multiplication by invertible elements in $\calo_{W',a'}$ through $f^* = h^k\cdot f^s$ with maximal exponent $k$. The value of $k$ is the order of $f^*$ along $E$, cf.÷ def.÷ \ref{844}. By abuse of notation this is written as $f^s =h^{-k}\cdot f^*=h^{-\ord_Zf}\cdot f^*$.  \end{defn}


\comment{[We use here that the symbolic and the ordinary powers of $I_E$ coincide since $I_E$ is principal, otherwise we would have to pass to the localization $\calo_{W',E}$ of $\calo_{W'}$ along $E$. It also has to be shown that the order of $I^*$ along $E$ coincides with the order of $I$ along $Z$, which is not a Cartier divisor.]}


\begin{lem} Let $f:R\to S$ be a ring homomorphism, $I$ an ideal of $R$ and $s$ an element of $R$, with induced ring homomorphism $f_s:R_s\too S_{f(s)}$. Let $I^e=f(I)\cdot S$ and $(I\cdot R_s)^e=f_s(I\cdot R_s)\cdot S_{f(s)}$ denote the respective extensions of ideals. Then $\bigcup_{i\geq0} (I^e:f(s)^i)=(I\cdot R_s)^e\cap S$.\end{lem}


\begin{proof} Let $u\in S$. Then $u\in \bigcup_{i\geq 0} (I^e:f(s)^i)$ if and only if $uf(s)^i\in I^e$ for some $i\geq 0$, say $uf(s)^i=\sum_j a_jf(x_j)$ for elements $x_j\in I$ and $a_j\in S$. Rewrite this as $u=\sum_j a_j f(\frac{x_j}{s^i})$. This just means that
$u\in (I\cdot R_s)^e$. \end{proof}


\begin{rem} If $I$ is generated locally by elements $f_1, \dots, f_k$ of $\calo_W$,
then $I^s$ contains the ideal generated by the strict transforms $f_1^s, \dots, f_k^s$ of $f_1, \dots, f_k$, but the inclusion can be strict, see the examples below. \end{rem}


\begin{defn} Let $\K[x]=\K[x_1,\ldots ,x_n]$ be the polynomial ring over $\K$, considered with the natural grading given by the degree. Denote by $\inin(g)$ the homogeneous form of lowest degree of a non-zero polynomial $g$ of $\K[x]$, called the \textit{initial form} of $g$. Set $\inin(0)=0$. For a non-zero ideal $I$, denote by $\inin(I)$ the ideal generated by all initial forms $\inin(g)$ of elements $g$ of $I$, called the \textit{initial ideal} of $I$. Elements $g_1,\ldots,g_k$ of an ideal $I$ of $\K[x]$ are a \textit{Macaulay basis} of $I$ if their initial forms $\inin(g_1),\ldots,\inin(g_k)$ generate $\inin(I)$. In \cite{Hironaka_Annals} III.1, def.÷ 3, p.÷ 208, such a basis was called a \textit{standard basis}, which is now used for a slightly more specific concept, see Rem.÷ \ref{279} below. By noetherianity of $\K[x]$, any ideal possesses a Macaulay basis.\end{defn}


\begin{prop}(Hironaka) \label{6.6} The strict transform of an ideal under blowup in a regular center is generated by the strict transforms of the elements of a Macaulay basis of the ideal. \end{prop}


\begin{proof} {} (\cite{Hironaka_Annals}, III.2, Lemma 6, p.÷ 216, and III.6, Thm.÷ 5, p.÷ 238) If $I\subset J$ are two ideals of $\K[[x]]$ such that $\inin(I)=\inin(J)$ then they are equal, $I=J$. This holds at least for degree compatible monomial orders, due to the Grauert-Hironaka-Galligo division theorem. Therefore it has to be shown that $\inin(g_1^s),\ldots, \inin(g_k^s)$ generate $\inin(I^s)$. But $\inin(I^s)=(\inin(I))^s$, and the assertion follows. \end{proof} \comment{[Do we have to assume here that the center $Z$ is non-singular?]}


\begin{rem} \label{279} The strict transform of a Macaulay basis at a point $a'$ of $W'$ need no longer be a Macaulay basis. \comment{[insert example]} This is however the case if the Macaulay basis is \textit{reduced} and the sequence of its orders has remained constant at $a'$, cf. \cite{Hironaka_Annals} III.8, Lemma 20, p.÷ 254. More generally, taking on $\K[x]$ instead of the grading by degree a grading so that all homogeneous pieces are one-dimensional and generated by monomials (i.e., a grading induced by a monomial order on $\NNN^n$), the initial form of a polynomial and the initial ideal are both monomial. In this case Macaulay bases are called \textit{standard bases}. A standard basis $g_1,\ldots,g_k$ is \textit{reduced} if no monomial of the tails $g_i-\inin(g_i)$ belongs to $\inin(I)$. If the monomial order is degree compatible, i.e., the induced grading a refinement of the natural grading of $\K[x]$ by degree, the strict transforms of the elements of a standard basis of $I$ generate the strict 
transform of the ideal. \comment{[insert proof]} \end{rem}


\begin{defn} Let $I$ be an ideal on $W$ and let $c\geq 0$ be a natural number less than or equal to the order $d$ of $I$ along the center $Z$, cf.÷ def.÷ \ref{844}. Then $I^*$ has order $\geq c$ along $E$. There exists a unique ideal $I^!$ of $\calo_{W'}$ such that $I^*= I_E^c\cdot I^!$, called the \textit{controlled transform} of $I$ with respect to the \textit{control} $c$. It is not defined for values of $c>d$. In case that $c=d$ attains the maximal value, $I^!$ is denoted by $I^\curlyvee$ and called the \textit{weak transform} of $I$. It is written as $I^\curlyvee =I_E^{-d}\cdot I^*=I_E^{-\ord_ZI}\cdot I^*$.\end{defn}


\begin {rem}The inclusions $I^*\subset I^\curlyvee\subset I^!\subset I^s$ are obvious. The components of $V(I^\curlyvee)$ which are not contained in $V(I^s)$ lie entirely in the exceptional divisor $E$, but can be strictly contained. For principal ideals, $I^\curlyvee$ and $I^s$ coincide. When the transforms are defined scheme-theoretically, the reduction $X^*_{red}$ of the total transform $X^*$ of $X$ consists of the union of $E$ with the strict transform $X^s$.\end{rem}


\comment{[? Let $x_1,\ldots,x_n$ be local coordinates on $W$ at $a$, and let $\pi: W'\too W$ be the blowup of $W$ along $Z$. Let $a'$ be a point of $W'$ above $a$, with local blowup $\pi_{a'}: (W',a') \too (W,a)$ and dual homomorphism $\pi_{a'}^*:\calo_{W,a} \too \calo_{W',a'}$. The images of $x_1,\ldots,x_n$ under $\pi_{a'}^*$ are called the \textit{coordinates at $a'$ induced by $x_1,\ldots,x_n$}.]} 


\begin{defn} A \textit{local flag} $\F$ on $W$ at $a$ is a chain $F_0=\{a\}\subset F_1\subset \ldots \subset  F_n=W$ of regular closed subvarieties, respectively subschemes,  $F_i$  of dimension $i$ of an open neighbourhood $U$ of $a$ in $W$. Local coordinates $x_1,\ldots,x_n$ on $W$ at $a$ are called \textit{subordinate to the flag $\F$} if $F_i=V(x_{i+1},\ldots,x_n)$ locally at $a$. The flag $\F$ at $a$ is \textit{transversal} to a regular subvariety $Z$ of $W$ if each $F_i$ is transversal to $Z$ at $a$ \cite{Ha_Power_Series, Panazzolo}. \end{defn}


\begin{prop} Let $\pi:W'\too W$ be the blowup of $W$ along a center $Z$ transversal to a flag $\F$ at $a\in Z$. Let $x_1,\ldots,x_n$ be local coordinates on $W$ at $a$ subordinate to $\F$. At each point $a'$ of $E$ above $a$ there exists a unique local flag $\F'$ such that the coordinates $x_1',\ldots,x_n'$ on $W'$ at $a'$ induced by $x_1,\ldots,x_n$ as in def.÷ \ref{blowup6} are subordinate to $\F'$ \cite{Ha_Power_Series} Thm.÷ 1.  \end{prop}

\begin{defn} The flag $\F'$ is called the \textit{transform of $\F$ under $\pi$}.\end{defn}


\begin{proof} It suffices to define $F_i'$ at $a'$ by $x_{i+1}',\ldots,x_n'$. For point blowups in $W$, the transform of $\F$ is defined as follows. The point $a'\in E$ is determined by a line $L$ in the tangent space T$_aW$ of $W$ at $a$. Let $k\leq n$ be the minimal index for which T$_aF_k$ contains $L$. For $i< k$, choose a regular $(i+1)$-dimensional subvariety $H_i$ of $W$ with tangent space $L$ + T$_aF_i$ at $a$. In particular, T$_aH_{k-1}=$ T$_aF_k$. Let $H_i^s$ be the strict transform of $H_i$ in $W'$. Then set 
\[F_i'=E\cap H_i^s\hskip 1cm \mathrm{ for }\hskip .1cm  i<k,\]
\[F_i'=F_i^s \hskip 1.7cm \mathrm{ for } \hskip .1cm i\geq k,\] 
to get the required flag $\F'$ at $a'$ .\end{proof}


\bigskip
\noindent\textit{\examples}


\begin{eg} \label{6.13} Blow up $\AAA^2$ in $0$ and compute the inverse image of the zerosets of $x^2+y^2=0$, $xy=0$ and $x(x-y^2)=0$, as well as their strict transforms.  \end{eg}


\begin{eg} \label{6.14}$\hint$ Compute the strict transform of $X=V(x^2-y^3,xy-z^3)\subset\AAA^3$ under the blowup of the origin. \end{eg}


\begin{eg} \label{6.15} Determine the total, weak and strict transform of $X=V(x^2-y^3,z^3)\subset\AAA^3$ under the blowup of the origin. Clarify the algebraic and geometric differences between them. \end{eg}


\begin{eg} \label{6.16} Blow up $\AAA^3$ along the curve $y^2-x^3+x=z=0$ and compute the strict transform of the lines $x=z=0$ and $y=z=0$. \end{eg}


\begin{eg} \label{6.17} Let $I_1$ and $I_2$ be ideals of $\K[x]$ of order $c_1$ and $c_2$ at $0$. Take the blowup of $\AAA^n$ at zero. Show that the weak transform of $I_1^{c_2}+I_2^{c_1}$ is the sum of the weak transforms of $I_1^{c_2}$ and $I_2^{c_1}$.\end{eg}


\begin{eg} \label{6.18}$\hint$ For $W=\AAA^3$, a flag at a point $a$ consists of a regular curve $F_1$ through $a$ and contained in a regular surface $F_2$. Blow up the point $a$ in $W$, so that $E\isom \PPP^2$ is the projective plane. The induced flag at a point $a'$ above $a$ depends on the location of $a$ on $E$: At the intersection point $p_1$ of the strict transform $C_1=F_1^s$ of $F_1$ with $E$, the transformed flag $\F'$ is given by $F_1^s\subset F_2^s$. Along the intersection $C_2$ of $F_2^s$ with $E$, the flag $\F'$ is given at each point $p_2$ different from $p_1$ by $C_2\subset F_2^s$. At any point $p_3$ not on $C_2$ the flag $\F'$ is given by  $C_3\subset E$, where $C_3$ is the projective line in $E$ through $p_1$ and $p_3$.  \end{eg}


\begin{eg} \label{6.19}$\hint$ Blowing up a regular curve $Z$ in $W=\AAA^3$ transversal to $\F$ there occur six possible configurations of $Z$ with respect to $\F$. Denoting by $L$ the plane in the tangent space $\mathrm{T}_aW$ of $W$ at $a$ corresponding to the point $a'$ in $E$ above $a$, these are: (1) $Z=F_1$ and $L=\mathrm{T}_aF_2$, (2) $Z=F_1$ and $L\neq\mathrm{T}_aF_2$, (3) $Z\neq F_1$, $Z\subset F_2$ and $L=\mathrm{T}_aF_2$, (4) $Z\neq F_1$, $Z\subset F_2$ and $L\neq\mathrm{T}_aF_2$, (5) $Z\not\subset F_2$ and $\mathrm{T}_aF_1\subset L$, (6) $Z\not\subset F_2$ and $\mathrm{T}_aF_1\not\subset L$. Determine in each case the flag $\F'$.
\end{eg}

\begin{eg} \label{6.20} Let $\wt \AAA^n\too \AAA^n$ be the blowup of $\AAA^n$ in $0$ and let $a'$ be the origin of the $x_n$-chart of $\wt \AAA^n$. Compute the total and strict transforms $g^*$ and $g^s$ for $g=x_1^d+\ldots +x_{n-1}^d+x_n^e$ for $e=d, 2d-1, 2d, 2d+1$ and $g=\prod_{i\neq j} (x_i-x_j)$.\end{eg}


\begin{eg} \label{6.21} Determine the total, weak and strict transform of $X=V(x^2-y^3,z^3)\subset\AAA^3$ under the blowup of $\AAA^3$ at the origin. Point out the geometric differences between the three types of transforms.\end{eg}




\begin{eg} \label{6.22}$\hint$ The inclusion $I^\curlyvee \subset I^s$ can be strict. Blow up $\AAA^2$ at $0$, and consider in the $y$-chart of $\wt \AAA^2$ the transforms of $I=(x^2,y^3)$. Show that $I^\curlyvee =(x^2, y)$, $I^s=(x^2,1)=\K[x,y]$. \end{eg}


\begin{eg} \label{6.23}$\hint$ Let $X=V(f)$ be a hypersurface in $\AAA^n$ and $Z$ a regular closed subvariety which is contained in the locus of points of $X$ where $f$ attains its maximal order. Let $\pi:\wt\AAA^n\to\AAA^n$ be the blowup along $Z$ and let $X'=V(f')$ be the strict transform of $X$. Show that for points $a\in Z$ and $a'\in E$ with $\pi(a')=a$ the inequality $\ord_{a'}f'\leq\ord_a f$ holds.\end{eg}




\section{Lecture VII: Resolution Statements}


\begin{defn} A \textit{non-embedded resolution} of the singularities of a variety $X$ is a non-singular variety $\wt X$ together with a proper birational morphism $\pi: \wt X \too X$ which induces a biregular isomorphism $\pi:\wt X\setminus E\too X\setminus \Sing(X)$ outside $E=\pi\inv(\Sing(X))$. \end{defn}


\begin{rem} Requiring properness excludes trivial cases as e.g.÷ taking for $\wt X$ the locus of regular points of $X$ and for $\pi$ the inclusion map.  One may ask for additional properties: (a) Any automorphism $\varphi:X\too X$ of $X$ shall lift to an automorphism $\wt \varphi:\wt X\too \wt X$ of $\wt X$ which commutes with $\pi$, i.e., $\pi\circ\wt\varphi=\varphi\circ\pi$. (b) If $X$ is defined over $\K$ and $\K\subset \LL$ is a field extension, any resolution of $X_\LL=X\times_\K \Spec(\LL)$ shall induce a resolution of $X=X_\K$.
\end{rem}


\begin{defn} \comment{[?]} A \textit{local non-embedded resolution} of a variety $X$ at a point $a$ is the germ $(\wt X,a')$ of a non-singular variety $\wt X$ together with a local morphism $\pi: (\wt X,a') \too (X,a)$ inducing an isomorphism of the function fields of $\wt X$ and $X$. \end{defn}


\comment{[A \textit{local uniformization} of a variety $X$ at a point $a$ is, for any valuation of the function field of $X$, ...]}


\begin{defn} Let $X$ be an affine irreducible variety with coordinate ring $R =\K[X]$. A (local, ring-theoretic) \textit{non-embedded resolution} of $X$ is a ring extension $R \hookrightarrow \wt R$ of $R$ into a regular ring $\wt R$ having the same quotient field as $R$.\end{defn}


\comment{[{rem}: In order to define a global resolution ring-theoretically, one has to consider several extensions of $R$ glued by transition maps.]}


\begin{defn} Let $X$ be a subvariety of a regular ambient variety $W$. An \textit{embedded resolution} of $X$ consists of a proper birational morphism $\pi: \wt W\too W$ from a regular variety $\wt W$ onto $W$ which is an isomorphism over $W\setminus \Sing(X)$ such that the strict transform $X^s$ of $X$ is regular and the total transform $X^* =\pi^{-1}(X)$ of $X$ has simple normal crossings.
\end{defn}


\begin{defn} A \textit{strong resolution} of a variety $X$ is, for each closed embedding of $X$ into a regular variety $W$, a birational proper morphism $\pi:\wt W\too W$ satisfying the following five properties \cite{EH}:\end{defn}

\textit{Embeddedness}. The variety $\wt W$ and the strict transform $X^s$ of $X$ are regular, and the total transform $X^* =\pi^{-1}(X)$ of $X$ in $\wt W$ has simple normal crossings.\medskip


\textit{Equivariance}. Let $W' \too W$ be a smooth morphism and let $X'$ be the inverse
image of $X$ in $W'$. The morphism $\pi': \wt W'  \too W'$
induced by $\pi: \wt X \too W$ by taking fiber product of $\wt X$ and $W'$ over $W$ is an embedded resolution of $X'$. \medskip


\textit{Excision}. The restriction $\tau:\wt X\too X$ of $\pi$ to $X$ does not depend on the choice of the embedding of $X$ in $W$. \comment{[This property requires that a resolution $\pi$ is associated to any pair $X\subset W$.]}\medskip


\textit{Explicitness}. The morphism $\pi$ is a composition of blowups along regular centers 
which are transversal to the exceptional loci created by the earlier blowups.\medskip


\textit{Effectiveness}. There exists, for all varieties $X$, a local upper semicontinuous invariant $\invv_a(X)$ in a well-ordered set $\Gamma$, depending only and up to isomorphism of the completed local rings of $X$ at $a$, such that: (a) $\invv_a(X)$ attains its minimal value if and only if $X$ is regular at $a$ (or has normal crossings at $a$); (b) the top locus $S$ of $\invv_a(X)$ is closed and regular in $X$; (c) blowing up $X$ along $S$ makes $\invv_a(X)$ drop at all points $a'$ above the points $a$ of $S$. 


\begin{rem} (a) Equivariance implies the \textit{economy} of the resolution, i.e., that $\pi:X^s\too X$ is an isomorphism outside $\Sing(X)$. It also implies that $\pi$ commutes with open immersions, localization, completion,  automorphisms of $W$ stabilizing $X$ and taking cartesian products with regular varieties. 

(b) One may require in addition that the centers of a resolution are transversal to the inverse images of a given normal crossings divisor $D$ in $W$, the \textit{boundary}.\end{rem}


\begin{defn} Let $I$ be an ideal on a regular ambient variety $W$. A \textit{log-resolution} of $I$ is a proper birational morphism $\pi: \wt W\too W$ from a regular variety $\wt W$ onto $W$ which is an isomorphism over $W\setminus \Sing(I)$ such that $I^* =\pi^{-1}(I)$ is a locally monomial ideal on $\wt W$.
\end{defn}

\goodbreak
\medskip 
\noindent\textit{\examples}


\begin{eg} \label{7.9} The blowup of the cusp $X: x^2=y^3$ in $\AAA^2$ at $0$ produces a non-embedded resolution. Further blouwps give an embedded resolution.\end{eg}


\begin{eg} \label{7.10}$\hint$ Determine the geometry at $0$ of the hypersurfaces in $\AAA^4$ defined by the following equations:

(a) $x +x^7 - 3yw + y^7z^2 + 17yzw^5 = 0$,

(b) $x + x^5y^2 - 3yw + y^7z^2 + 17yzw^5 = 0$,

(c) $x^3y^2 + x^5y^2 - 3yw + y^7z^2 + 17yzw^5 = 0$.\end{eg}


\begin{eg} \label{7.11} Let $X$ be the variety in $\AAA_\CCC^3$ given by the equation $(x^2-y^3)^2=(z^2-y^2)^3$. Show that the map $\alpha: \AAA^3 \too \AAA^3: (x,y,z) \too (u^2z^3,u yz^2, uz^2)$, with $u(x,y)=x(y^2-1)+y$, resolves the singularities of $X$. What is the inverse image $\alpha^{-1}(Z)$? Produce instructive pictures of $X$ over $\RRR$.\end{eg}


\begin{eg} \label{7.12} Consider the inverse images of the cusp $X=V(y^2-x^3)$ in $\AAA^2$ under the maps $\pi_x,\pi_y:\AAA^2 \too\AAA^2$, $\pi_x(x,y)=(x,xy)$, $\pi_y(x,y)=(xy,y)$. Factor the maximal power of $x$, respectively $y$, from the equation of the inverse image of $X$ and show that the resulting equation defines in both cases a regular variety. Apply the same process to the variety $E_8=V(x^2+y^3+z^5)\subset\AAA^3$ repeatedly until all resulting equations define non-singular varieties.\end{eg}


\begin{eg} \label{7.13} Let $R$ be the coordinate ring of an irreducible plane algebraic curve $X$. The integral closure $\wt R$ of $R$ in the field of fractions of $R$ is a regular ring and thus resolves $R$. The resulting curve $\wt X$ is the normalization of $X$ \cite{Mumford} III.8, \cite{DeJong_Pfister} 4.4. \end{eg}


\begin{eg} \label{7.14}$\hint$ Consider a cartesian product $X=Y\times Z$ with $Z$ a regular variety. Show that a resolution of $X$ can be obtained from a resolution $Y'$ of $Y$ by taking the cartesian product $Y'\times Z$ of $Y'$ with $Z$.\end{eg}


\begin{eg} \label{7.15} Let $X$ and $Y$ be two varieties (schemes, analytic spaces) with singular loci $\Sing(X)$ and $\Sing(Y)$ respectively. Suppose
that $X'$ and $Y'$ are regular varieties (schemes, analytic spaces, within the same category as $X$ and $Y$) together with proper birational morphisms $\pi : X' \too X$ and $\tau: Y'
\too Y$ which define resolutions of $X$ and $Y$ respectively, and which are isomorphisms outside $\Sing(X)$ and $\Sing(Y)$. There is a naturally defined proper birational morphism $f: X' \times Y' \too X \times Y$ giving rise to a resolution of $X\times Y$.\end{eg}


\goodbreak

\section{Lecture VIII: Invariants of Singularities}


\begin{defn}A \textit{stratification} of an algebraic variety $X$ is a decomposition of $X$ into finitely many disjoint locally closed subvarieties $X_i$, called the \textit{strata},
\[
X = \dot\bigcup_{i}\, X_i,
\]
such that the \textit{boundaries} $\overline X_i\setminus X_i$ of strata are unions of strata. This last property is called the \textit{frontier condition}.  Two strata are called \textit{adjacent} if one lies in the closure of the other.\end{defn}


\begin{defn} A \textit{local invariant} on an algebraic variety $X$ is a function $\invv(X): X\too \Gamma$ from $X$ to a well-ordered set $(\Gamma, \leq)$ which associates to each point $a \in X$ an element $\invv_a(X)$ depending only on the formal isomorphism class of $X$ at $a$: If $(X,a)$ and $(X,b)$ are formally isomorphic, viz $\wh \calo_{X,a}\isom \wh \calo_{X,b}$, then $\invv_a(X)=\invv_b(X)$. Usually, the ordering on $\Gamma$ will also be total: for any $c,d\in\Gamma$ either $c\leq d$ or $d\leq c$ holds. The invariant is \textit{upper semicontinuous along a subvariety $S$} of $X$ if for all $c\in \Gamma$, the sets
\[
\ttop_S(\invv, c)=\{a\in S,\, \invv_a(X)\geq c\}
\] 
are closed in $S$. If $S=X$, the map $\invv(X)$ is called \textit{upper semicontinuous}.\end{defn}

\begin{rem} The upper semicontinuity signifies that the value of $\invv_a(X)$ can only go up or remain the same when $a$ approaches a limit point. In the case of schemes, the value of $\invv(X)$ also has to be defined and taken into account at non-closed points of $X$. \end{rem}


\begin{defn} \label{844} Let $X$ be a subvariety of a not necessarily regular ambient variety $W$ defined by an ideal $I$, and let $Z$ be an irreducible subvariety of $W$ defined by the prime ideal $J$. The \textit{order} of $X$ or $I$ in $W$ along $Z$ or with respect to $J$ is the maximal integer $k=\ord_Z(X)=\ord_Z(I)$ such that $I_Z\subset J_Z^k$, where $I_Z=I\cdot\calo_{W,Z}$ and $J_Z=J\cdot\calo_{W,Z}$ denote the ideals generated by $I$ and $J$ in the localization $\calo_{W,Z}$ of $W$ along $Z$. If $Z=\{a\}$ is a point of $W$, the order of $X$ and $I$ at $a$ is denoted by $\ord_a(X)=\ord_a(I)$ or $\ord_{\mm_a}(X)=\ord_{\mm_a}(I)$.\end{defn}


\begin{defn} Let $R$ be a local ring with maximal ideal $\mm$. Let $k\in\NNN$ be an integer. The $k$-th symbolic power $J^{(k)}$ of a prime ideal $J$ is defined as the ideal generated by all elements $x\in R$ for which there is an element $y\in R\setminus J$ such that $y\cdot x^k\in J^k$. Equivalently, $J^{(k)}=J^k\cdot R_J\cap R$.\end{defn}


\begin{rem} The symbolic power is the smallest $J$-primary ideal containing $J^k$. If $J$ is a complete intersection, the ordinary power $J^k$ and the symbolic power $J^{(k)}$ coincide \cite{Zariski_Samuel} IV, \cite{Hochster_73}, \S12, \cite{Pellikaan}.\end{rem}


\begin{prop} Let $X$ be a subvariety of a not necessarily regular ambient variety $W$ defined by an ideal $I$, and let $Z$ be an irreducible subvariety of $W$ defined by the prime ideal $J$. The order of $X$ along $Z$ is the maximal integer $k$ such that $I\subset J^{(k)}$.\end{prop}


\begin{proof} This follows from the equality $J^k\cdot R_J =(J\cdot R_J)^k$. \end{proof}


\comment{[Compare the order with the multiplicity at a point, defined as the leading coefficient of the Hilbert-Samuel function divided by $n!$, independently of the embedding.]}


\begin{prop} \label{222} Let $R$ be a noetherian local ring with maximal ideal $\mm$, and let $\wh R$ denote its completion with maximal ideal $\wh m=m\cdot \wh R$. Let $I$ be an ideal of $R$, with completion $\wh I=I\cdot \wh R$. Then $\ord_\mm(I) =\ord_{\wh\mm}({\wh I})$. \end{prop}


\begin{proof} {} If $I\subset \mm^k$, then also $\wh I\subset \wh \mm^k$, and hence $\ord_\mm(I) \leq\ord_{\wh\mm}({\wh I})$.  Conversely, $\mm=\wh\mm\cap R$ and $\bigcap_{i\geq0}(I+\mm^i)=\wh I\cap R$ by Lemma \ref{191}. Hence, if $\wh I\subset \wh \mm^k$, then $I\subset \wh I\cap R\subset \wh \mm^k\cap R=\mm^k$, so that $\ord_\mm(I) \geq\ord_{\wh\mm}({\wh I})$.\end{proof}


\begin{defn} Let $X$ be a subvariety of a regular variety $W$, and let $a$ be a point of $W$. The \textit{local top locus} $\ttop_a(X)$ of $X$ at $a$ with respect to the order is the stratum $S$ of points of an open neighborhood $U$ of $a$ in $W$ where the order of $X$ equals the order of $X$ at $a$. The \textit{top locus} $\ttop(X)$ of $X$ with respect to the order is the (global) stratum $S$ of points of $W$ where the order of $X$ attains its maximal value. For $c\in \NNN$, define $\ttop_a(X,c)$ and $\ttop(X,c)$ as the local and global stratum of points of $W$ where the order of $X$ is at least $c$. \end{defn}


\begin{rem} The analogous definition holds for ideals on $W$ and can be made for other local invariants. By the upper semicontinuity of the order, the local top locus of $X$ at $a$ is locally closed in $W$, and the top locus of $X$ is closed in $W$.\end{rem}


\begin{prop} The order of a variety $X$ or an ideal $I$ in a regular variety $W$ at points of $W$ defines an upper semicontinuous local invariant on $W$.\end{prop}

\begin{proof} {} In characteristic zero, the assertion follows from the next proposition. For the case of positive characteristic, see \cite{Hironaka_Annals} III.3, Cor.÷ 1, p.÷ 220.\end{proof}


\begin{prop} Over fields of zero characteristic, the local top locus $\ttop_a(I)$ of an ideal $I$ at $a$ is defined by the vanishing of all partial derivatives of elements of $I$ up to order $o-1$, where $o$ is the order of $I$ at $a$.\end{prop}


\begin{proof} In zero characteristic, a polynomial has order $o$ at a point $a$ if and only if all its partial derivatives up to order $o-1$ vanish at $a$. \end{proof}


\begin{prop} \label{899} Let $X$ be a subvariety of a regular ambient variety $W$. Let $Z$ be a non-singular subvariety of $W$ and $a$ a point on $Z$ such that locally at $a$ the order of $X$ is constant along $Z$, say equal to $d=\ord_aX=\ord_Z X$. Consider the blowup $\pi:W'\to W$ of $W$ along $Z$ with exceptional divisor $E=\pi^{-1}(Z)$. Let $a'$ be a point on $E$ mapping under $\pi$ to $a$. Denote by $X^*$, $X^\curlyvee$ and $X^s$ the total, weak and strict transform of $X$ respectively (def.÷ \ref{totaltransform} and  \ref{stricttransform}). Then, locally at $a'$, the order of $X^*$ along $E$ is $d$, and
\[
\ord_{a'}X^s\leq \ord_{a'}X^\curlyvee\leq d.
\]
\end{prop}


\begin{proof} By Prop.÷ \ref{111} there exist local coordinates $x_1,\ldots,x_n$ of $W$ at $a$ such that $Z$ is defined locally at $a$ by $x_1=\ldots=x_k=0$ for some $k\leq n$, and such that $a'$ is the origin of the $x_1$-chart of the blowup. Let $I\subset \calo_{W,a}$ be the local ideal of $X$ in $W$ at $a$, and let $f$ be an element of $I$. It has an expansion $f=\sum_{\alpha\in\mathbb{N}^n}c_\alpha x^\alpha$ with respect to the coordinates $x_1,\ldots,x_n$, with coefficients $c_\alpha\in\K$. 

Set $\alpha_+=(\alpha_1,\ldots,\alpha_k)$. Since the order of $X$ is constant along $Z$, the inequality $|\alpha_+|=\alpha_1+\ldots+\alpha_k\geq d$ holds whenever $c_\alpha\neq0$. There is an element $f$ with an exponent $\alpha$ such that $c_\alpha\neq0$ and such that $|\alpha|=|\alpha_+|=d$. The total transform $X^*$ and the weak transform $X^\curlyvee$ are given locally at $a'$ by the ideal generated by all
\[
f^*=\sum_{\alpha\in\NNN^n} c_\alpha x_1^{|\alpha_+|}x_2^{\alpha_2}\ldots x_n^{\alpha_n},
\]
respectively 
\[
f^\curlyvee=\sum_{\alpha\in\NNN^n}c_ \alpha x_1^{|\alpha_+|-d}x_2^{\alpha_2}\ldots x_n^{\alpha_n},
\]
for $f$ varying in $I$. The exceptional divisor $E$ is given locally at $a'$ by the equation $x_1=0$. This implies that, locally at $a'$, $\ord_E f^*\geq d$ for all $f\in I$ and $\ord_E \,f^*=d$ for the special $f$ chosen above with $|\alpha|=|\alpha_+|=d$. Therefore $\ord_E X^*=d$ locally at $a'$. 

Since $\ord_{a'}\ f^\curlyvee\leq |\alpha_-|\leq d$ for the chosen $f$, it follows that $\ord_{a'}X^\curlyvee\leq d$. The ideal $I^s$ of the strict transform $X^s$ contains the ideal $I^\curlyvee$ of the weak transform $X^\curlyvee$, thus also $\ord_{a'}X^s\leq \ord_{a'}X^\curlyvee$.\end{proof}


\begin{cor} If the order of $X$ is globally constant along $Z$, the order of $X^*$ along $E$ is globally equal to $d$. \end{cor}


\comment{[Lemma SP: Let $I$ be an ideal on $W$ and assume that $I$ has order $d$ along the center $Z$. Then the order of the total transform $I^*$ along $E$ is $d$.
Proof: Let $U$ be the set $U=\{a\in Z:\ord_aI=d\}$. By the upper-semicontinuity of the order, this set is open in $Z$. Thus, at every $a\in U$, the order of $I$ is locally constant along $Z$. By Prop.÷ \ref{899}, locally at all $a'\in \pi^{-1}(U)$ the order of $I^*$ along $E$ is $d$. By upper-semicontinuity of the order, the order of $I^*$ along the closure of $\pi^{-1}(U)$ is $d$. If $U$ is dense in $Z$, then also $\pi^{-1}(U)$ is dense in $E$ and we are done. If $U$ is not dense in $Z$, then $Z_1=Z\setminus\overline{U}$ is open in $Z$ and $\ord_{Z_1}I>d$. We can apply the previous argument again to $Z_1$ and find an open subset $U_1\subset Z_1$ such that $\ord_{a'}I^*>d$ for all $a'$ in the closure of $\pi^{-1}(U_1)$. The process terminates after finitely many steps since $\ord_a I$ assumes only finitely many values for $a\in Z$.]}


\begin{defn} Let $X$ be a subvariety of a regular variety $W$, and let $W'\too W$ be a blowup with center $Z$. Denote by $X'$ the strict or weak transform in $W'$. A point $a'\in W'$ above $a\in Z$ is called \textit{infinitesimally near to $a$} or \textit{equiconstant} if $\ord_{a'}X'=\ord_aX$.  \end{defn}


\begin{defn} Let $X$ be a variety defined over a field $\K$, and let $a$ be a point of $W$. The \textit{Hilbert-Samuel function} $\mathrm{HS}_{a}(X):\NNN\too \NNN$ of $X$ at $a$ is defined by
\[
\mathrm{HS}_{a}(X)(k)=\dim_\K(\mm_{X,a}^k/\mm_{X,a}^{k+1}),
\]
where $\mm_{X,a}$ denotes the maximal ideal of the local ring $\calo_{X,a}$ of $X$ at $a$.  If $X$ is a subvariety of a regular variety $W$ defined by an ideal $I$, with local ring $\calo_{X,a} =\calo_{W,a}/I$, one also writes $\mathrm{HS}_{a}(I)$ for $\mathrm{HS}_{a}(X)$.\end{defn}


\begin{rem} The Hilbert-Samuel function does not depend on the embedding of $X$ in $W$. There exists a univariate polynomial $P(t)\in\QQQ[t]$, called the \textit{Hilbert-Samuel polynomial} of $X$ at $a$, such that $HS_a(X)(k) = P(k)$ for all sufficiently large $k$ \cite{Serre}. The Hilbert-Samuel polynomial provides local information on the singularity of $X$ at a point as e.g.÷ the multiplicity and the local dimension.\end{rem}


\begin{thm} The Hilbert-Samuel function of a subvariety $X$ of a regular variety $W$ defines an upper semicontinuous local invariant on $X$ with respect to the lexicographic ordering of integer sequences.\end{thm}

\begin{proof} {} \cite{Bennett} Thm.÷ 4, p.÷ 82, cf.÷ also \cite{Ha_Power_Series}.\end{proof}


\comment{{defn}: A variety $X$ is called \textit{normally flat} along a subvariety $Z$ if the Hilbert-Samuel function of $X$ is constant along $Z$ \cite{Hironaka_Annals}.}


\begin{thm} Let $\pi: W' \To W$ be the blowup of $W$ along a non-singular center $Z$. Let $I$ be an ideal of $\calo_W$. Assume that the Hilbert-Samuel function of $I$ is constant along $Z$, and denote by $I^s$ the strict transform of $I$ in $W'$. Let $a' \in E$ be a point in the exceptional divisor mapping under $\pi$ to $a$. Then 
\[
\mathrm{HS}_{a'}(I^s)\leq \mathrm{HS}_{a}(I) 
\]
holds with respect to the lexicographic ordering of integer sequences.\end{thm}


\begin{proof} {} \cite{Bennett} Thm.÷ 0, \cite{Ha_Power_Series}.\end{proof}


\comment{[The minimal initial ideal of $I$ with respect to a degree compatible monomial order on $\NNN^n$ does not increase under blowup when passing to the strict transform \cite{Ha_Power_Series} Thm.÷ 6. Hironaka showed that the sequence $\nu_a(I)$ of orders of a minimal Macaulay basis of $I$ does not increase under blowup when passing to the strict transform \cite{Hironaka_Annals} III.5, Lemma 14, p.÷ 227 and Thm.÷ 3, p.÷ 233.]}


\begin{thm} \label{order_under_localization} (Zariski-Nagata, \cite{Hironaka_Annals} III.3, Thm.÷ 1, p.÷ 218) Let $S\subset T$ be closed irreducible subvarieties of a closed subvariety $X$ of a regular ambient variety $W$. The order of $X$ in $W$ along $T$ is less than or equal to the order of $X$ in $W$ along $S$. \comment{[Moreover: The order of $X$ along $S$ is the minimum of the orders of $X$ at points of $S$.]} \end{thm}


\begin{rem} In the case of schemes, the assertion says that if $X$ is embedded in a regular ambient scheme $W$, and $a,b $ are points of $X$ such that $a$ lies in in the closure of $b$, then $\ord_b X\leq \ord_a X$.\end{rem}


\begin{proof} The proof goes in several steps and relies on the resolution of curves. Let $R$ be a regular local ring, $m$ its maximal ideal, $p$ a prime ideal in $R$ and $I\neq0$ a non-zero ideal in $R$. Denote by $\nu_p(I)$ the maximal integer $\nu\geq0$ such that $I\subset p^\nu$. Recall that $\ord_p(I)$ is the maximal integer $n$ such that $I R_p\subset p^n R_p$ or, equivalently, $I\subset p^{(\nu)}$, where $p^{(\nu)}=p^\nu R_p\cap R$ denotes the $\nu$-th symbolic power of $p$. Thus $\nu_p(I)\leq \ord_p(I)$ and the inequality can be strict if $p$ is not a complete intersection, in which case the usual and the symbolic powers of $p$ may differ.  Write $\nu(I)=\nu_m(I)$ for the maximal ideal $m$ so that $\nu_p(I)\leq\nu(I)$ and $\nu_p(I)\leq\nu(IR_p)=\ord_p(I)$. 

The assertion of the theorem is equivalent to the inequality $\nu(IR_p)\leq \nu(I)$, taking for $R$ the local ring $\calo_{W,S}$ of $W$ along $S$, for $I$ the ideal of $R$ defining $X$ in $W$ along $S$ and for $p$ the ideal defining $T$ in $W$ along $S$.  


If $R/p$ is regular, then $\nu_p(I)=\nu(IR_p)$ because $p^n R_p\cap R=p^n$: The inclusion $p^n\subseteq p^nR_p\cap R$ is clear, so suppose that $x\in p^nR_p\cap R$. Choose $y\in p^n$ and $s\not\in p$ such that $xs=y$. Then $n\leq \nu_p(y) = \nu_p(xs)=\nu_p(x)+\nu_p(s)=\nu_p(x)$, hence $x\in p^n$. \comment{[where did we use that $R/p$ is regular?]} It follows that $\nu(IR_p)\leq \nu(I)$.


It remains to prove the inequality in the case that $R/p$ is not regular. Since $R$ is regular, there is a chain of prime ideals $p=p_0\subset\ldots\subset p_k=m$ in $R$ with $\dim(R_{p_i}/p_{i-1}R_{p_i})=1$ for $1\leq i\leq n$.  By induction on the dimension of $R/p$, it therefore suffices to prove the inequality in the case $\dim(R/p)=1$. \comment{[this is probably too quick]} Since the order remains constant under completion of a local ring by Prop.÷ \ref{222}, it can be assumed that $R$ is complete. \comment{[where is this used below?]}


By the resolution of curve singularities there exists a sequence of complete regular local rings $R=R_0 \too R_1\too\ldots\too R_k$ with prime ideals $p_0=p$ and $p_i\subset R_i$ with the following properties: \comment{[is $p_{i+1}$ the weak transform of $p_i$?]}
\begin{enumerate}
\item $R_{i+1}$ is the blowup of $R_i$ with center the maximal ideal $m_i$ of $R_i$.
\item $(R_{i+1})_{p_{i+1}}=(R_i)_{p_i}$. 
\item $R_k/p_k$ is regular.
\end{enumerate}
Set $I_0=I$ and let $I_{i+1}$ be the weak transform of $I_i$ under $R_i\to R_{i+1}$. Set $R'=R_k$, $I'=I_k$ and $p'=p_k$. The second condition on the blowups $R_i$ of $R_{i-1}$ implies $\nu(IR_p)=\nu(I'R'_{p'})$. By Prop.÷ \ref{899} one knows that $\nu(I')\leq\nu(I)$. Further, since $R'/p'$ is regular, $\nu(I'R'_{p'})\leq \nu(I')$. Combining these inequalities yields $\nu(IR_p)\leq \nu(I)$. \end{proof}


\comment{\noindent [From Stefan: Order under localization, modified proof]}

\ignore

\begin{thm} \label{order_under_localization}
Let $X$ be a scheme, embedded in a regular ambient scheme $W$. Let $a,b \in X$ be points such that $a$ lies in in the closure of $\{b\}$. Then $\ord_b X\leq \ord_a X$.
\end{thm}

We will prove this result in a series of lemmas and propositions. In the following, let $R$ always be a regular local ring, $m$ its maximal ideal, $p$ a prime ideal in $R$ and $J\neq0$ a non-zero ideal in $R$. Denote with $\nu_p(J)$ the maximal integer $\nu\geq0$ such that 
$J\subseteq p^\nu$. Notice that this is different from the order of $J$ in $p$, which would be the maximal integer $n$ such that $J R_p\subseteq p^n R_p$. Write $\nu(J)=\nu_m(J)$ for the maximal ideal $m$.

\begin{lem}
$\nu_p(J)\leq\nu(J)$.
\end{lem}
\begin{proof}
Since $m$ is the maximal ideal, $J\subseteq p^n$ implies $J\subseteq m^n$ for all $n\geq0$.
\end{proof}

\begin{lem}
$\nu_p(J)\leq\nu(JR_p)$.
\end{lem}
\begin{proof}
The result follows from the fact that $J\subseteq p^n$ implies $J R_p\subseteq p^n R_p$ for all $n\geq0$.
\end{proof}

\begin{lem}
If $R/p$ is regular, then $\nu_p(J)=\nu(JR_p)$. 
\end{lem}
\begin{proof}
We will show that $p^n R_p\cap R=p^n$ for all $n\geq0$ by making use of a regular system of parameters $x_1,\ldots,x_k$ of $R$ such that $p=(x_1,\ldots,x_m)$ for an $m\leq k$. The inclusion $p^n\subseteq p^nR_p\cap R$ is clear. Now suppose $x\in p^nR_p\cap R$. There is a $y\in p^n$ and a $s\notin R$ such that $xs=y$. Consider the expansion of the elements $x,y,s$ with respect to the r.s.p. $x_1,\ldots,x_n$. Then we see that $n\leq\ord_{(x_1,\ldots,x_m)}y=\ord_{(x_1,\ldots,x_m)}x+\ord_{(x_1,\ldots,x_m)}s$. Since $s\notin p$, $\ord_{(x_1,\ldots,x_m)}s=0$. Thus, $\ord_{(x_1,\ldots,x_m)}x\geq n$ and $x\in p^n$.
\end{proof}


Combining the lemmas, we gain the following result:

\begin{prop} \label{nu_reg}
If $R/p$ is regular, then $\nu(JR_p)\leq \nu(J)$.
\end{prop}

It remains to prove this result in the case that $R/p$ is not regular. (This is a theorem due to Zariski-Nagata, cf. [Hi1]). We will do this is in several small steps, beginning with complete rings $R$ and prime ideals $p$ with $\dim(R/p)=1$. To prove the result in this case, we will use the resolution of curves. For this, we  introduce yet another notion of blowup: The blowup in the category of regular local rings, called monoidal transform.


\begin{defn}
Let $R$ be a regular local ring with prime ideal $p\neq0$ such that $R/p$ is regular. Let $x\in p$ and let $q$ be a prime ideal of $R[px^{-1}]$ containing the maximal ideal of $R$. Then the localization $R[px^{-1}]_q$ is called monoidal transform of $R$ with center $p$.
\end{defn}


\begin{prop} \label{nu_compl}
Assume that $R$ is complete and $\dim(R/p)=1$. Then $\nu(JR_p)\leq\nu(J)$.
\end{prop}
\begin{proof}
Let $S=R/p$. Denote the normalization of $S$ with $\overline{S}$. We know that $\overline{S}$ is a finite extension of $S$ in its quotient ring and that every localization of $\overline{S}$ is regular. We will now construct a sequence of regular local rings $R=R_0,R_1,\ldots,R_n$ with prime ideals $p_i\subseteq R_i$ ($p_0=p$) fulfilling the following properties:
\begin{itemize}
\item $R_{i+1}$ is a monoidal transform of $R_i$ with center the maximal ideal $m_i$.
\item $(R_{i+1})_{p_{i+1}}=(R_i)_{p_i}$.
\item $R_i/p_i$ is a finite extension of $S$ in its quotient ring and $R_i/p_i\subsetneq R_{i+1}/p_{i+1}\subseteq \overline{S}$. [Maybe have to pass to localization of $\overline{S}$.]
\end{itemize}
Assume that we have already constructed $R_0,\ldots,R_i$ and $R_i/p_i\neq\overline{S}$. Thus, there is an element $z=\frac{x}{s}\in\overline{S}\setminus R_i/p_i$ with $x\in R_i/p_i$ and $s\in(R_i/p_i)^*$. By abuse of notation, use the same symbols for representatives $x\in R_i$ and $s\in R_i\setminus p_i$. We can assume $s\in m_i$, otherwise $s$ would be a unit and $\frac{x}{s}\in R_i$. Notice that the rings $R_i[m_i s^{-1}]$ and $R_i[m_i \frac{x}{s}]=R_i[m_iz]$ are isomorphic and will be identified with each other. We define $R_{i+1}$ as the monoidal transform $R_i[m_iz]_q$ where $q$ is any prime ideal of $R_i[m_iz]$ that contains $m_i$. Define $p_{i+1}$ as a minimal prime associated to $p_i R_{i+1}$. 

[The following part is not final yet. Determine where we work, common supersets for the $R_i$, $(R_i)_{p_i}$ or $R_i/p_i$.]

By abuse of notation call $p_{i+1}$ the prime ideal of $R_i[m_iz]$ that induces $p_{i+1}$ in $R_{i+1}$. Then $(R_{i+1})_{p_{i+1}}=R_i[m_iz]_{p_{i+1}}$. Now show that $(R_i)_{p_i}=R_i[m_iz]_{p_{i+1}}$ [...].

By induction, $R_i/p_i\subseteq\overline{S}$ is a finite extension of $S$. We have that $R_{i+1}/p_{i+1}\subseteq R_{i+1}/R_{i+1}p_i=(R_i/p_i)[m_iz]_q\subseteq\overline{R_i/p_i}\subseteq\overline{S}$. The inclusion $R_i/p_i\subsetneq R_{i+1}/p_{i+1}$ is strict since $z\in R_{i+1}/p_{i+1}$.

Since $\overline{S}$ is a finite extension of $S$, this procedure has to terminate in a finite number of steps. We end up with a regular local ring $R'=R_n$ with prime ideal $p'=p_n$ such that $R'/p'$ is regular and $R'_{p'}=R_p$.

Let $J_i$ be ideals in $R_i$ such that $J_0=J$ and $J_{i+1}$ is the weak transform of $J_i$ under the blowup $R_i\to R_{i+1}$. Call $J'=J_n$. By [either cite or prove], $\nu(J')\leq\nu(J)$. Further, by Prop.÷ \ref{nu_reg}, $\nu(J')\leq\nu(J'R'_{p'})$. Combining these results yields $\nu(J)\leq \nu(J'R'_{p'})=\nu(JR_p)$.
\end{proof}

It is now relatively easy to generalize this result to arbitrary regular local rings $R$.

\begin{prop} \label{nu_wo_compl}
Let $R$ be any regular local ring and $\dim(R/p)=1$. Then $\nu(JR_p)\leq\nu(J)$.
\end{prop}
\begin{proof}
Let $\wh{R}$ denote the completion of $R$ and $\wh{J}=J\wh{R}$. Let $q$ be a minimal prime associated to $p\wh{R}$. We will now show that there is a canonical local homomorphism $R_p\to \wh{R}_q$. We know that $q\cap R$ is a prime ideal lying between $p$ and $m$. Since $\dim(R/p)=1$, either $q\cap R=p$ or $q\cap R=m$. But $q\cap R=m$ is not possible since $q$ is non-embedded [Details to clarify]. Thus, $q\cap R=p$. It follows that $(\wh{R}_q\cap(R\setminus p))\subseteq (\wh{R}_q)^*$. By the universal property of the localization $R\to R_p$, there is a unique homomorphism $R_p\to \wh{R}_q$, mapping $pR_p$ into $q\wh{R}_q$. Using this homomorphism, we see that $JR_p\subseteq p^nR_p$ implies  $J\wh{R}_q\subseteq q^n\wh{R}_q$ for all $n\geq0$. Since $J\wh{R}_q=\wh{J}\wh{R}_q$, this implies that $\nu(JR_p)\leq\nu(\wh{J}\wh{R}_q)$.

On the other hand, $\nu(\wh{J})=\nu(J)$ since $\nu(\wh{J})\geq n\iff J\wh{R}\subseteq m^n\wh{R} \iff \forall k\geq0: J/m^k\subseteq m^n/m^k \iff J\subseteq m^n \iff \nu(J)\geq n$ for all $n\geq0$.

By Prop.÷ \ref{nu_compl}, we know that $\nu(\wh{J}\wh{R}_q)\leq\nu(\wh{J})$. Putting the three inequalities together, we see that $\nu(JR_p)\leq\nu(J)$.
\end{proof}


\begin{prop} \label{nu_general}
Let $R$ be any regular local ring and $p$ any prime ideal. Then $\nu(JR_p)\leq\nu(J)$.
\end{prop}
\begin{proof}
Since $R$ is regular, there is a chain of prime ideals $p=p_0\subseteq\ldots\subseteq p_n=m$ with $\dim(R_{p_i}/p_{i-1}R_{p_i})=1$ for $1\leq i\leq n$. By Prop.÷ \ref{nu_wo_compl}, $\nu(JR_{p_{i-1}})\leq\nu(JR_{p_i})$ for $1\leq i\leq n$. This now implies that $\nu(JR_p)=\nu(JR_{p_0})\leq\nu(JR_{p_k})=\nu(J)$.
\end{proof}


\begin{proof}[Proof of Theorem \ref{order_under_localization}]
Let $R=\mathcal{O}_{W,a}$ denote the local ring of $W$ in $a$. Let $p$ be the prime ideal in $\mathcal{O}_{W,a}$ defining $b$. Then $\mathcal{O}_{W,b}=R_p$. Let $J$ be the ideal that locally defines $X$ in $\mathcal{O}_{W,a}$. Using Prop.÷ \ref{nu_general}, we see that $\ord_b X=\nu(JR_p)\leq\nu(J)=\ord_a X$.\end{proof}


\recognize


\begin{prop} An upper semicontinuous local invariant $\invv(X):X\too \Gamma$ with values in a totally well ordered set $\Gamma$ induces, up to refinement, a stratification of $X$ with strata $X_c=\{a\in X,\, \invv_a(X)=c\}$, for $c\in \Gamma$.\end{prop}


\begin{proof} {} For a given value $c\in\Gamma$, let $S=\{a\in X,\, \invv_a(X)\geq c\}$ and $T=\{a\in X,\, \invv_a(X)= c\}$. The set $S$ is a closed subset of $X$. If $c$ is a maximal value of the invariant on $X$, the set $T$ equals $S$ and is thus a closed stratum. If $c$ is not maximal, let $c'>c$ be an element of $\Gamma$ which is minimal with $c'>c$. Such elements exist since $\Gamma$ is well-ordered. As $\Gamma$ is totally ordered, $c'$ is unique. Therefore 
\[
S'=\{a\in X,\, \invv_a(X)\geq c'\}=\{a\in X,\, \invv_a(X)> c\}= S\setminus T 
\]
is closed in $X$. Therefore $T$ is open in $S$ and hence locally closed in $X$. As $S$ is closed in $X$ and $T\subset S$, the closure $\overline T$ is contained in $S$.  It follows that the boundary $\overline T\setminus T$ is contained in $S'$ and closed in $X$. \comment{[Still to be worked out]} Refine the strata by replacing $S'$ by $\overline T\setminus T$ and $S'\setminus \overline T$ to get a stratification.\end{proof}


\bigskip
\noindent\textit{\examples}


\begin{eg} \label{8.23} The singular locus $S=\Sing(X)$ of a variety is closed and properly contained in $X$. Let $X_1=\Reg(X)$ be the set of regular points of $X$. It is open and dense in $X$.  Repeat the procedure with $X\setminus X_1 =\Sing(X)$. By noetherianity, the process eventually stops, yielding a stratification of $X$ in regular strata. The strata of locally minimal dimension are closed and non-singular. The frontier condition holds, since the regular points of a variety are dense in the variety, and hence $\overline X_1= X =X_1\, \dot\cup\, \Sing(X)$. \end{eg}


\begin{eg} \label{8.24} There exists a threefold $X$ whose singular locus $\Sing(X)$ consists of two components, a non-singular surface and a singular curve meeting the surface at a singular point of the curve. The stratification by iterated singular loci as in the preceding example satisfies the frontier condition.\end{eg}


\begin{eg} \label{8.25}$\hint$ Give an example of a variety whose stratification by the iterated singular loci has four different types of strata. \end{eg}


\begin{eg} \label{8.26} Find interesting stratifications for three three-folds.\end{eg}


\begin{eg} \label{8.27} The stratification of an upper semicontinuous invariant need not be finite. Take for $\Gamma$ the set $X$ underlying a variety $X$ with the trivial (partial) ordering $a\leq b$ if and only if $a=b$. \end{eg}


\begin{eg} \label{8.28} Let $X$ be the non-reduced scheme defined by $xy^2=0$ in $\AAA^2$. The order of $X$ at points of the $y$-axis outside $0$ is $1$, at points of the $x$-axis outside $0$ it is $2$, and at the origin it is $3$. The local embedding dimension equals $1$ at points of the $y$-axis outside $0$, and $2$ at all points of the $x$-axis.\end{eg}


\begin{eg} \label{8.29} Determine the stratification given by the order for the following varieties. If the smallest stratum is regular, blow it up and determine the stratification of the strict transform. Produce pictures and describe the geometric changes. 

(a) Cross: $xyz=0$,

(b) Whitney umbrella: $x^2-yz^2=0$,

(c) Kolibri: $x^3+x^2z^2-y^2=0$,

(d) Xano: $x^4+z^3-yz^2=0$,

(e) Cusp \& Plane: $(y^2-x^3)z=0$. \end{eg}


\begin{eg} \label{8.30}$\hint$ Consider the order $\ord_ZI$ of an ideal $I$ in $\K[x_1,\ldots,x_n]$ along a closed subvariety $Z$ of $\AAA^n$, defined as the order of $I$ in the localization of $\K[x_1,\ldots,x_n]_J$ of $K[x_1,\ldots,x_n]$ at the ideal $J$ defining $Z$ in $\AAA^n$. Express this in terms of the symbolic powers $J^{(k)}= J^k\cdot \K[x_1,\ldots,x_n]_J\cap \K[x_1,\ldots,x_n]$ of $J$. Give an example to show that $\ord_ZI$ need not coincide with the maximal power $k$ such that $I\subset J^k$. If $J$ defines a complete intersection, the ordinary powers $J^k$ and the symbolic powers $J^{(k)}$ coincide. This holds in particular when $Z$ is a coordinate subspace of $\AAA^n$  \cite{Zariski_Samuel} IV, \S12, \cite{Hochster_73},\cite{Pellikaan}. \comment{[does non-singular suffice?]}\end{eg}


\begin{eg} \label{8.31}$\hint$\footnote {{} This example is due to Macaulay and was kindly communicated by M.÷ Hochster.} \comment{[Symbolic powers do not equal ordinary powers]}
The ideal $I=(y^2-xz, yz-x^3, z^2-x^2y)$ of $\K[x,y,z]$ has symbolic square $I^{(2)}$ which strictly contains $I^2$. \comment{[Hochster, Comm.÷ Alg.÷ I, p.÷ 111]}\end{eg}


\begin{eg} \label{8.32}$\hint$ The variety defined by $I=(y^2-xz, yz-x^3, z^2-x^2y)$ in $\mathbb{A}^3$ has parametrization $t\mapsto(t^3,t^4,t^5)$. Let $f=x^5+xy^3+z^3-3x^2yz$. For all maximal ideals $\mm$ which contain $I$, $f\in \mm^2$ and thus, $\ord_\mm f\geq 2$. But $f\not\in I^2$. Since $x f\in I^2$ and $x$ does not belong to $I$, it follows that $f\in I^2R_I$ and thus $\ord_I f\geq2$, in fact, $\ord_I f=2$.\end{eg}


\begin{eg} \label{8.33} The order of an ideal depends on the embedding of $X$ in $W$. If $X$ is not minimally embedded locally at $a$, the order of $X$ at $a$ is $1$ and not significant for measuring the complexity of the singularity of $X$ at $a$.
\end{eg}


\begin{eg} \label{8.34} Associate to a stratification of a variety the so called \textit{Hasse diagram}, i.e., the directed graph whose nodes and edges correspond to strata, respectively to the adjacency of strata. Determine the Hasse diagram for the surface $X$ given in $\AAA^4$ as the cartesian product of the cusp $C:x^2=y^3$ with the node $D:x^2=y^2+y^3$. Then project $X$ to  $\AAA^3$ by means of $\AAA^4\too \AAA^3, (x,y,z,w)\mapsto (x,y+z,w)$ and compute the Hasse diagram of the image $Y$ of $X$ under this projection.\end{eg}


\begin{eg} \label{8.35} Show that the order of a hypersurface, the dimension and the Hilbert-Samuel function of a variety, the embedding-dimension of a variety and the local number of irreducible components are invariant under local formal isomorphisms, and determine whether they are upper or lower semicontinuous. How does each of these invariants behave under localization and completion?\end{eg}


\begin{eg} \label{8.36} Take $\invv_a(X)=\dim_a(X)$, the dimension of $X$ at $a$. It is constant on irreducible varieties, and upper semicontinuous on arbitrary ones, because at an intersection point of several components, $\dim_a(X)$ is defined as the maximum of the dimensions of the components. \end{eg}


\begin{eg} \label{8.37} Take $\invv_a(X)={}$ the number of irreducible components of $X$ passing through $a$. If the components carry multiplicities as e.g. a divisor, one may alternatively take the sum of the multiplicities of the components passing through $a$. Both options produce an upper semicontinuous local invariant.\end{eg}


\begin{eg} \label{8.38} Take $\invv_a(X)=\textrm{embdim}_a(X)=\dim \mathrm{T}_aX$, the local embedding dimension of $X$ at $a$. It is upper semicontinuous. At regular points, it equals the dimension of $X$ at $a$, at singular points it exceeds this dimension.\end{eg}


\begin{eg} \label{8.39}$\hint$ Consider for a given coordinate system  $x,y_1,\ldots, y_n$ on $\AAA^{n+1}$ a polynomial of order $c$ at $0$ of the form  
\[
g(x,y_1,\ldots, y_n)= x^c+\sum_{i=0}^{c-1} g_i(y)\cdot x^i.
\] 
Express the order and the top locus of $g$ nearby $0$ in terms of the orders of the coefficients $g_i$.\end{eg}


\begin{eg} \label{8.40} Take $\invv_a(X)=\textrm{HS}_a(X)$, the Hilbert-Samuel function of $X$ at $a$. The lexicographic order on integer sequences defines a well-ordering on $\Gamma=\{\gamma:\NNN\too\NNN\}$. Find two varieties $X$ and $Y$ with points $a$ and $b$ where $\textrm{HS}_a(X)$ and $\textrm{HS}_b(Y)$ only differ from the fourth entry on.
\end{eg}


\begin{eg} \label{8.41} Take $\invv_a(X)=\nu_a^*(X)$ the increasingly ordered sequence of the orders of a minimal Macaulay basis of the ideal $I$ defining $X$ in $W$ at $a$ \cite{Hironaka_Annals} III.1, def.÷ 1 and Lemma 1, p.÷ 205. It is upper semicontinuous  but does not behave well under specialization \cite{Hironaka_Annals} III.3, Thm.÷ 2, p.÷ 220 and the remark after Cor.÷ 2, p.÷ 220, see also \cite{Ha_Obstacles}, ex.÷ 12. \end{eg}


\begin{eg} \label{8.42} Let a monomial order $<_\varepsilon$ on $\NNN^n$ be given, i.e., a total ordering with minimal element $0$ which is compatible with addition in $\NNN^n$. The \textit{initial ideal} $\inin(I)$ of an ideal $I$ of $\K[[x_1,\ldots,x_n]]$ with respect to $<_\varepsilon$ is the ideal generated by all \textit{initial monomials} of elements $f$ of $I$, i.e., the monomials with minimal exponent with respect to $<_\varepsilon$ in the series expansion of $f$. The initial monomial of $0$ is $0$.

The initial ideal is a monomial ideal in $\K[[x_1,\ldots,x_n]]$ and depends on the choice of coordinates. If the monomial order $<_\varepsilon$ is compatible with degree, $\inin(I)$ determines the Hilbert-Samuel function $\textrm{HS}_a(I)$ of $I$ \cite{Ha_Power_Series}.

Order monomial ideals totally by comparing their increasingly ordered unique minimal monomial generator system lexicographically, where any two monomial generators are compared with respect to $<_\varepsilon$. If two monomial ideals have generator systems of different length, complete the sequences of their exponents by a symbol $\infty$ so as to be able to compare them. This defines a well-order on the set of monomial ideals. 

Take for $\invv_a(X)$ the minimum $\min(I)$ or the maximum $\max(I)$ of the initial ideal of the ideal $I$ of $X$ at $a$, the minimum and maximum being taken over all choices of local coordinates, say regular parameter systems of $\K[[x_1,\ldots,x_n]]$. Both exist and  define local invariants which are upper semicontinuous with respect to localization \cite{Ha_Power_Series} Thms.÷ 3 and 8. 

(a) The minimal initial ideal ${\rm min}(I)={\rm min}_{x}\{{\rm in}(I)\}$ over all choices of regular parameter systems of $\K[[x_1,\ldots,x_n]]$ is achieved for almost all regular parameter systems. \comment{[For this, the set of regular parameter systems has to be equipped with a suitable topology.]}

(b)$\challenge$ Let $I$ be an ideal in $\K[x_1,\ldots,x_n]$. For a point $a=(a_1,\ldots,a_n)\in\AAA^n$  denote by $I_a$ the induced ideal in $\K[[x_1-a_n,\ldots,x_n-a_n]]$. The minimal initial ideal ${\rm min} (I_a)$ is upper semicontinuous when the point $a$ varies.

(c) Compare the induced stratification of $\AAA^n$ with the stratification by the Hilbert-Samuel function of $I$.

(d)  Let $Z$ be a regular center inside a stratum of the stratification induced by the initial ideal ${\rm in}(I)$, and consider the induced blowup $\wt\AAA^n\to \AAA^n$ along $Z$. Let $a\in Z$ and $a'\in W'$ be a point above $a$. Assume that the monomial order $<_\varepsilon$ is compatible with degree. Show that ${\rm min}_a(I)$ and ${\rm max}_a(I)$ do not increase when passing to the strict transform of $I_a$ at $a'$ \cite{Ha_Power_Series} Thm.÷ 6.\end{eg}


\begin{eg} \label{8.43}$\hint$ Let $(W',a')\too (W,a)$ be the composition of two monomial point blowups of $W=\AAA^2$ with respect to coordinates $y,z$, defined as follows. The first is the blowup of $\AAA^2$ with center $0$ considered at the origin of the $y$-chart, the second has as center the origin of the $y$-chart and is considered at the origin of the $z$-chart. Show that the order of the strict transform $g'(y,z)$ at $a'$ of any non zero polynomial $g(y,z)$ in $W$ is at most the half of the order of $g(y,z)$ at $a$.\end{eg}




\begin{eg} \label{8.44} (B.÷ Schober) Let $\K$ be a non perfect field of characteristic $3$, let $t\in \K\setminus \K^2$ be an element which is not a square. Stratify the hypersurface $X: x^2+y(z^2+t w^2)=0$ in $\AAA^4$ according to its singularities.  Show that this stratification does not provide suitable centers for a resolution.\end{eg}


\begin{eg} \label{8.45} Let $I = (x^2 + y^{17})$ be the ideal defining an affine plane curve singularity $X$ with singular locus the origin of $W=\AAA^2$. The order of $X$ at $0$ is $2$. The blowup $\pi: W'\too W$ of $W$ at the origin with exceptional divisor $E$ is covered by two affine charts, the $x$- and the $y$-chart. The total and strict transform of $I$ in the $x$-chart are as follows:
\begin{eqnarray*}
I^* &=& (x^2 + x^{17}y^{17}),\\
I^s &=& (1 + x^{15}y^{17}).
\end{eqnarray*}
At the origin of this chart, the order of $I^s$ has dropped to zero, so the strict transform $X^s$ of $X$ does not contain this point. Therefore it suffices to consider the complement of this point in $E$, which lies entirely in the $y$-chart. There, one obtains the following transforms:
\begin{eqnarray*}
I^* &=& (x^2y^2 + y^{17}),\\
I^s &=& (x^2 + y^{15}).
\end{eqnarray*}
The origin $a'$ of the $y$-chart is the only singular point of $X^s$. The order of $X^s$ at $a'$ has remained constant equal to $2$. Find a local invariant of $X$ which has improved at $a'$. Make sure that it does not depend on any choices. \end{eg}


\begin{eg} \label{8.46} The ideal $I = (x^2 + y^{16})$ has in the $y$-chart of the point blowup of $\AAA^2$ at $0$ the strict transform $I^s = (x^2 + y^{14})$. If the ground field has characteristic $2$, the $y$-exponents $16$ and $14$ are irrelevant to measure an improvement of the singularity because the coordinate change $x\mapsto x + y^8$ transforms $I$ into $(x^2)$.\end{eg}


\begin{eg} \label{8.47} The ideal $I =(x^2 + 2xy^7 + y^{14} + y^{17})$ has in the $y$-chart the strict transform $I^s = (x^2 + 2xy^6 + y^{12} + y^{15})$. Here the drop of the $y$-exponent of the last monomial from $17$ to $15$ is significant, whereas the terms $ 2xy^7 + y^{14}$ can be eliminated in any characteristic by the coordinate change  $x\mapsto x + y^7$. \end{eg}


\begin{eg} \label{8.48} The ideal $I = (x^2 + xy^9)=(x)(x+y^9)$ defines the union of two non-singular curves in $\AAA^2$ which have a common tangent line at their intersection point $0$. The strict transform is $I^s = (x^2 + xy^8)$. The degree of tangency, viz the intersection multiplicity, has decreased. \end{eg}


\begin{eg} \label{8.49} Let $X$ and $Y$ be two non-singular curves in $\AAA^2$, meeting at one point $a$. Show that there exists a sequence of point blowups which separates the two curves, i.e., so that the strict transforms of $X$ and $Y$ do not intersect.\end{eg}


\begin{eg} \label{8.50} Take $I = (x^2 + g(y))$ where $g$ is a polynomial in $y$ with order  at least $3$ at $0$. The strict transform under point blowup in the $y$-chart is $I^s = x^2 + y^{-2}g(y)$, with order at $0$ equal to $2$ again. This suggests to take the order of $g$ as a secondary invariant. In characteristic $2$ it may depend on the choice of coordinates.\end{eg}


\begin{eg} \label{8.51} Take $I = (x^2 + xg(y) + h(y))$ where $g$ and $h$ are polynomials in $y$ of order at least $1$, respectively $2$, at $0$. The order of $I$ at $0$ is $2$. The strict transform equals $I^s = (x^2 + xy^{-1}g(y) + y^{-2}h(y))$ of order $2$ at $0$. Here it is less clear how to detect a secondary invariant which represents an improvement.\end{eg}


\begin{eg} \label{8.52} Take $I = (x^2 + y^3z^3)$ in $\AAA^3$, and apply the blowup of $\AAA^3$ in the origin. The strict transform of $I$ in the $y$-chart equals $I^s=(x^2+y^4z^3)$ and the singularity seems to have gotten worse.  \end{eg}


\begin{eg} \label{8.53}$\hint$ Let $X$ be a surface in three-space, and $S$ its top locus. Assume that $S$ is singular at $a$, and let $X'$ be the blowup of $X$ in $a$. Determine the top locus of $X'$. \end{eg}


\section{Lecture IX: Hypersurfaces of Maximal Contact}


\begin{prop} \label{181} (Zariski) Let $X$ be a subvariety of a regular variety $W$, defined over a field of arbitrary characteristic. Let $W'\too W$ be the blowup of $W$ along a regular center $Z$ contained in the top locus of $X$, and let $a$ be a point of $Z$.
There exists, in a neighbourhood $U$ of $a$, a regular closed hypersurface $V$ of $U$ whose strict transform $V^s$ in $W'$ contains all points $a'$ of $W'$ lying above $a$ where the order of the strict transform $X^s$ of $X$ in $W'$ has remained constant equal to the order of $X$ along $Z$. \end{prop}


\begin{proof} Choose local coordinates $x_1,\ldots,x_n$ of $W$ at $a$. The associated graded ring of $\calo_{W,a}$ can be identified with $\K[x_1,\ldots,x_n]$.  Let $\inin(I)\subset\K[x_1,\ldots,x_n]$ denote the ideal of initial forms of elements of $I$ at $a$. Apply a linear coordinate change so that generators of $\inin(I)$ are expressed with the minimal number of variables, say $x_1,\ldots, x_k$, for some $k\leq n$. Choose any $1\leq i\leq k$ and define $V$ in $W$ at $a$ by $x_i=0$. It follows that the local top locus of $X$ at $a$ is contained in $V$. Hence $Z\subset V$, locally at $a$. Let $a'$ be a point of $W'$ above $a$ where the order of $X$ has remained constant. By Prop.÷ \ref{111} the local blowup $(W',a')\too (W,a)$ can be made monomial by a suitable coordinate change. The assertion then follows by computation, cf.÷ ex.÷ \ref{9.8} and \cite{Zariski_1944}. \comment{[insert complete proof]}\end{proof}




\begin{defn} Let $X$ be a subvariety of a regular variety $W$, and let $a$ be a point of $W$. A \textit{hypersurface of maximal contact for $X$ at $a$} is a regular closed hypersurface $V$ of an open neighborhood $U$ of $a$ in $W$ such that

(1) $V$ contains the local top locus $S$ of $X$ at $a$, i.e., the points of $U$ where the order of $X$ equals the order of $X$ at $a$.

(2) The strict transform $V^s$ of $V$ under any blowup of $U$ along a regular center $Z$ contained in $S$ contains all points $a'$ above $a$ where the order of $X^s$ has remained constant equal to the order of $X$ at $a$.

(3) Property (2) is preserved in any sequence of blowups with regular centers contained in the successive top loci of the strict transforms of $X$ along which the order of $X$ has remained constant. \comment{[Said differently, $V^s$ is a directrix of $X^s$ at all equiconstant points $a'$ above $a$, and this repeats.]}\end{defn}


\hskip 3cm $a'\in E \cap V^s \subset U'\subset W'$\\

\hskip 3cm $\downarrow \hskip .6cm\downarrow\hskip 2.25cm \downarrow \pi$\\

\hskip 3cm  $a\in \, \,Z\, \subset \, V \subset U\subset W$


\begin{defn} Assume that the characteristic of the ground field is zero. Let $X$ be a subvariety of a regular variety $W$ defined locally at a point $a$ of $W$ by the ideal $I$. Let $o$ be the order of $X$ at $a$. An \textit{osculating hypersurface for $X$ at $a$} is a regular closed hypersurface $V$ of a neighbourhood $U$ of $a$ in $W$ defined by a derivative of order $o-1$ of an element $f$ of order $o$ of $I$ \cite{EH}.\end{defn}


\begin{rem} \label{9.4} The element $f$ has necessarily order $o$ at $a$, and its $(o-1)$-st derivative   has order $1$ at $a$, so that it defines a regular hypersurface at $a$. The concept is due to Abhyankar and Zariski \cite{Abhyankar_Zariski}. Abhyankar called the local isomorphism constructing an osculating hypersurface $V$ from a given regular hypersurface $H$ of $W$ \textit{ Tschirnhaus transformation}. \comment{[Add reference to Tschirnhaus]} If $H$ is given by $x_n=0$ for some coordinates $x_1,\ldots, x_n$, this transformation eliminates from $f$ all monomials whose $x_n$-exponent is $o-1$.  The existence of osculating hypersurfaces was exploited systematically by Hironaka in his proof of characteristic zero resolution \cite{Hironaka_Annals}. \end{rem}




 \comment{In characteristic zero, osculating hypersurfaces are directrices.} For each point $a$ in $X$, osculating hypersurfaces contain locally at $a$ the local top locus $S=\ttop_a(X)$, and their strict transform contain the equiconstant points above $a$.


\begin{prop} (Abhyankar, Hironaka) \label{9.5} Let $X$ be a subvariety of a regular variety $W$, and let $a$ be a point of $W$. For ground fields of characteristic zero there exist, locally at $a$ in $W$, hypersurfaces of maximal contact for $X$. Any osculating hypersurface $V$ at $a$ has maximal contact with $X$ at $a$.\end{prop}

\begin{proof} \cite{EH}.\end{proof}


\begin{rem} The assertion of the proposition does not hold over fields of positive characteristic: R.÷ Nara\-simhan, a student of Abhyankar, gave an example of a hypersurface $X$ in $\AAA^4$ over a field of characteristic $2$ whose top locus is not contained at $0$ in any regular local hypersurface, and Kawanoue describes a whole family of such varieties \cite{Narasimhan, Kawanoue}, \cite{Ha_Obstacles}, ex.÷ 8. See also ex.÷ \ref{775} below. In Narasimhan's example, there is a sequence of point blowups for which there is no regular local hypersurface $V$ of $\AAA^4$ at $0$ whose strict transforms contain all points where the transforms of $X$ have order $2$ as at the beginning \cite{Ha_BAMS_1} II.14, ex.÷ 2, p.÷ 388.\end{rem}


\begin{rem} The existence of hypersurfaces of maximal contact in zero characteristic suggests to associate to $X$ locally at a point $a$ a variety $Y$ defined by an ideal $J$ in the hypersurface $V$ and to observe the behaviour of $X$ under blowup by means of the behaviour of $Y$ under the induced blowup: the transform of $Y$ under the blowup of $V$ along a center $Z$ of $W$ locally contained in $V$ should equal the variety $Y'$ which is associated in a similar manner as $Y$ to $X$ to the strict transform $X^s$ of $X$ in $V^s$ at points $a'$ above $a$ where the order of $X^s$ has remained constant. The variety $Y$ or the ideal $J$ and their transforms may then help to measure the improvement of $X^s$ at $a'$ by looking at their respective orders. This is precisely the way how the proof of resolution in zero characteristic proceeds. The reasoning is also known as \textit{descent in dimension}. The main problem in this approach is to define properly the variety $Y$, respectively the ideal $J$, and to show that 
the local construction is independent of the choice of $V$ and patches to give a global resolution algorithm.  \end{rem}


\hskip 3cm $a'\in E\subset W'\hskip .2cm\leadsto \hskip .2cmV'\supset Y'$\\

\hskip 3cm $\downarrow \hskip .6cm \downarrow \hskip .6cm\downarrow \pi \hskip .8cm \downarrow \hskip .6cm \downarrow \pi_{\vert Y'}$\\

\hskip 3cm  $a\in \, \,Z \subset W\hskip .3cm\leadsto \hskip .2cm V\, \,\supset Y$

\goodbreak


\bigskip
\noindent\textit{\examples}


\begin{eg} \label{9.8}$\hint$ Let $\pi:(W',a')\too (W,a)$ be a local blowup and let $x_1,\ldots, x_n$ be local coordinates on $W$ at $a$ such that $\pi$ is monomial. Assume that $x_1$ appears in the initial form of an element $f\in\calo_{W,a}$, and let $V\subset W$ be the local hypersurface at $a$ defined by $x_1=0$. If the order of the strict transform $f^s$ of $f$ at $a'$ has remained constant equal to the order of $f$ at $a$, the point $a'$ belongs to the strict transform $V^s$ of $V$.\end{eg}


\begin{eg} \label{9.9} Let the characteristic of the ground field be different from $3$. Apply the second order differential operator $\partial = \frac{\partial^2}{\partial x^2}$ to $f = x^3 + x^2yz + z^5$ so that $\partial f = 6x + 2yz$. This defines a hypersurface of maximal contact for $f$ at $0$. Replacing in $f$ the variable $x$ by $x - {1\over 3}yz$ gives
\[
g = x^3 - \frac{1}{3}xy^2z^2  + \frac{2}{27}y^3z^3 + z^5.
\]
The term of degree $2$ in $x$ has been eliminated, and $x=0$ defines an osculating hypersurface for $g$ at $0$.\end{eg}


\begin{eg} \label{9.10}$\hint$ Assume that the characteristic is $0$. Let $X\subset \AAA^n$ be a hypersurface defined locally at the origin by a polynomial $f=x_n^d+\sum_{i=0}^{d-1}a_i(y)x_n^i$ where $y=(x_1,\ldots,x_{n-1})$ and $\ord_0 a_i(y)\geq d-i$. Make the change of coordinates $x_n\mapsto x_n-\frac{1}{d}a_{d-1}(y)$. Show that after this change, the hypersurface defined by $x_n=0$ has maximal contact with $X$ at the origin. What prevents this technique from working in positive characteristic?\end{eg}


\begin{eg} \label{9.11}$\hint$ Consider the hypersurface $X\subset\AAA^3$ given by the equation $x^2 y+x y^2-x^2 z+y^2 z-x z^2-y z^2=0$. Show that the hypersurface $V$ given by $x=0$ does not have maximal contact with $X$ at $0$. In particular, consider the blowup of $\AAA^3$ in the origin. Find a point $a'$ on the exceptional divisor that lies in the $x$-chart of the blowup such that the strict transform of $X$ has order $3$ at $a'$. Then show that the strict transform of $V$ does not contain this point.\end{eg}


\begin{eg} \label{9.12} Hypersurfaces of maximal contact are only defined locally. They need not patch to give a globally defined hypersurface of maximal contact on $W$. Find an example for this.\end{eg}


\begin{eg} \label{9.13} Consider $f=x^4+y^4+z^6$, $g=x^4+y^4+z^{10}$ and $h= xy+z^{10}$ under point blowup. Determine, according to the characteristic of the ground field, the points where the orders of $f$, $g$ and $h$ have remained constant.\end{eg}


\begin{eg} \label{9.14}$\hint$ Let $f=x^c+g(y_1,\ldots, y_m)\in K[[x,y_1,\ldots, y_m]]$ be a formal power series with $g$ a series of order $\geq c$ at $0$. Show that there exists in any characteristic a formal coordinate change maximizing the order of $g$.\end{eg}


\begin{eg} \label{9.15} Let $f=x^c+g(y_1,\ldots, y_m)\in K[x,y_1,\ldots, y_m]$ be a polynomial with $g$ a polynomial of order $\geq c$ at $0$. Does there exist a local coordinate change in $\AAA^{1+m}$ at $0$ maximizing the order of $g$? \end{eg}


\begin{eg} \label{9.16} Let $f$ be a polynomial or power series in $n$ variables $x_1,\ldots,x_n$ of order $c$ at $0$. Assume that the ground field is infinite. There exists a linear coordinate change after which $f(0,\ldots, 0, x_n)$ has order $c$ at $0$. Such polynomials and series, called {\it $x_n$-regular of order} $c$ at $0$, appear in the Weierstrass preparation theorem, which was frequently used by Abhyankar in resolution arguments.\end{eg}


\begin{eg} \label{9.17}$\hint$ Let $(W',a')\too (W,a)$ be a composition of local blowups in regular centers such that $a'$ lies in the intersection of $n$ exceptional components where $n$ is the dimension of $W$ at $a$. Let $f\in \calo_{W,a}$ and assume that the characteristic is zero. Show that the order of $f$ has dropped between $a$ and $a'$.\end{eg}


\begin{eg} \label{9.18}$\challenge$ Show the same in positive characteristic.\end{eg}


\section{Lecture X: Coefficient Ideals}


\begin{defn} Let $I$ be an ideal in a regular variety $W$, let $a$ be a point of $W$ with open neighbourhood $U$, and let $V$ be a regular closed hypersurface of $U$ containing $a$. Let $x_1,\ldots, x_n$ be coordinates on $U$ such that $V$ is defined in $U$ by $x_n=0$. The restrictions of $x_1,\ldots, x_{n-1}$ to $V$ form coordinates on $V$ and will be abbreviated by $x'$. For $f\in \calo_U$, denote by $\sum a_{f,i}(x')\cdot x_n^i$ the expansion of $f$ with respect to $x_n$, with coefficients $a_{f,i}=a_{f,i}(x')\in\calo_{V}$. The \textit{coefficient ideal of $I$ in $V$ at $a$} is the ideal $J_V(I)$ on $V$ defined by 
\[
J_V(I)=\sum_{i=0}^{o-1}(a_{f,i}, f\in I)^{o!\over o-i},
\]
where $o$ denotes the order of $I$ at $a$.\end{defn}


\begin{rem} The coefficient ideal is defined on whole $V$. It depends on the choice of the coordinates $x_1,\ldots, x_n$ on $U$, even so the notation only refers to $V$. Actually, $J_V(I)$ depends on the choice of a section $\calo_{V,a}\too \calo_{U,a}$ of the map $\calo_{U,a}\too \calo_{V,a}$ defined by restriction to $V$. The same definition applies to stalks of ideals in $W$ at points $a$, giving rise to an ideal, also denoted by $J_V(I)$, in the local ring $\calo_{V,a}$. The weights ${o!\over o-i}$ in the exponents are chosen so as to obtain a systematic behavior of the coefficient ideal under blowup, cf.÷ Prop.÷ \ref{884} below. The chosen algebraic definition of the coefficient ideal is modelled so as to commute with blowups \cite{EH}, but is less conceptual than definitions through differential operators proposed and used by Encinas-Villamayor, Bierstone-Milman,  W\l odarczyk, Kawanoue-Matsuki and Hironaka \cite{EV_Obergurgl, EV_Good_Points, BM_Canonical_Desing, Wlodarczyk, Kawanoue_Matsuki, Hironaka_Korea}.\end{rem}


\begin{prop} The passage to the coefficient ideal $J_V(I)$ of $I$ in $V$ commutes with taking germs along the local top locus $S=\ttop_a(I)\cap V$ of points of $V$ where the order of $I$ in $W$ is equal to the order of $I$ in $W$ at $a$: Let $x_1,\ldots x_n$ be coordinates of $W$ at $a$, defined on an open neighborhood $U$ of $a$, and let $V$ be closed and regular in $U$, defined by $x_n=0$. The stalks of $J_V(I)$ at points $b$ of $S$ inside $U$ coincide with the coefficient ideals of the stalks of $I$ at $b$.
\end{prop}


\begin{proof} Clear from the definition of coefficient ideals.\end{proof}


\begin{cor} In the above situation, for any fixed closed hypersurface $V$ in $U\subset W$ open, the order of $J_V(I)$ at points of $S\cap V$ is upper semicontinuous along $S$, locally at $a$. \end{cor}


\begin{rem} In general, $V$ need not contain, locally at $a$, the top locus of $I$ in $W$. This can, however, be achieved in zero characteristic by choosing for $V$ an osculating hypersurface, cf.÷ Prop.÷ \ref{414} below. In this case, the order of $J_V(I)$ at points of $S=\ttop_a(I)$ does not depend on the choice of the hypersurface, cf.÷ Prop.÷ \ref{669}. In arbitrary characteristic, a local hypersurface $V$ will be chosen separately at each point $b\in S$ in order to maximize the order of $J_V(I)$ at $b$, cf.÷ Prop.÷ \ref{575}. In this case,  the order of $J_V(I)$ at $b$ does not depend on the choice of $V$, and its upper semicontinuity as $b$ moves along $S$ holds again, but is more difficult to prove \cite{Ha_Power_Series}.\end{rem}


\begin{prop} \label{884} The passage to the coefficient ideal $J_V(I)$ of $I$ at $a$ commutes with blowup: Let $\pi: W'\too W$ be the blowup of $W$ along a regular center $Z$ contained locally at $a$ in $S=\ttop_a(I)$. Let $V$ be a local regular hypersurface of $W$ at $a$ containing $Z$ and such that $V^s$ contains all points $a'$ above $a$ where the order of the weak transform $I^\curlyvee$ has remained constant equal to the order of $I$ at $a$. For any such point $a'$, the coefficient ideal $J_{V^s}(I^\curlyvee)$ of $I^\curlyvee$ equals the controlled transform $J_{V}(I)^!= I_E^{-c}\cdot J_{V}(I)^*$ of $J_{V}(I)$ with respect to the control $c=o!$ with $o=\ord_a(I)$.
\end{prop}


\begin{proof} Write $V'$ for $V^s$, and let $h=0$ be a local equation of $E\cap V'$ in $V'$.  The weak transform $I^\curlyvee$ of $I$ is generated by the elements $f^\curlyvee=h^{-o}\cdot f^*$ for $f$ varying in $I$, where ${}^*$ denotes the total transform. The coefficients $a_{f,i}$ of the monomials $x_n^i$ of the expansion of an element $f$ of $I$ of order $o$ at $a$ in the coordinates $x_1,\ldots, x_n$ satisfy $a_{f^{\curlyvee},i}=h^{i-o} \cdot(a_{f,i})^{*}$. Then \comment{[align vertically]}
\[J_{V'} (I^\curlyvee) =J_{V'} (\sum_i a_{f^\curlyvee,i}\cdot x_n^i,\ f^\curlyvee\in I^\curlyvee)\]  
\[ =J_{V'} (\sum_i a_{h^{-o}\cdot f^*,i}\cdot x_n^i,\ f\in I)\]  

\[=J_{V'} (\sum_i h^{-o}\cdot (a_{f,i}\cdot x_n^{i})^*,\ f\in I)\] 
\[=\sum_{i<o} h^{-o!}\cdot (a_{f,i}^*,\ f\in I)^{o!/(o-i)}\] 
\[=h^{-o!}\cdot (\sum_{i<o} (a_{f,i},\ f\in I)^{o!/(o-i)})^*\] 
\[=h^{-o!}\cdot (J_VI)^*= (J_VI)^!.\]  
This proves the claim.\end{proof}


\begin{rem} The definitions of the coefficient ideal used in \cite{EV_Obergurgl, EV_Good_Points, BM_Canonical_Desing, Wlodarczyk, Kawanoue_Matsuki, Hironaka_Korea} produce a weaker commutativity property with respect to blowups, typically only for the radicals of the coefficient ideals.\end{rem}


\begin{rem} The order of the coefficient ideal $J_V(I)$ of $I$ is not directly suited as a secondary invariant when the order of $I$ remains constant since, by the proposition, the coefficient ideal passes under blowup to its controlled transform, and thus its order may increase. In order to get a practicable secondary invariant it is appropriate to decompose $J_V(I)$ and $(J_V(I))^!$ into products of two ideals: the first factor is a principal monomial ideal supported by the exceptional locus, the second, possibly singular factor, is an ideal called the \textit{residual factor}, and supposed to pass under blowup in the factorization to its weak transform. Choosing suitably the exceptional monomial factor it can be shown that such factorizations always exist \cite{EH}. In this situation the order of the residual factor does not increase under blowup by Prop.÷ \ref{899} and can thus serve as a secondary invariant whenever the order of the ideal $I$ remains constant under blowup.\end{rem}


\begin{prop} \label{414} Assume that the characteristic of the ground field is zero. Let $I$ have order $o$ at a point $a\in W$, and let $V$ be a regular hypersurface for $I$ at $a$, with coefficient ideal $J=J_V(I)$. The locus $\ttop_a(J, o!)$ of points of $V$ where $J$ has order $\geq o!$ coincides with $\ttop_a(I)$,
\[
\ttop_a(J, o!)=\ttop_a(I).\] \end{prop}


\begin{proof} {} Choose local coordinates $x_1,\ldots ,x_n$ in $W$ at $a$ so that $V$ is defined by $x_n=0$. Expand the elements $f$ of $I$ with respect to $x_n$ with coefficients $a_{f,i}\in\calo_{V,a}$, and choose representatives of them on a suitable neighbourhood of $a$. Let $b$ be a point in a sufficiently small neighborhood of $a$. Then, by the upper semicontinuity of the order, $b$ belongs to $\ttop_a(I)$ if and only if $\ord_b I\geq o$, which is equivalent to $\sum_{i<o}a_{f,i}\cdot x_n^i$ having order $\geq o$ at $b$ for all $f\in I$. This, in turn, holds if and only if $a_{f,i}$ has order $\geq o-i$ at $b$, say $a_{f,i}^{\frac{o!}{o-i}}$ has order $\geq o!$ at $b$. Hence $b\in \ttop_a(I)$ if and only if $b\in\ttop_a(J_V(I),o!).$  \end{proof}


\begin{cor} Assume that the characteristic of the ground field is zero. Let $a$ be a point in $W$ and set $S=\ttop_a(I)$. Let $U$ be a neighbourhood of $a$ on which there exists a closed regular hypersurface $V$ which is osculating for $I$ at all points of $S\cap U$. Let $J_V(I)$ be the coefficient ideal of $I$ in $V$. 

(a) The top locus $\ttop_a(J_V(I))$ of $J_V(I)$ on $V$ is contained in $\ttop_a(I)$. 

(b) The blowup of $U$ along a regular locally closed subvariety $Z$ of $\ttop_a(J_V(I))$ $Z$ commutes with the passage to the coefficient ideals of $I$ and $I^\curlyvee$ in $V$ and $V^s$.
\end{cor}


\begin{proof} Assertion (a) is immediate from the proposition, and (b) follows from Prop.÷ \ref{884}.\end{proof}
 

\begin{prop} \label{575} (Encinas-Hauser) Assume that the characteristic of the ground field is zero. The order of the coefficient ideal $J_V(I)$ of $I$ at $a$ with respect to an osculating hypersurface $V$ at $a$ attains the maximal value of the orders of the coefficient ideals over all local regular hypersurfaces. In particular, it is independent of the choice of $V$. \end{prop}


\begin{proof} \comment{[If the coefficient ideal is zero, we might have to work in the completion of $\calo_{W,a}$.]} Choose local coordinates $x_1,\ldots, x_n$ in $W$ at $a$ such that the appropriate derivative of the chosen element $f\in I$ is given by $x_n$. Let $o$ be the order of $f$ at $a$. The choice of coordinates implies that the expansion of $f$ with respect to $x_n$ has a monomial $x_n^o$ with coefficient $1$ and no monomial in $x_n$ of degree $o-1$. Any other local regular hypersurface $U$ is obtained from $V$ by a local isomorphism $\varphi$ of $W$ at $a$. Assume that the order of $J_V(\varphi^*(I))$ is larger than the order of $J_V(I)$. Let $g=\varphi^*(f)$. This signifies that the order of all coefficients $a_{g,i}^{o!\over o-i}$ is larger than the order of $J_V(I)$. Therefore $\varphi^*$ must eliminate from $f$ the terms of $a_{f,i}$ for which $a_{f,i}^{o!\over o-i}$ has order equal to the order of $J_V(I)$. But then $\varphi^*$ produces from $x_n^o$ a non-zero coefficient $a_{g,o-1}$ such 
that $a_{g,o-1}^{o!}$ has order equal to the order of $J_V(I)$, contradiction.
\end{proof}


\begin{rem} This result suggests to consider in positive characteristic as a substitute for hypersurfaces of maximal contact local regular hypersurfaces which maximize the order of the associated coefficient ideal. Such hypersurfaces are used in recent approaches to resolution of singularities in positive characteristic \cite{Hironaka_CMI,  Ha_BAMS_2, HW}, relying on the work of Moh on the behaviour of the coefficient ideal in this situation \cite{Moh}. 
\end{rem}


\begin{prop} \label{669} Let the characteristic of the ground field be arbitrary. The supremum in $\NNN\cup\{\infty\}$ of the orders of the coefficient ideals $J_V(I)$ of $I$ in local regular hypersurfaces $V$ in $W$ at $a$ is realized by a formal local regular hypersurface $V$ in $W$ at $a$ (i.e., $V$ is defined by an element of the complete local ring $\wt\calo_{W,a}$). If the supremum is finite, it can be realized by a local regular hypersurface $V$ in $W$ at $a$.\end{prop}


\begin{proof} If the supremum is finite, the existence of some $V$ realizing this value is obvious. If the supremum is infinite, one uses the completeness of $\wh\calo_{W,a}$ to construct $V$, see \cite{EH, Ha_Power_Series}. \end{proof}


\begin{defn} A formal local regular hypersurface $V$ realizing the supremum of the order of the coefficient ideal $J_V(I)$ of $I$ at $a$ is called a \textit{hypersurface of weak maximal contact of $I$ at $a$}. If the supremum is finite, it will always be assumed to be a local hypersurface. \end{defn}

\begin{prop}(Zariski) Let $V$ be a formal local regular hypersurface in $W$ at $a$ of weak maximal contact with $I$ at $a$. Let $\pi : W' \too W$ be the blowup of $W$ along a closed regular center $Z$ contained locally at $a$ in $S=\ttop_a(I)$. The points $a' \in W'$ above $a$ for which the order of the weak transform $I^\curlyvee$ of $I$ at $a'$ has not decreased are contained in the strict transform $V^s$ of $V$. 
\end{prop}


\begin{proof} By definition of weak maximal contact, the variable $x_n$ defining $V$ in $W$ at $a$ appears in the initial form of some element $f$ of $I$ of order $o=\ttop_a(I)$ at $a$, cf.÷ ex.÷ \ref{10.23} below. The argument then goes analogously to the proof of Prop.÷ \ref{181}. \end{proof}


\begin{rem} In characteristic zero and if $V$ has been chosen osculating at $a$, the hypersurface $V'$ is again osculating at points $a'$ above $a$ where $\ord_{a'}(I')=\ord_a(I)$, hence it has again weak maximal contact with $I'$ at such points $a'$. In positive characteristic this is no longer true, cf.÷ ex.÷ \ref{10.24}.\end{rem}


\bigskip
\noindent\textit{\examples}




\begin{eg} \label{10.17} Determine in all characteristics the points of the blowup $\wt\AAA^2$ of $\AAA^2$ at $0$ where the strict transform of $g= x^4+kx^2y^2 + y^4+3y^7 + 5y^8 + 7y^9$ under the blowup of $\AAA^2$ at $0$ has order $4$, for any $k\in \NNN$.\end{eg}


\begin{eg} \label{10.18} Determine for all characteristics the maximal order of the coefficient ideal of $I=(x^3 + 5y^3 + 3(x^2y^2+xy^4) + y^6 + 7y^7 +y^9+y^{10})$ in regular local hypersurfaces $V$ at $0$.\end{eg}


\begin{eg} \label{10.19} Same as before for $I = (xy + y^4+3y^7 + 5y^8 + 7y^9)$.\end{eg}


\begin{eg} \label{10.20}$\hint$ Compute the coefficient ideal of $I=(x^5 + x^2y^4 + y^k)$ in the hypersurfaces $x=0$, respectively $y=0$. According to the value of $k$ and the characteristic, which hypersurface is osculating or has weak maximal contact?\end{eg}


\begin{eg} \label{10.21}$\hint$ Consider $f=x^2+y^3z^3 + y^7+z^7$. Show that $V\subset \AAA^3$ defined by $x=0$ is a local hypersurface of weak maximal contact for $f$. Blow up $\AAA^3$ at the origin. How does the order of the coefficient ideal of $f$ in $V$ behave under these blowups at the points where the order of $f$ has remained constant? Factorize suitably the controlled transform of the coefficient ideal with respect to the exceptional factor and observe the behaviour of the order of the residual factor.\end{eg}


\begin{eg} \label{10.22}$\hint$ Assume that, for a given coordinate system $x_1,\ldots, x_n$ in $W$ at $a$, the hypersurface $V$ defined by $x_n=0$ is osculating for a polynomial $f$ and that the coefficient ideal of $f$ in $V$ is a principal monomial ideal. Show that there is a sequence of blowups in coordinate subspaces of the induced affine charts which eventually makes the order of $f$ drop.\end{eg}


\begin{eg} \label{10.23} The variable $x_n$ defining a hypersurface $V$ in $W$ at $a$ of weak maximal contact with an ideal $I$ appears in the initial form of some element $f$ of $I$ of order $o=\ord_a(I)$ at $a$.\end{eg}


\begin{eg} \label{10.24} In positive characteristic, a hypersurface of weak maximal contact for an ideal $I$ need not have again weak maximal contact after blowup with the weak transform $I^\curlyvee$ at points $a'$ where the order of $I^\curlyvee$ has remained constant.\end{eg}


\begin{eg} \label{10.25} Compute in the following situations the coefficient ideal of $I$ in $W$ at $a$ with respect to the given local coordinates $x$, $y$, $z$ in $\AAA^3$ and the hypersurface $V$. Determine in each case whether the order of the coefficient ideal is maximal. If not, find a coordinate change which maximizes it.  

(a) $a=0\in\AAA^1$, $x$, $V:x=0$, $I=(x)$ and $I=(x+x^2)$.

(b) $a=0\in\AAA^2$, $x,y$, $V:x=0$, $I=(x)$, $I=(x+y^2)$, $I=(y+x^2)$, $I=(xy)$.

(c) $a=(1,0,0)\in\AAA^3$, $x,y,z$, $V:y+z=0$, $I=(x^2)$, $I=(xy)$, $I=(x^3+z^3)$.

(d) $a=0\in\AAA^3$, $x,y,z$, $V:x=0$, $I=(xyz)$, $I=(x^2+y^3+z^5)$.

(e) $a=0\in\AAA^2$, $x,y$, $V:x=0$, $I=(x^2+y^4, y^4+x^2)$. 
\end{eg}


\begin{eg} \label{10.26} Blow up in each of the preceding examples the origin and determine the points of the exceptional divisor $E$ where the order of the weak transform $I^\curlyvee$ of $I$ has remained constant. Check at these points whether commutativity holds for the descent to the coefficient ideal and its controlled transform.\end{eg}


\begin{eg} \label{10.27}$\hint$ Show that the maximum of the order of the coefficient ideal $J_V(I)$ of an ideal $I$ over all regular parameter systems of $\widehat \calo_{\AAA^n,0}$ is attained (it might be $\infty$). \comment{[{\it Hint}: Use the Artin approximation theorem.]}\end{eg}


\begin{eg} \label{10.28} Show that the supremum of the order of the coefficient ideal of an ideal $I$ in $W$ at a point $a$ can be realized, if the supremum is finite, by a regular system of parameters of the local ring $\calo_{W,a}$ without passing to the completion.\end{eg}


\begin{eg} \label{10.29}$\challenge$ Show that this maximum, when taken at any point of the top locus $S$ of $\AAA^n$ where $I$ has maximal order $o$, defines an upper semicontinuous function on $S$.\end{eg}


\begin{eg} \label{10.30}$\hint$ Let $V$ be the hypersurface $x_n=0$ of $\AAA^n$ and let $V'\too V$ be the blowup with regular center $Z$ in $V$. Let $I$ be an ideal of order $o$ at a point $a$ of $Z$, with weak transform $I^\curlyvee$. Assume that $\ord_{a'}I^\curlyvee=\ord_aI$ at the origin $a'$ of an $x_j$-chart for a $j<n$. Compare the controlled transform of the coefficient ideal $J_V(I)$ of $I$ with respect to the control $c=o!$ with the coefficient ideal $J_{V'}(I^\curlyvee)$ of $I^\curlyvee$.\end{eg}


\begin{eg} \label{10.31} Show that, in characteristic zero, the top locus of an ideal $I$ of $W$, when taken locally at a point $a$, is contained in a local regular hypersurface $V$ through $a$. Does this hypersurface maximize the order of the associated coefficient ideal?\end{eg}


\begin{eg} \label{10.32} Consider $f=x^3+y^2z$ in $\AAA^3$ and the point blowup of $\AAA^3$ at the origin. Find, according to the characteristic, at all points of the exceptional divisor a hypersurface of weak maximal contact for the strict transform of $f$.\end{eg}






\begin{eg} \label{10.33} Consider surfaces defined by polynomials $f=x^o + y^az^b\cdot g(y,z)$ where $y^az^b$ is considered as an exceptional monomial factor of the coefficient ideal in the hypersurface $V$ defined by $x=0$ (up to raising the coefficient ideal to the power $c=o!$). Assume that $a+b+\ord_0g\geq o$. Give three examples where the order of $g$ at $0$ is not maximal over all choices of local hypersurfaces at $0$, and indicate the coordinate changes which make it maximal.\end{eg}


\begin{eg} \label{10.34} Consider surfaces defined by polynomials $f=x^o + y^az^b\cdot g(y,z)$ where $y^az^b$ is considered again as an exceptional monomial factor of the coefficient ideal in the hypersurface $V$ defined by $x=0$. Assume that $a+b+\ord_0g\geq o$. Compute the strict transform $f'=x^o + y^{a'}z^{b'}\cdot g'(y,z)$ of $f$ under point blowup at points where the order of $f$ has remained equal to $o$.  Find three examples where the order of $g'$ is not maximal over all local coordinate choices.\end{eg} 


\begin{eg} \label{10.35} For a flag of local regular subvarieties $V_{n-1}\supset\ldots\supset V_1$ at $a$ in $W=\AAA^n$ one gets from an ideal $I=I_n$ in $W$ a chain of coefficient ideals $J_{n-1},\ldots, J_1$ in $V_{n-1},\ldots, V_1$ respectively, defined recursively as follows. Assume that $J_{n-1},\ldots, J_{i+1}$ have been constructed and that $J_{i+1}$ can be decomposed into $J_{i+1}=M_{i+1}\cdot I_{i+1}$ with prescribed monomial factor $M_{i+1}$ and some residual factor $I_{i+1}$. Then set  $J_i=J_{V_i}(I_{i+1})$, the coefficient ideal of $I_{i+1}$ in $V_i$ at $a$. Show that the lexicographic maximum of the vector of orders of the ideals $J_i$ at $a$ over all choices of flags at $a$ admitting the above factorizations of the ideals $J_i$ can be realized stepwise, choosing first a local hypersurface $V_{n-1}$ in $W$ at $a$ maximizing the order of $J_{n-1}$ at $a$, then a local hypersurface $V_{n-2}$ in $V_{n-1}$ at $a$ maximizing the order of $J_{n-2}$, and iterating this process.\end{eg}

\goodbreak

\section{Lecture XI: Resolution in Zero Characteristic}

\noindent The inductive proof of resolution of singularities in characteristic zero requires a more detailed statement about the nature of the resolution:


\begin{thm} \label{strongresolution} Let $W$ be a regular ambient variety and let $E \subset W$ be a (possibly empty) divisor with normal crossings. Assume that the characteristic of the ground field is $0$. Let $J$ be an ideal on $W$, together with a decomposition $J=M\cdot I$ into a principal monomial ideal $M$, the monomial factor of $J$, supported on a normal crossings divisor $D$ transversal to $E$, and an ideal $I$, the residual factor of $J$. Let $c_+ \geq 1$ be a given number, the control of $J$. 

There exists a sequence of blowups of $W$ along regular centers $Z$ transversal to $E$ and $D$ and their total transforms, contained in the locus $\ttop(J,c_+)$ of points where $J$ and its controlled transforms with respect to $c_+$ have order $\geq c_+$, and satisfying the requirements \textit{equivariance} and \textit{excision} of a strong resolution so that the order of the controlled transform of $J$ with respect to $c_+$ drops eventually at all points below $c_+$.\end{thm} 


\begin{defn} In the situation of the theorem, with prescribed divisor $D$ and control $c_+$, the ideal $J$ is called \textit{resolved} with respect to $D$ and $c_+$ if the order of $J$ at all points of $W$ is $<c_+$.\end{defn} 


\begin{rem} Once the order of the controlled transform of $J$ has dropped below $c_+$, induction on the order can be applied to find an additional sequence of blowups which makes the order of the controlled transform of  $J$ drop everywhere to $0$. At that stage, the controlled transform of $J$ has become the whole coordinate ring of $W$, and the total transform of $J$, which differs from the controlled transform by a monomial exceptional factor, has become a monomial ideal supported on a normal crossings divisor $D$ transversal to $E$. This establishes the existence of a strong embedded resolution of $J$, respectively of the singular variety $X$ defined by $J$ in $W$. \end{rem}


\begin{proof} The technical details can be found in \cite{EH}, and motivations are given in \cite{Ha_BAMS_1}. The main argument is the following. 

The resolution process has two different stages: In the first, a sequence of blowups is chosen through a local analysis of the singularities of $J$ and by induction on the ambient dimension. The order of the weak transforms of $I$ will be forced to drop eventually below the maximum of the order of $I$ at the points $a$ of $W$. By induction on the order one can then apply additional blowups until the order of the weak transform of $I$ has become equal to $0$. At that moment, the weak transform of $I$ equals the coordinate ring/structure sheaf of the ambient variety, and the controlled transform of $J$ has become a principal monomial ideal supported on a suitable transform of $D$, which will again be a normal crossings divisor meeting the respective transform of $E$ transversally. 

For simplicity, denote this controlled transform again by $J$. It is a principal monomial ideal supported by exceptional components. The second stage of the resolution process makes the order of $J$ drop below $c_+$. The sequence of blowups is now chosen globally according to the multiplicities of the exceptional factors appearing in $J$. This is a completely combinatorial and quite simple procedure which can be found in many places in the literature \cite {Hironaka_Annals, EV_Obergurgl, EH}.

The first part of the resolution process is much more involved and will be described now. The centers of blowup are defined locally at the points where the order of $I$ is maximal. Then it is shown that the definition does not depend on the local choices and thus defines a global, regular and closed center in $W$. Blowing up $W$ along this center will improve the singularities of $I$, and finitely many further blowups will make the order of the weak transform of $I$ drop.

The local definition of the center goes as follows. Let $S=\ttop(I)$ be the stratum of points in $W$ where $I$ has maximal order. Let $a$ be a point of $S$ and denote by $o$ the order of $I$ at $a$. Choose an osculating hypersurface $V$ for $I$ at $a$ in an open neighbourhood $U$ of $a$ in $W$. This is a closed regular hypersurface of $U$ given by a suitable partial derivative of an element of $I$ of order $o-1$. The hypersurface $V$ maximizes the order of the coefficient ideal $J_{n-1}=J_V(I)$ over all choices of local regular hypersurfaces, cf.÷ Prop.÷ \ref{575}. Thus $\ord_a(J_{n-1})$ does not depend on the choice of $V$. There exists \comment{[by Prop.÷ \ref{887},]} a closed stratum $T$ in $S$ at $a$ of points $b\in S$ where the order of $J_{n-1}$ equals the order of $J_{n-1}$ at $a$.

The ideal $J_{n-1}$ on $V$ together with the new control $c=o!$ can now be resolved by induction on the dimension of the ambient space, applying the respective statement of the theorem. The new normal crossings divisor $E_{n-1}$ in $V$ is defined as $V\cap E$. For this it is necessary that the hypersurface $V$ can be chosen transversal to $E$. This is indeed possible, but will not be shown here. Also, it is used that $J_{n-1}$ admits again a decomposition $J_{n-1}=M_{n-1}\cdot I_{n-1}$ with $M_{n-1}$ a principal monomial ideal supported on some normal crossings divisor $D_{n-1}$ on $V$, and $I_{n-1}$ an ideal, the residual factor of $J_{n-1}$.

There thus exists a sequence of blowups of $V$ along regular centers transversal to $E_{n-1}$ and $D_{n-1}$ and their total transforms, contained in the locus $\ttop(J_{n-1},c)$ of points where $J_{n-1}$ and its controlled transforms with respect to $c$ have order $\geq c$, and satisfying the requirements of a strong resolution so that the order of the weak transform of $I_{n-1}$ drops eventually at all points below the maximum of the order of $I_{n-1}$ at points $a$ of $V$.

All this is well defined on the open neighbourhood $U$ of $a$ in $W$, but depends a priori on the choice of the hypersurface $V$, since the centers are chosen locally in each $V$. It is not clear that the local choices patch to give a globally defined center. One can show that this is indeed the case, even though the local ideals $J_{n-1}$ do depend on $V$. 
The argument relies on the fact that the centers of blowup are constructed as the maximum locus of an invariant associated to $J_{n-1}$. By induction on the ambient dimension, such an invariant exists for each $J_{n-1}$: In dimension $1$, it is just the order. In higher dimension, it is a vector of orders of suitably defined coefficient ideals. One then shows that the invariant of $J_{n-1}$ is independent of the choice of $V$ and defines an upper semicontinuous function on the stratum $S$. Its maximum locus is therefore well defined and closed. Again by induction on the ambient dimension, it can be assumed that it is also regular and transversal to the divisors $E_{n-1}$ and $D_{n-1}$.

By the assertion of the theorem in dimension $n-1$, the order of the weak transform of $J_{n-1}$ can be made everywhere smaller than $c=o!$ by a suitable sequence of blowups. 
The sequence of blowups also transforms the original ideal $J=M\cdot I$, producing controlled transforms of $J$ and weak transforms of $I$. The order of $I$ cannot increase in this sequence, if the centers are always chosen inside $S$. To achieve this inclusion a technical adjustment of the definition of coefficient ideals is required which will be omitted here. If the order of $I$ drops, induction applies. So one is left to consider points where possible the order of the weak transform has remained constant. There, one will use the commutativity of blowups with the passage to coefficient ideals, Prop.÷ \ref{884}: The coefficient ideal of the final transform of $I$ will equal the controlled transform of the coefficient ideal $J_{n-1}$ of $I$. But the order of this controlled transform has dropped below $c$, hence, by Prop.÷ \ref{414}, also the order of the weak transform of $I$ must have dropped. This proves the existence of a resolution.\end{proof}


\bigskip
\noindent\textit{\examples}


\begin{eg} \label{11.4} In the situation of the theorem, take $W = \AAA^2$, $E = \emptyset$, and $J = (x^2y^3)=1\cdot I$ with control $c_+=1$. The ideal $J$ is monomial, but not resolved yet, since it is not supported on the exceptional divisor. It has order $5$ at $0$, order $3$ along the $x$-axis, and order $2$ along the $y$-axis. Blow up $\AAA^2$ in $0$. The controlled transform in the $x$-chart is $J^! = (x^4y^3)$ with exceptional factor $I_E = (x^4)$ and residual factor $I_1 = (y^3)$, the strict transform of $J$. The controlled transform in the $y$-chart is $J^! = (x^2y^4)$ with exceptional factor $I_E = (y^4)$ and residual factor $I_1 = (x^2)$, the strict transform of $J$. One additional blowup resolves $J$.
\end{eg}


\begin{eg} \label{11.5} In the situation of the preceding example, replace $E=\emptyset$ by $E = V(x +
y)$, respectively $E=V(x+y^2)$, and resolve the ideal $J$.\end{eg}


\begin{eg} \label{11.6} Take $W = \AAA^3$, $E = V(x + z^2)$, and $J = (x^2y^3)$ with control $c_+=1$. The variety $X$ defined by $J$ is the union of the $xz$- and the $yz$-plane. The top locus of $J$ is the $z$-axis, which is tangent to $E$. Therefore it is not allowed to take it as the center of the first blowup. The only possible center is the origin. Applying this blowup, the top locus of the controlled transform of $J$ and the total transform of $E$ have normal crossings, so that the top locus can now be chosen as the center of the next blowup. Resolve the singularities of $X$.\end{eg}


\begin{rem} The non-transversality of the candidate center of blowup with already existing exceptional components is known as the \textit{transversality problem}. The preceding example gives a first instance of the problem, see \cite{EH,  Ha_BAMS_1} for more details.\end{rem}


\begin{eg} \label{11.7} Take $W = \AAA^3$, $E =D= \emptyset$ and $J = (x^2 + yz)$. In characteristic different from $2$, the origin of $\AAA^3$ is the unique isolated singular point of the cone $X$ defined by $J$. The ideal $J$ has order $2$ at $0$. After blowing up the origin, the singularity is resolved and the strict transform of $J$ defines a regular surface. It is transversal to the exceptional divisor.\end{eg}


\begin{eg} \label{11.8} Take $W = \AAA^3$, $E=D = \emptyset$, $J = (x^2 + y^az^b)$ with $a,b\in\NNN$. According to the values of $a$ and $b$, the top locus of $J$ is either the origin, the $y$- or $z$-axis, or the union of the $y$- and the $z$-axis. If $a+b\geq 2$, the $yz$-plane is a hypersurface of maximal contact for $J$ at $0$. The coefficient ideal is $J_1 = (y^az^b)$ (up to raising it to the square). This is a monomial ideal, but not supported yet on exceptional components. Resolve $J$. \end{eg}


\begin{eg} \label{11.9} Take $W = \AAA^4$, $E=D = V(yz)$, $J = (x^2 + y^2z^2(y + w))$. The order of $J$ at $0$ is $2$. The $yzw$-hyperplane has maximal contact with $J$ at $0$, with coefficient ideal $J_1 = y^2z^2(y + w)$, in which the factor $M_1=(y^2z^2)$ is exceptional and $I_1=(y + w)$ is residual. Resolve $J$. The top locus of $I_1$ is the plane $V(y + w) \subset \AAA^3$, hence not contained in the top locus of $J$. This technical complication is handled by introducing an intermediate ideal, the \textit{companion ideal}  \cite{Villamayor_Hypersurface, EH}.\end{eg}


\begin{eg} \label{11.10} Compute the first few steps of the resolution process for the three surfaces defined in $\AAA^3$ by the polynomials $x^2+y^2z$, $x^2+y^3+z^5$ and $x^3 + y^4z^5+z^{11}$.\end{eg}


\begin{eg} \label{11.11} Prove with all details the embedded resolution of plane curves in characteristic zero according to the above description.\end{eg}






\begin{eg} \label{11.12} Resolve the following items according to Thm.÷ \ref{strongresolution}, taking into account different characteristics of the ground field.  

(a) $W=\AAA^2$, $E=\emptyset$, $J=I=(x^2y^3)$, with control $c_+=1$.

(b)   $W=\AAA^2$, $E=V(x)$, $J=(x^2y^3)$, $c_+=1$.

(c)   $W=\AAA^2$, $E=V(xy)$, $J=(x^2y^3)$, $c_+=1$.

(d)    $W=\AAA^3$, $E=V(x+z^2)$, $J=(x^2y^3)$, $c_+=1$.

(e)    $W=\AAA^3$, $E=V(y+z)$, $J=(x^3+(y+z)z^2)$, $c_+=3$.

(f)    $W=\AAA^3$, $E=\emptyset$, $J=I=(x^2+yz)$, $c_+=2$.

(g)   $W=\AAA^3$, $E=\emptyset$, $J=I=(x^3+y^2z^2)$, $c_+=3$.

(h)    $W=\AAA^3$, $E=\emptyset$, $J=I=(x^2+xy^2+y^5)$, $c_+=2$.

(i)    $W=\AAA^3$, $E=\emptyset$, $J=I=(x^2+xy^3+y^5)$, $c_+=2$.

(j)    $W=\AAA^3$, $E=V(x+y^2)$, $J=(x^3+(x+y^2)y^4z^5)$, $c_+=3$.

(k)    $W=\AAA^3$, $E=\emptyset$, $J=I=(x^2+y^2z^2(y+z))$, $c_+=12$.

(l)    $W=\AAA^3$, $E=V(yz)$, $J=(x^2+y^2z^2(y+z))$, $c_+=2$.

(m)   $W=\AAA^4$, $E=V(yz)$, $J=(x^2+y^2z^2(y+w))$, $c_+=2$. 
\end{eg}


\begin{eg} \label{11.13} Consider $J=I=(x^2+y^7+yz^4)$ in $\AAA^3$, with $E=\emptyset$ and $c_+=2$. Consider the sequence of three point blowups with the following centers. First the origin of $\AAA^3$, then the origin of the $y$-chart, then the origin of the $z$-chart. On the last blowup, consider the midpoint between the origin of the $y$- and the $z$-chart. Show that $J$ is resolved at this point if the characteristic is zero or $>2$.\end{eg}


\begin{eg} \label{11.14}$\hint$ What happens in the preceding example if the characteristic is equal to $2$? Describe in detail the behaviour of the first two components of the resolution invariant under the three local blowups.\end{eg} 


\begin{eg} \label{11.15}$\hint$ A combinatorial version of resolution is known as Hironaka's polyhedral game: Let $N$ be an integral convex polyhedron in ${\Bbb R}^n_+$, i.e., the positive convex hull $N={\rm conv} (S+{\Bbb R}^n_+)$ of a finite set $S$ of points in $\NNN^n$. Player $A$ chooses a subset $J$ of $\{1,\ldots, n\}$, then player $B$ chooses an element $j\in J$. After these moves, $S$ is replaced by the set $S'$ of points $\alpha$ defined by $\alpha_i'=\alpha_i$ if $i\neq j$ and $\alpha_j'=\sum_{k\in J}\alpha_k$, giving rise to a new polyhedron $N'$. Player $A$ has won if after finitely many rounds the polyhedron has become an orthant $\alpha + {\Bbb R}^n_+$. Player $B$ can never win, but only prevent player $A$ from winning. Show that player $A$ has a winning strategy, first with and then without using induction on $n$ \cite{Spivakovsky, Zeillinger_Polyhedra_Game}. \end{eg}

\ignore

\begin{eg} \label{} (Bernd Schober's first example)  For an algebraic variety $X$, consider the following stratification
$$
X_1 := Reg(X),\; Y_1 := X \setminus X_1,\; X_2 := Reg(Y_1); \ldots .
$$
This defines a stratification and the smallest strata is closed and regular.
Blowing up this smallest strata does in general not give a resolution of singularities.

{\it First example (complicated)}

Let $K$ be a non-perfect field of characteristic $p > 0$.
Then there exists an $\lambda \in K \setminus K^p$.
Consider the hypersurface (over $K$) given by
$$
f = x^{p-1} + y (z^{p-1} + \lambda w^{p-1} ) = 0.
$$
Then the singular locus $Y_1$ is $V(x,w,z)$ and $V(x,y,z^{p-2} + \lambda w^{p-2})$ 
(since the Jacobian ideal is $< x^{p-2}, z^{p-1} + \lambda w^{p-1}, y z^{p-2}, y w^{p-2}>$).
But $Y_1$ has a singularity in the origin $0$ 
(either you consider the last component or just take the intersection of both components).
Thus $X_2 = Y_1 \setminus \{0\}$ and $X_3 = \{0\}$.
Following the strategy the center of the blow up is the origin $0 = V(x,y,z,w)$.
Look in the $Y$-chart
(i.e. apply the coordinate tranformation $(x,y,z,w) \mapsto (x'y,y,yz',yw')$; here the exceptional component is given by $y=0$)
where the strict transform is given by
$$
f'= y^{-(p-1)} \cdot f(x'y,y,yz',yw') = x'^{p-1} + y (z'^{p-1} + \lambda w'^{p-1} ).
$$
This is the same variety and the process has made a cycle. Hence the claim is true. 

After I constructed this example Franklin explained me in a conversation the following easier example which is valid for all fields.
\end{eg}


\begin{eg} \label{} (Bernd Schober's second example)

Consider the variety given by 
$$
x^n  + y^m z^n = 0 
$$
over an {\it arbitrary} field, for some $m,n\in \NNN, m\geq 3, n\geq 2$ and if the characteristic of the ground field is $p>0$, then further require either $n \neq p$ or $m\neq p$ (e.g. $m=3$, $n=2$).

(If you don't like this general case, just consider the case $m=3$, $n=2$, i.e. the variety given by $x^2 + y^3 z^2=0$).

One can easily compute that the singular locus consists of two curves, namely $V(x,y)$ and $V(x,z)$ (no matter what the charateristic is). 
One can also express the singular locus as $V(x,yz)$.
It follows that the singular locus has a singularity in the origin $V(x,y,z)$.
Blow up and look at the $y$-chart: The strict transform is the same equation as before. So again a cycle has occured.\end{eg}

\recognize

\section{Lecture XII: Positive Characteristic Phenomena}

\noindent The existence of the resolution of varieties of arbitrary dimension over a field of positive characteristic is still an open problem. For curves, there exist various proofs \cite{Kollar_Book} chap.÷ 1. For surfaces, the first proof of non-embedded resolution was given by Abhyankar in his thesis \cite{Abhyankar_56}. Later, he proved embedded resolution for surfaces, but the proof is scattered over several papers which sum up to over 500 pages  \cite{Abhyankar_Book_59, Abhyankar_64, Abhyankar_66, Abhyankar_66b, Abhyankar_67, Abhyankar_Book_98}. Cutkosky was able to shorten and simplify the argumens substantially \cite{Cutkosky_Skeleton}. An invariant similar to the one used by Abhyankar was developed independently by Zeillinger and Wagner \cite{Zeillinger_Thesis, Wagner_Thesis, HW}. Lipman gave an elegent proof of non-embedded resolution for arbitrary two-dimensional schemes \cite{Lipman_Surfaces, Artin}. Hironaka proposed an invariant based on the Newton polyhedron which allows to prove embedded 
resolution of surfaces which are hypersurfaces \cite{Hironaka_Bowdoin, Cossart_Excellent, Ha_Excellent}. Hironaka's invariant seems to be restricted to work only for surfaces. The case of higher codimensional surfaces was settled by Cossart-Jannsen-Saito, extending Hironaka's invariant \cite{CJS}. Different proofs were recently proposed by Benito-Villamayor and Kawanoue-Matsuki \cite{Benito_Villamayor_Math_Ann, Kawanoue_Matsuki_Surfaces}.

For three-folds, Abhyankar and Cossart gave partial results. Quite recently, Cossart-Piltant proved non-embedded resolution of three-folds by a long case-by-case study \cite{Cossart_Piltant_1, Cossart_Piltant_2,  Cossart_Piltant_3,  Cossart_Piltant_4}. See also \cite{Cutkosky_3-Folds}.

Programs and techniques for resolution in arbitrary dimension and characteristic have been developed quite recently, among others, by Hironaka, Teissier, Kuhlmann, Kawanoue-Matsuki, Benito-Bravo-Villamayor, Hauser-Schicho, Cossart \cite{Hironaka_CMI, Teissier, Kuhlmann, Kawanoue, Kawanoue_Matsuki, Benito_Villamayor_Math_Ann, Benito_Villamayor_Comp, Bravo_Villamayor_10, Bravo_Villamayor_11,  Ha_Wild, HS_Game, Cossart_WMC}.  


\begin{rem} Let $\K$ be an algebraically closed field of prime characteristic $p>0$, and let $X$ be an affine variety defined in $\AAA^n$. The main problems already appear in the hypersurface case. Let $f=0$ be an equation for $X$ in $\AAA^n$. Various properties of singularities used in the characteristic zero proof fail in positive characteristic:

(a) \label{775} The top locus of $f$ of points of maximal order need not be contained locally in a regular hypersurface. Take the variety in $\AAA^4$ defined by $f = x^2 + yz^3 + zw^3 + y^7w$ over a field of characteristic $2$  \cite{Narasimhan, Mulay, Kawanoue, Ha_Obstacles}.

(b) There exist sequences of blowups for which the sequence of points above a given point $a$ where the order of the strict transforms of $f$ has remained constant are eventually not contained  in the strict transforms of any regular local hypersurface passing through $a$. \comment{[Take the same example]}

(c) Derivatives cannot be used to construct hypersurfaces of maximal contact. \comment{[Take a polynomial which is a $p$-th power]}

(d) The characteristic zero invariant is no longer upper semicontinuous when translated to positive characteristic. \comment{[$x^p+y^pz$, insert Hironaka's example of generic going up]}
\end{rem}


\begin{eg} \label{12.2} A typical situation where the characteristic zero resolution invariant does not work in positive characteristic is as follows. Take the polynomial $f = x^2 + y^7 + yz^4$ over a field of characteristic $2$. There exists a sequence of blowups along which the order of the strict transforms of $f$ remains constant equal to $2$ but where eventually the order of the residual factor of the coefficient ideal of $f$ with respect to a hypersurface of weak maximal contact increases.

The hypersurface $V$ of $W=\AAA^3$ defined by $x = 0$ produces -- up to raising the ideal to the square -- as coefficient ideal of $f$ the ideal on $\AAA^2$ generated by $y^7 + yz^4$. Its order at the origin is $5$. 

Blow up $\AAA^3$ at the origin. In the $y$-chart $W'$ of the blowup, the total transform of $f$ is given by
\[
f^* =  y^2 (x^2 + y^5 + y^3z^4),
\]
with $f' =x^2 + y^3(y^2 + z^4)$ the strict transform of $f$. It has order $2$ at the origin of $W'$. The generator of the coefficient ideal of $f'$ in $V': x= 0$ decomposes into a monomial factor $y^3$ and a residual factor $y^2 + z^4$.  The order of the residual factor at the origin of $W'$ is $2$. Blow up $W'$ along the $z$-axis, and consider the $y$-chart $W''$, with strict transform 
\[
f'' = x^2 + y(y^2 + z^4),
\]
of $f$. The residual factor of the coefficient ideal in $V'': x=0$ equals again $y^2 + z^4$. Blow up $W''$ at the origin and consider the $z$-chart $W'''$, with strict transform 
\[
f''' = x^2 + yz(y^2 + z^2).
\]
The origin of $W'''$ is the intersection point of the two exceptional components $y=0$ and $z=0$. The residual factor of the coefficient ideal in $V''': x=0$ equals $y^2 + z^2$, of order $2$ at the origin of $W'''$. 

Blow up the origin of $W'''$ and consider the affine chart $W^{(iv)}$ given by the coordinate transformation 
\[
x \mapsto xz, y \mapsto yz + z, \textrm{ and } z \mapsto z.
\]
The origin of this chart is the midpoint of the new exceptional component. The strict transform of $f'''$ equals
\[
f^{(iv)} =x^2 + (y + 1)z^2y^2,
\]
which, after the coordinate change $x \mapsto x + yz$, becomes 
\[
f^{(iv)} = x^2 + y^3z^2 + y^2z^2.
\]
The order of the strict transforms of $f$ has remained constant equal to $2$ along the sequence of local blowups. The order of the residual factor of the associated coefficient ideal has decreased from $5$ to $2$ in the first blowup, then remained constant until the last blowup, where it increased from $2$ to $3$.
\end{eg}


\begin{rem} The preceding example shows that the order of the residual factor of the coefficient ideal of the defining ideal of a singularity with respect to a local hypersurface of weak maximal contact is not directly suited for an induction argument as in the case of zero characteristic. For surfaces, practicable modifications of this invariant are studied in \cite{HW}.\end{rem}


\begin{defn} A hypersurface singularity $X$ at a point $a$ of affine space $W=\AAA^n_\K$ over a field $\K$ of characteristic $p$ is called \textit{purely inseparable of order $p^e$ at $a$} if there exist local coordinates $x_1,\ldots,x_n$ on $W$ at $a$ such that $a=0$ and such that the local equation $f$ of $X$ at $a$ is of the form
\[
f=x_n^{p^e} + F(x_1,\ldots,x_{n-1})
\]
for some $e\geq 1$ and a polynomial $F\in \K[x_1,\ldots,x_{n-1}]$ of order $\geq p^e$ at $a$. \end{defn}


\begin{prop} For a purely inseparable hypersurface singularity $X$ at $a$, the polynomial $F$ is unique up to multiplication by units in the local ring $\calo_{W,a}$ and the addition of $p^e$-th powers in $\K[x_1,\ldots,x_{n-1}]$. \end{prop}


\begin{proof} Multiplication by units does not change the local geometry of $X$ at $a$. A 
coordinate change in $x_n$ of the form $x_n\mapsto x_n+a(x_1,\ldots, x_{n-1})$ with $a\in \K[x_1,\ldots,x_{n-1}]$ transforms $f$ into $f=x_n^{p^e} + a(x_1,\ldots,x_{n-1})^{p^e}+F(x_1,\ldots,x_{n-1})$. This implies the assertion.\end{proof}


\begin{defn} Let affine space $W=\AAA^n$ be equipped with an exceptional normal crossings divisor $E$ produced by earlier blowups with multiplicities $r_1,\ldots,r_n$. Let $x_1,\ldots,x_n$ be local coordinates at a point $a$ of $W$ such that $E$ is defined at $a$ by $x_1^{r_1}\cdots x_n^{r_n}=0$. Let $f=x_n^{p^e} + F(x_1,\ldots,x_{n-1})$ define a purely inseparable singularity $X$ of order $p^e$ at the origin $a=0$ of $\AAA^n$ such that $F$ factorizes into
\[
F(x_1,\ldots,x_{n-1})=x_1^{r_1}\cdots x_{n-1}^{r_{n-1}}\cdot G(x_1,\ldots,x_{n-1}).
\] 
The \textit{residual order of $X$ at $a$ with respect to $E$} is the maximum of the orders of the polynomials $G$ at $a$ over all choices of local coordinates such that $f$ has the above form \cite{Hironaka_CMI, Ha_Wild}.
\end{defn}


\begin{rem} The residual order can be defined for arbitrary singularities \cite{ Ha_Wild}. In characteristic zero, the definition coincides with the second component of the local resolution invariant, defined by the choice of an osculating hypersurface or, more generally, of a hypersurface of weak maximal contact and the factorization of the associated coefficient ideal.\end{rem}


\begin{rem} In view of the preceding example, one is led to investigate the behaviour of the residual order under blowup at points where the order of the singularity remains constant. Moh showed that it can increase at most by $p^{e-1}$ \cite{Moh}. Abhyankar seems to have observed already this bound in the case of surfaces. He defines a correction term $\varepsilon$ taking values equal to $0$ or $p^{e-1}$ which is added to the residual order according to the situation in order to make up for the occasional increases of the residual order \cite{Abhyankar_67, Cutkosky_Skeleton}. A similar construction has been proposed by Zeillinger and Hauser-Wagner \cite{Zeillinger_Thesis, Wagner_Thesis, HW}. This allows, at least for surfaces, to define a secondary invariant after the order of the singularity, the \textit{modified residual order of the coefficient ideal}, which does not increase under blowup. The problem then is to handle the case where the order of the singularity and the modified residual order remain 
constant. It is not clear how to define a third  invariant which manifests the improvement of the singularity. \end{rem}


\begin{rem} Following ideas of Giraud, Cossart has studied the behaviour of the order of the Jacobian ideal of $f$, defined by certain partial derivatives of $f$. Again, it seems that hypersurfaces of maximal contact do not exist for this invariant \cite{Giraud, Cossart_WMC}. There appeared promising recent approaches by Hironaka, using the machinery of differential operators in positive characteristic, by Villamayor and collaborators using instead of the restriction to hypersurfaces of maximal contact projections to regular hypersurfaces via elimination algebras, and by Kawanoue-Matsuki using their theory of idealistic filtrations and differential closures. None of these proposals has been able to produce an invariant or a resolution strategy which works in positive characteristic for all dimensions.\end{rem}


\begin{rem} Another approach consists in analyzing the singularities and blowups for which the residual order increases under blowup. This leads to the notion of kangaroo singularities:\end{rem}


\begin{defn} A hypersurface singularity $X$ defined at a point $a$ of affine space $W=\AAA^n_\K$ over a field $\K$ of characteristic $p$ by a polynomial equation $f=0$ is called \textit{a kangaroo singularity} if there exists a local blowup $\pi: (\wt W,a')\too (W,a)$ of $W$ along a regular center $Z$ contained in the top locus of $X$ and transversal to an already existing exceptional normal crossings divisor $E$ such that the order of the strict transform of $X$ remains constant at $a'$ but the residual order of the strict transform of $f$ increases at $a'$.  The point $a'$ is then called \textit{kangaroo point of $X$ above $a$}.\end{defn}


\begin{rem} Kangaroo singularities can be defined for arbitrary singularities. They have been characterized in all dimensions by Hauser \cite{Ha_BAMS_2, Ha_Wild}. However, the knowledge of the algebraic structure of these singularities did not yet give any hint how to overcome the obstruction caused by the increase of the residual order.\end{rem}


\begin{prop} If a polynomial $f=x_n^{p^e}+F(x_1,\ldots,x_{n-1})= x_n^{p^e}+ x_1^{r_1}\cdots x_{n-1}^{r_{n-1}}\cdot G(x_1,\ldots,x_{n-1})$ defines a kangaroo singularity of order $p^e$ at $0$ the following conditions must hold: (a) the sum of the $r_i$ and of the order of $G$ at $0$ is divisible by $p^e$; (b) the sum of the residues of the exceptional multiplicities $r_i$ modulo $p^e$ is bounded by $(m-1)\cdot p^e$ with $m$ the number of exceptional multiplicities not divisible by $p^e$; (c)  the initial form of $F$ equals a specific homogeneous polynomial which can be explicitly described. Any kangaroo point $a'$ of $X$ above $a$ lies outside the strict transform of the components of the exceptional divisor at $a$ whose multiplicities are not a multiple of $p^e$.\end{prop}


\begin{rem} A more detailed description of kangaroo singularities and a further discussion of typical characteristic $p$ phenomena can be found in \cite{Ha_BAMS_2, Ha_Wild}.\end{rem}

\goodbreak
\bigskip
\noindent\textit{\examples}


\begin{eg} \label{12.15} Prove the resolution of plane curves in arbitary characteristic by using the order and the residual order as the resolution invariants.\end{eg}


\begin{eg} \label{12.16} Let $\K$ be an algebraically closed field of characteristic $p>0$. Develop a significant notion of resolution for elements of the quotient of rings $\K[x,y]/\K[x^p,y^p]$. Then prove that such a resolution always exists.\end{eg}


\begin{eg} \label{12.17}$\hint$ Consider the polynomial $f = x^2 + yz^3 + zw^3 + y^7w$ on $\AAA^4$ over a ground field of characteristic $2$. Its maximal order is $2$, and the respective top locus is the image of the monomial curve $(t^{32}, t^7, t^{19}, t^{15}), t \in \K$. The image curve has embedding dimension $4$ at $0$ and cannot be embedded locally at $0$ into a regular hypersurface of $\AAA^4$.  Hence there is no hypersurface of maximal contact with $f$ locally at the origin.\end{eg}


\begin{eg} \label{12.18}$\hint$ Find a surface $X$ in positive characteristic and a sequence of point blowups starting at $a\in X$ so that some of the points above $a$ where the order of the weak transforms of $X$ remains constant eventually leave the transforms of any local regular hypersurface passing through $a$. \end{eg}


\begin{eg} \label{12.19}$\hint$ Show that $f=x^2+yz^3+zw^3+y^7w$ has in characteristic $2$ top locus $\ttop(f)$ equal to the parametrized curve $(t^{32}, t^7,t^{19}, t^{15})$ in $\AAA^4$ \cite{Narasimhan, Mulay, Kawanoue}. \end{eg}


\begin{eg} \label{12.20}$\hint$ Show that $f$ is not contained in the square of the ideal defining the parametrized curve $(t^{32}, t^7,t^{19}, t^{15})$. \comment{[This signifies that the order of $f$ along this curve has to be really computed in the localization of $K[x,y,z,w]$ with respect to the ideal of the curve. This phenomenon is related to the difference between powers of ideals and symbolic powers of ideals \cite{Zariski_Samuel, Hochster_73, Pellikaan}.]}\end{eg}


\begin{eg} \label{12.21}$\hint$ Find the defining ideal for the image of the monomial curve $(t^{32}, t^7,t^{19}, t^{15})$ in $\AAA^4$. What is the local embedding dimension at $0$?
\end{eg}


\begin{eg} \label{12.22}$\hint$ Show that $f=x^2+yz^3+zw^3+y^7w$ admits in characteristic $2$ at the point $0$ no local regular hypersurface of permanent maximal contact (i.e., whose successive strict transforms contain all points where the order of $f$ has remained constant in any sequence of blowups with regular centers inside the top locus).\end{eg}


\begin{eg} \label{12.23} Consider $f=x^2+y^7+yz^4$ in characteristic $2$. Show that there exists a sequence of point blowups for which $f$ admits at the point $0$ no local regular hypersurface whose transforms have weak maximal contact with the transforms of $f$ as long as the order of $f$ remains equal to $2$.\end{eg}


\comment{[\begin{eg} \label{} Let $c$ be an integer $\geq 2$, and $[a,b]\subset\NNN$ an interval. Determine the number of points in $[a,b]\cap c\cdot\NNN$ according to the residues $\overline a$ and $\overline b$ of $a$ and $b$ modulo $c$.\end{eg}]}


\comment{[\begin{eg} \label{} For $r,s,k\in\NNN$, let $\Delta$ be the subset of $\NNN^2$ defined by $\Delta=((r,s)+\NNN^2)\cap \{(a,b)\in\NNN^2,\, a+b\leq k\}$. For $c\geq 2$, evaluate the cardinality of $\Delta\cap c\cdot\NNN^2$ according to the values of $r,s,k$.\end{eg}]}


\begin{eg} \label{12.24}$\hint$ Define the $p$-th order derivative of polynomials in $\K[x_1,\ldots,x_n]$ for $\K$ a field of characteristic $p$.\end{eg}


\begin{eg} \label{12.25}$\hint$ Construct a surface of order $p^5$ in $\AAA^3$ over a field of characteristic $p$ for which the residual order increases under blowup.\end{eg}


\begin{eg} \label{12.26}$\hint$ Show that for the polynomial $f=x^p+y^pz$ over a field of characteristic $p$ and taking $E=\emptyset$ the residual order of $f$ along the (closed) points of the $z$-axis is not equal to its value at the generic point.\end{eg}


\begin{eg} \label{12.27}$\hint$ Let $y_1,\ldots, y_m$ be fixed coordinates, and consider a homogeneous polynomial $G(y) = y^r\cdot g(y)$ with $r\in \NNN^m$ and $g(y)$ homogeneous of degree $k$. Let $G^+(y)$ be the polynomial obtained from $G$ by the linear coordinate change $y_i\too y_i+y_m$ for $i=1,\ldots, m-1$. Show that the order of $G^+$ along the $y_m$-axis is at most $k$. \end{eg}


\begin{eg} \label{12.28} Express the assertion of the preceding example through the invertibility of a matrix of multinomial coefficients.\end{eg}


\begin{eg} \label{12.29}$\hint$ Consider $G(y,z)= y^rz^s\sum_{i=0}^k { k+r\choose i+r} y^i (tz-y)^{k-i}$. Compute for $t\in \K^*$ the polynomial $G^+(y,z)=G(y+tz,z)$ and its order with respect to $y$ modulo $p$-th power polynomials.\end{eg}


\begin{eg} \label{12.30}$\hint$ Determine all homogeneous polynomials $G(y,z)=y^rz^s g(y,z)$  so that $G^+(y,z)$ has order $k+1$ with respect to $y$ modulo $p$-th power polynomials.\end{eg}


\begin{eg} \label{12.31}$\challenge$ Find a new  systematic proof for the embedded resolution of surfaces in three-space in arbitrary characteristic.\end{eg}


\begin{eg} \label{12.32} Let $G(x)$ be a polynomial in one variable over a field $\K$ of characteristic $p$, of degree $d$ and order $k$ at $0$. Let $t\in K$, and consider the equivalence class $\overline G$ of $G(x+t)$ in $K[x]/K[x^p]$ (i.e., consider $K(x+t)$ modulo $p$-th power polynomials). What is the maximal order of $\overline G$ at $0$? Describe all examples where this maximum is achieved.\end{eg}




\section{Discussion of selected examples}

The comments and hints below were compiled by Stefan Perlega and Valerie Roitner.\\



{\sc Ex.÷ \ref{1.1}.} Let $X$ be defined by $27x^2y^3z^2+(x^2+y^3-z^2)^3=0$ and let $Y=C\times C$ be the cartesian product of the cusp $C$ defined by $x^3-y^2=0$  with itself. The surface $Y$ can be parametrized by $(s,t)\too (s^3,s^2,t^3,t^2)$.
Composing this map with the projection $(x,y,z,w)\mapsto(x,-y+w,z)$ from $\AAA^4$ to $\AAA^3$ gives $(s,t)\mapsto(s^3,t^2-s^2,t^3)$. Substitution into the equation of $X$ gives $0$,
$$27s^6t^6(t^2-s^2)^3+(s^6+(t^2-s^2)^3-t^6)^3=$$
$$=27(s^6t^{12}-3s^8t^{10}+3s^{10}t^8-s^{12}t^6)+27(s^{12}t^6-3s^{10}t^8+3s^8t^{10}-s^6t^{12})= 0.$$
Therefore the image of $Y$ under the projection lies inside $X$. It remains to show that every point in $X$ is obtained in this way. The restriction $Y\too X$ of the projection $\AAA^4\too\AAA^3$ is a finite map (as can e.g.÷ be checked by using a computer algebra program), hence the image of $Y$ is closed in $X$. As it has dimension two and $X$ is an irreducible surface, the image of $Y$ is whole $X$.\\


{\sc Ex.÷ \ref{1.2}.} Asides from the symmetries in the text, replacing $x$ with $\pm\sqrt{-1}\cdot x$ and $z$ with $\pm\sqrt{-1}\cdot z$ gives a new symmetry.

In characteristic 2 the defining polynomial of $X$ equals $f=x^2y^3z^3+(x^2+y^3+z^3)^3$ and there appears an additional symmetry which is given by interchanging $x$ with $z$.\\


{\sc Ex.÷ \ref{1.3}.} Replacing $z$ by $\sqrt{-1}\cdot z$ gives the equation $g=-27x^2y^3z^2+(x^2+y^3+z^2)^3 =0$, cf. Figure 3.
Symmetries of this surface are e.g. given by replacing $x$ by $-x$ or $z$ by $-z$ or by interchanging $x$ with $z$. Also sending $(x,y,z)$ to $(t^3x,t^2y, t^3z)$ with $t\in\K$ gives a symmetry.

\noindent
The partial derivatives of $g$ give the equations for $\Sing(X)$,
\[
x\cdot (-9y^3z^2+(x^2+y^3+z^2)^3)=0,
\]
\[
 y^2\cdot (-9x^2z^2+(x^2+y^3+z^2)^3)=0,
\]
\[
 z\cdot (-9y^3x^2+(x^2+y^3+z^2)^3)=0.
\]
Therefore the singular locus of $X$ has six components, given by the equations
\[
 x=y=z=0,
\]
\[
 x=y^3+z^2=0,
\]
\[
 z=x^2+y^3=0,
\]
\[
 y=x^2+z^2=0,
\]
\[
 z-x=y^3-z^2=0,
\]
\[
 z+x=y^3-z^2=0.
 \]

\begin{center}
   \includegraphics[scale=0.17]{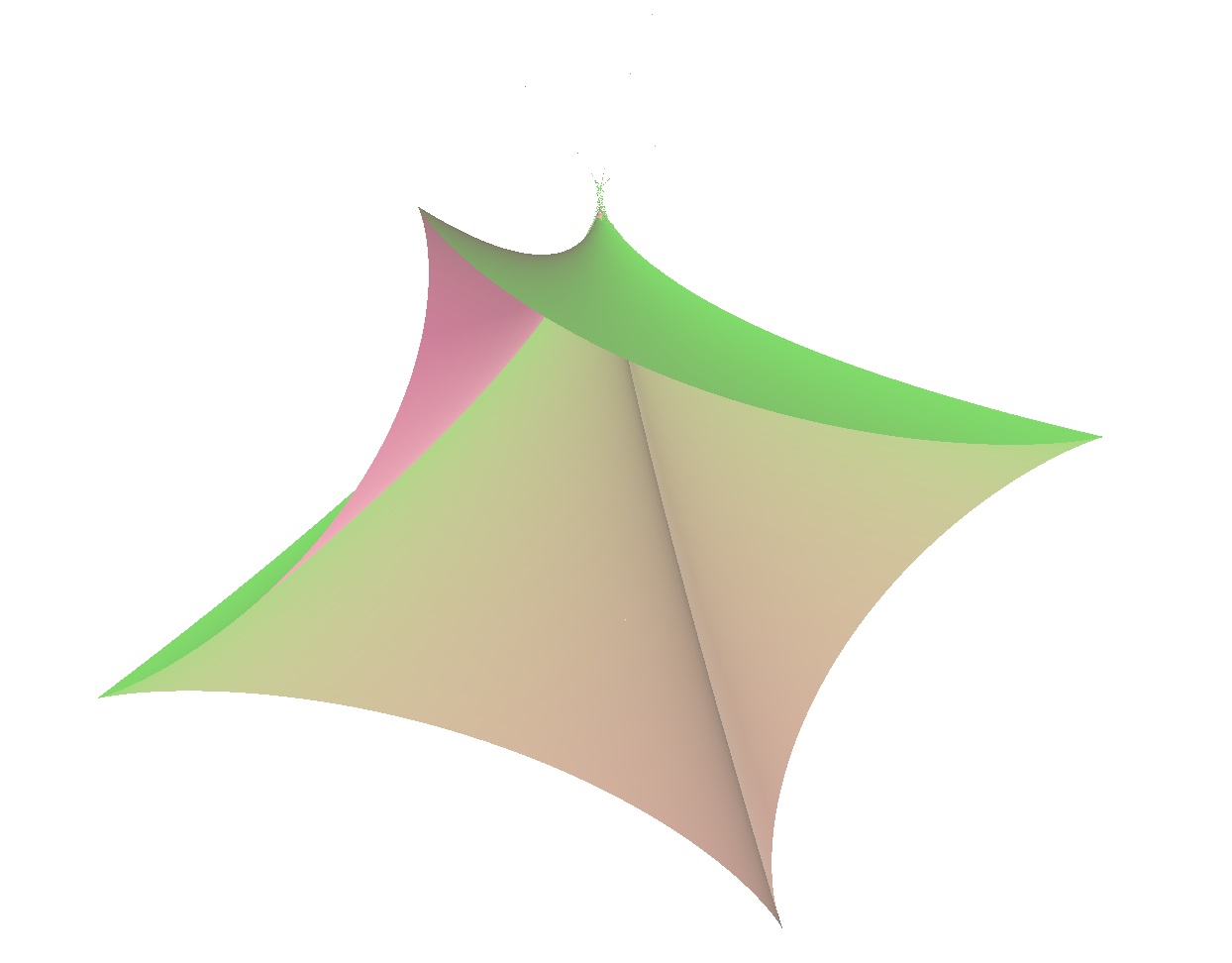}
   \\
Figure 3: The surface of equation $27x^2y^3z^2=(x^2+y^3+z^2)^3 $.
\end{center}\medskip

\medskip


{\sc Ex.÷ \ref{1.4}.} The point $a=(0,1,1)$ lies in the component $S$  of $\Sing(X)$ defined by $g=x-y^3+z^2=0$. Since $(\partial_x g(a), \partial_y g(a), \partial_z g(a))=(1,3,2)$, the component $S$ is regular at $a$. The curve given by the parametrization $(0,t^2,t^3)$ lies in $S$ and passes through $a$ for $t=1$. Its tangent vector at $t=1$ is $(0,2,3)$, which is a normal vector to the plane $P$. Therefore $P$ intersects $S$ transversally at $a$.

The intersection $X\cap P$ is given by the equations $y=\frac{1}{2}(5-3z)$ and $h(x,z)=f(x, \frac{1}{2}(5-3z),z)=27x^2z^2 \frac{1}{8}(5-3z)^3+(x^2+\frac{1}{8}(5-3z)^3-z^2)^3=0$.
Computing the points where $\partial_x h=0$ and $\partial_z h=0$ gives (among others) the solution $x=0$, $z=1$, whence $y=1$. Therefore $X\cap P$ has a singularity at $(0,1,1)$.
The Taylor expansion of $h$ at $(x,z)=(0,1)$ is given by
$$\left(-\frac{2197}{8}w^3 + O(w^4) \right)+\left(27-\frac{135w}{2}+93w^2+O(w^3)\right)x^2+O(x^4),$$
where $w=z-1$.\\


{\sc Ex.÷ \ref{1.6}.} The blowup $\wt{X}$ of $X$ in the origin of $\AAA^3$ is defined in $X\times\PPP^2$ by the equations
$$xu_2-yu_1=xu_3-zu_1=yu_3-zu_2=0,$$
where $(u_1:u_2:u_3)$ are projective coordinates on $\PPP^2$. The blowup map $\pi:\wt X\to X$ is the restriction to $\wt X$ of the projection $X\times\PPP^2\too X$.

The equation of the $x$-chart of the blowup of $X$ is obtained by replacing $(x,y,z)$ with $(x,xy,xz)$ in the equation defining $X$ and by factoring the appropriate power of the polynomial $x$ defining the exceptional divisor. This gives the equation 
$$C_x: 27x^7y^3z^2+(x^2+x^3y^3-x^2z^2)=x^6\cdot (27xy^3z^2+(1+xy^3-z^2)^3)=0,$$
so that the patch of $\wt X$ in the $x$-chart is defined by $27xy^3z^2+(1+xy^3-z^2)^3=0$ in $\AAA^3$. See Figure 4.
Similarly, the $y$- and the $z$-chart are obtained by replacing $(x,y,z)$ with $(xy,y,yz)$ and $(xz,yz,z)$ respectively and give the equations
$$C_y: y^6\cdot (27x^2yz^2+(x^2+y-z^2)^3)=0$$
and
$$C_z: z^6\cdot (27x^2y^3z+(x^2+y^3z-1)^3)=0.$$
The blowup $\wt Y$ of $Y=C\times C$ in the origin of $\AAA^4$ is defined in
$Y\times\PPP^3$ by 
$$xu_2-yu_1=xu_3-zu_1=xu_4-wu_1=yu_3-zu_2=yu_4-wu_2=zu_4-wu_3=0.$$
The $x$-chart of $\wt Y$ is obtained by replacing $(x,y,z,w)$ with $(x,xy,xz,xw)$ in the equations defining $Y$ and by then factoring the appropriate powers of $x$. It is hence given in $\AAA^4$ by the equations $1-xy^3=z^2-xw^3=0.$
The other charts are obtained similarly: in the $y$-chart, $x^2-y=z^2-yw^3=0$;
in the $z$-chart, $x^2-y^3z=1-zw^3=0$; in the $w$-chart, $x^2-y^3w=z^2-w=0$. 


\begin{center}
   \includegraphics[scale=0.17]{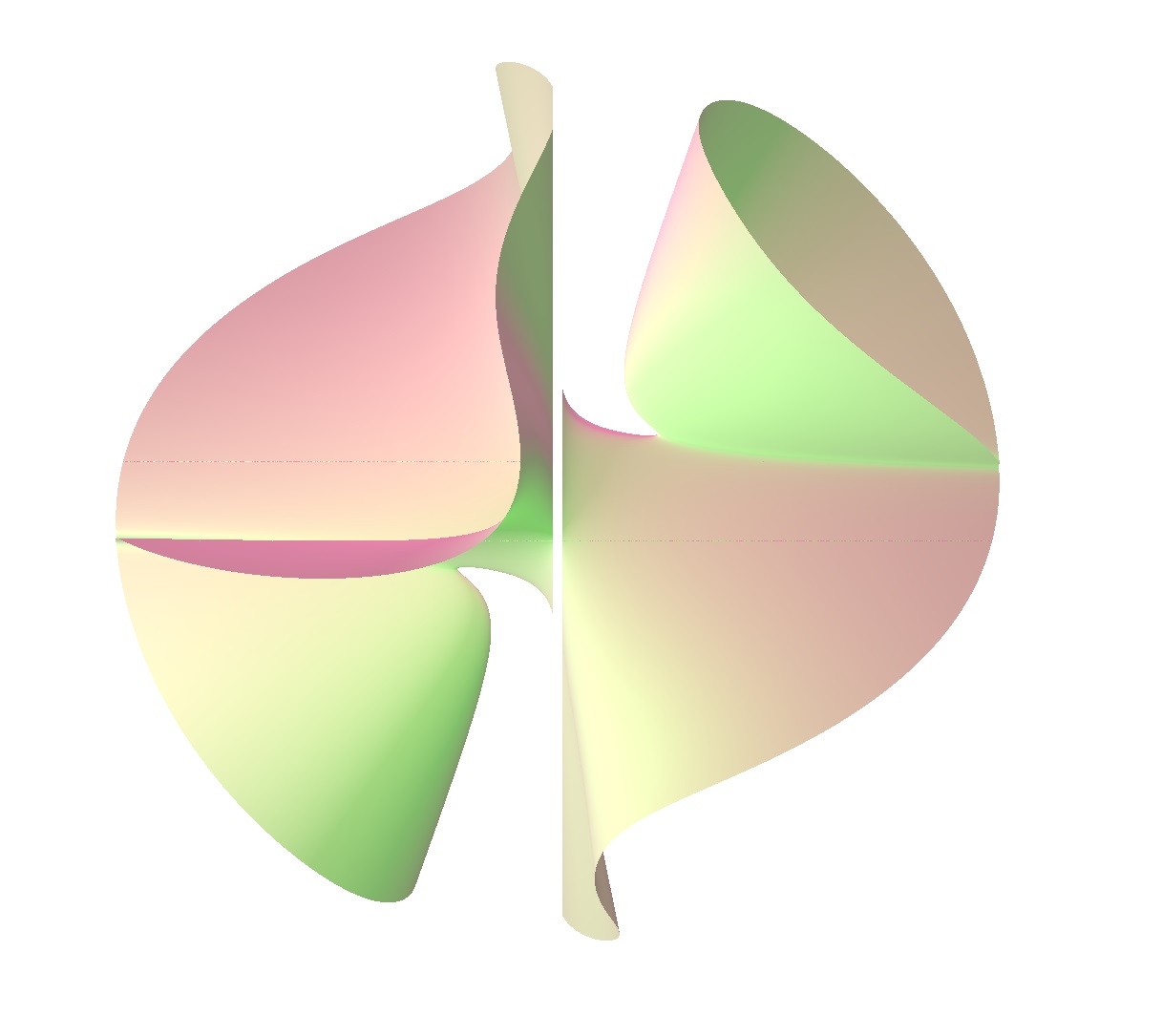}
   \includegraphics[scale=0.15]{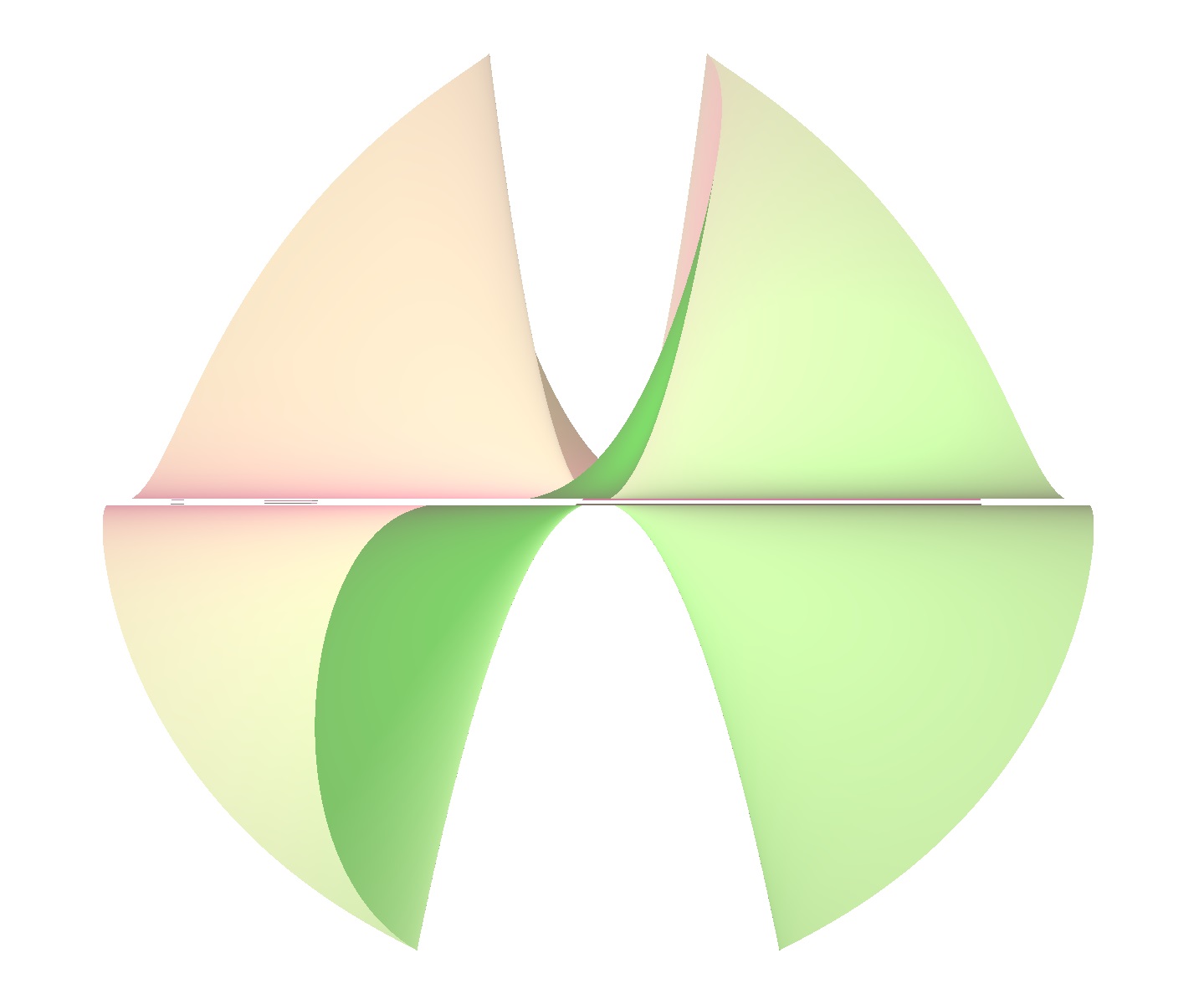}
   \includegraphics[scale=0.2]{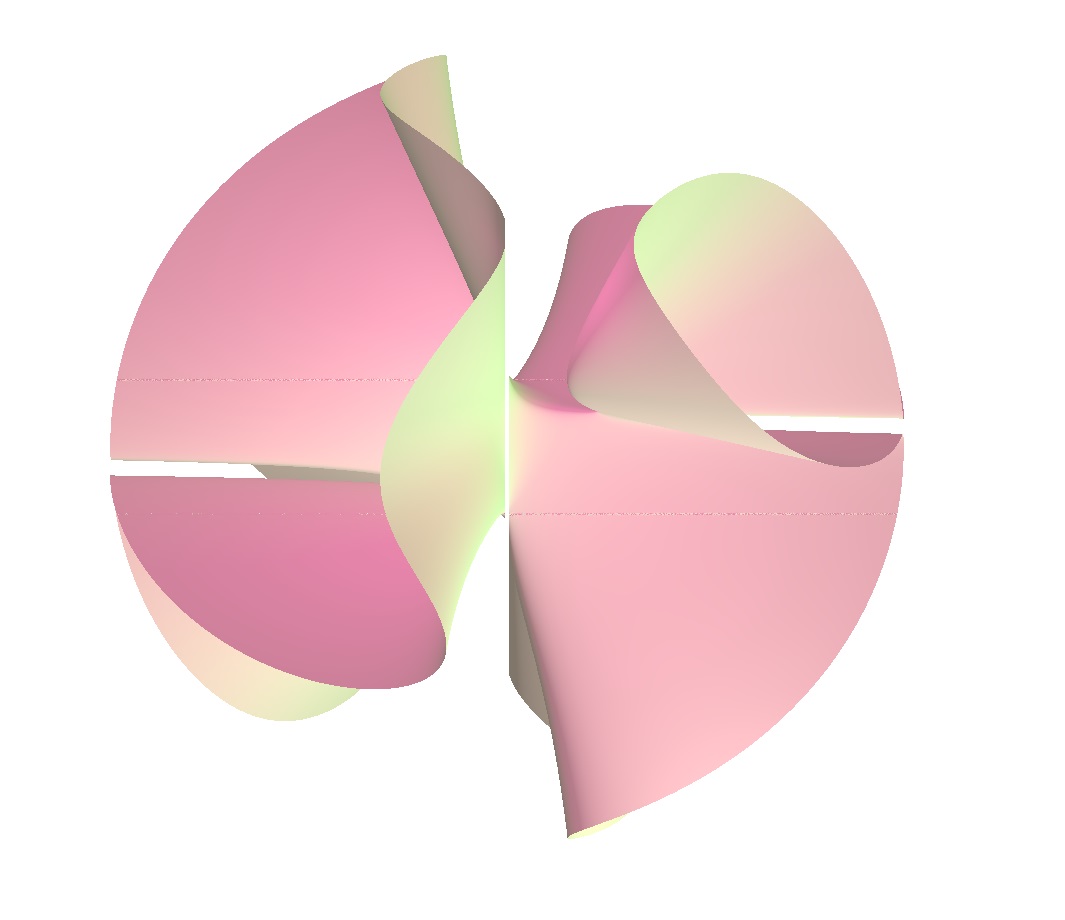}\\
Figure 4: Chart expressions of the blowup of the surface Camelia.
\end{center}\medskip


{\sc Ex.÷ \ref{2.34}.} Since $\mathcal{O}_{X,a}\subseteq\wh{\mathcal{O}}_{X,a}$ and $\wh\mm_{X,a}=\mm_{X,a}\cdot\wh {\calo}_{X,a}$ every regular system of parameters $x_1,\dots, x_n$ of $\mathcal{O}_{X,a}$ is also a generator system of $\wh \mm_{X,a}$ in $\wh{\mathcal{O}}_{X,a}$. But there exists a subset $S$ of $\{x_1,\dots,x_n\}$ such that $S$ is a regular system of parameters of $\wh{\calo}_{X,a}$. The converse does not hold, since in general a regular system of parameters of $\wh{\calo}_{X,a}$ need not belong to $\calo_{X,a}$.\\



{\sc Ex.÷ \ref{2.38}.} The image of $f:(x,y)\mapsto (xy,y)$ is $f(\AAA^2)=(\AAA^2\setminus (\AAA^1\times \{0\}))\cup \{(0,0)\}$. This is a constructible and dense subset of $\AAA^2$. The inverse $f\inv:(x,y)\mapsto(\frac{x}{y},{y})$ of $f$ is defined on the complement of the $x$-axis. It cannot be extended to the origin $0$ of $\AAA^2$, since $f$ contracts the $x$-axis of $\AAA^2$ onto $0$.\\


{\sc Ex.÷ \ref{2.40}.}  It is easily checked that the inverse of $\varphi_{ij}$ is $\varphi_{ji}$:
\[
[\varphi_{ji}\circ\varphi_{ij}(x_1,\dots,x_n)]_k=
\begin{cases}
 \frac{x_k}{x_j}\cdot{x_j}=x_k, &\hskip 1cm  k\not= i,j, \\
 \frac{1}{x_j}\cdot x_ix_j=x_i, & \hskip 1cm k=i, \\
 \frac{1}{1/x_j}=x_j, & \hskip 1cm k=j,\\
\end{cases}
\]
where $[-]_k$ denotes the $k$-th component. Therefore $\varphi_{ji}\circ\varphi_{ij}$ is the identity. The domain $U_{ij}$ of $\varphi_{ij}$ is $\AAA^n\setminus V(x_j)$, a dense open subset of $\AAA^n$. The image $\varphi_{ij}(U_{ij})$ equals $\AAA^n\setminus V(x_i)$. It is open, too. Since all components of $\varphi_{ij}$ are rational and the denominators do not vanish on $U_{ij}$, the maps $\varphi_{ij}$ induce biregular maps $U_{ij}\too U_{ji}$.\\


{\sc Ex.÷ \ref{2.41}.} Elliptic curves admit an additive group structure. It therefore suffices to restrict to the origin $a=0\in\AAA^2$ and to show that the curve is formally isomorphic at $0$ to $(\wh{\AAA}^1,0)$. But $\K[[x,y]]/(y^2-x^3+x)\isom \K[[y]]$ since, by the implicit function theorem for formal power series, the residue class of $x$ in $\K[[x,y]]/(y^2-x^3+x)$ can be expressed as a series in $y$. To show that $X$ is nowhere locally biregular to the affine line is much harder.\\


{\sc Ex.÷ \ref{2.42}.} A morphism $f:X\to Y$ of algebraic varieties is proper if it is separated and universally closed, i.e., if for any variety $Z$ and any morphism $h:Z\to Y$ the induced morphism $g:X\times_Y Z\to Y, g(x,z)=f(x)=h(z)$ is closed. In the example, already the image $\pi(X)=\AAA^1\setminus\{0\}$ of $X$ is not closed.\\


{\sc Ex.÷ \ref{2.43}.} A formal isomorphism between $(\wh X,0)$ and $(\wh Y,0)$ would send the the jacobian ideals generated by the partial derivatives of $x^3-y^2$ and $x^5-y^2$ to each other. This is impossible. On the other hand, the formal isomorphism defined by $(x,y)\mapsto (x\sqrt[3]{1+x^2},y)$ sends $(\wh X,0)$ onto $(\wh Z,0)$. It is an isomorphism by the inverse function theorem for formal power series. \\


{\sc Ex.÷ \ref{3.27}.} The maximal ideal of $\K[[x]]$ is generated by $x$. Since $\sqrt{1+x}$ is a unit in $\K[[x]]$ (its inverse is 
$(1+x)^{-1/2}=\sum_{k\geq0}\binom{-1/2}{k}(-1)^k x^k$), also $x\cdot \sqrt{1+x}$ generates this ideal. Therefore $x\cdot \sqrt{1+x}$ is a regular system of parameters in $\K[[x]]$. It does not stem from a regular parameter system in $\K[x]_{(x)}$ since $\sqrt{1+x}$ cannot be written as a quotient of polynomials.\\


{\sc Ex.÷ \ref{3.28}.} The maximal ideal of the local ring $\calo_{\AAA^2,0}=\K[x,y]_{(x,y)}$ is generated by $x$ and $y$.  As $x^2+1$ is invertible in $\K[x,y]_{(x,y)}$, also $y^2-x^3-x=y^2-x(x^2+1)$ and $y$ generate this ideal. They hence form a regular system of parameters of $\calo_{\AAA^2,0}$. For the remaining assertions, see ex. \ref{2.41}.\\


{\sc Ex.÷ \ref{3.31}.} Computing the derivatives of $x^2-y^2z$ with respect to $x$, $y$ and $z$ and setting them zero shows that the singular locus of $X$ is the $z$-axis.
Let $a=(0,0,t)$ with $t\not=0$ be a point of the $z$-axis outside the origin. Then $\wh{\calo}_{X,a}=\CCC[[x,y,z]]/(x^2-y^2(z-t))$. Since $\sqrt{z-t}$ is for $t\not=0$ a formal power series in $z$ the quotient can be rewritten as $\CCC[[x,y,z]]/(x-y\sqrt{z-t})(x+y\sqrt{z-t})$. The product $(x-y\sqrt{z-t})(x+y\sqrt{z-t})=0$ defines two smooth formal surfaces intersecting each other transversally. They are formally isomorphic to the union of the two planes defined by $(x-y\sqrt{-t})(x+y\sqrt{-t})=0$. Hence $a$ is a normal crossings point of $X$. But it is not a simple normal crossings point since, globally, $X$ consists of only one component which is moreover singular at $a$.

The formal neighborhood of $0$ is $\wh{\calo}_{X,0}=K[[x,y,z]]/(x^2-y^2z)$. Since $x^2-y^2z$ is irreducible in $K[[x,y,z]]$ and $X$ is singular at $0$ the origin is not a normal crossings point of $X$.\\


{\sc Ex.÷ \ref{3.34}.} Over $\CCC$ or any finite field $\FFF_q$ with $q$ congruent to $1$ modulo $4$ there exists a square root of $-1$, hence the variety defined by $x^2+y^2=(x+\sqrt{-1}y)(x-\sqrt{-1}y)=0$ has normal crossings at the origin. Since both components are regular at $0$ it also has simple normal crossings.\\\goodbreak 

The surface defined by $x^2+y^2+z^2=0$ does not have normal crossings over $\RRR$ or $\CCC$, but it has normal crossings over a field of characteristic $2$, since $x^2+y^2+z^2=(x+y+z)^2$ in characteristic $2$ so that the variety is a double plane.  The same holds for the variety defined by $x^2+y^2+z^2+w^2=0$. The variety $xy(x-y)=0$ has three components passing through $0$, but as $\AAA^2$ has only two coordinate subspaces the variety does not have normal crossings at $0$. The same argument works for $xy(x^2-y)=0$.  The variety defined by $(x-y)z(z-x)=0$ has both normal and simple normal crossings at the origin over any field.\\


{\sc Ex.÷ \ref{3.35}.} The visualization of the surface looks as follows, cf.÷ figure 5.\\


\begin{center}
  \includegraphics[width=6cm]{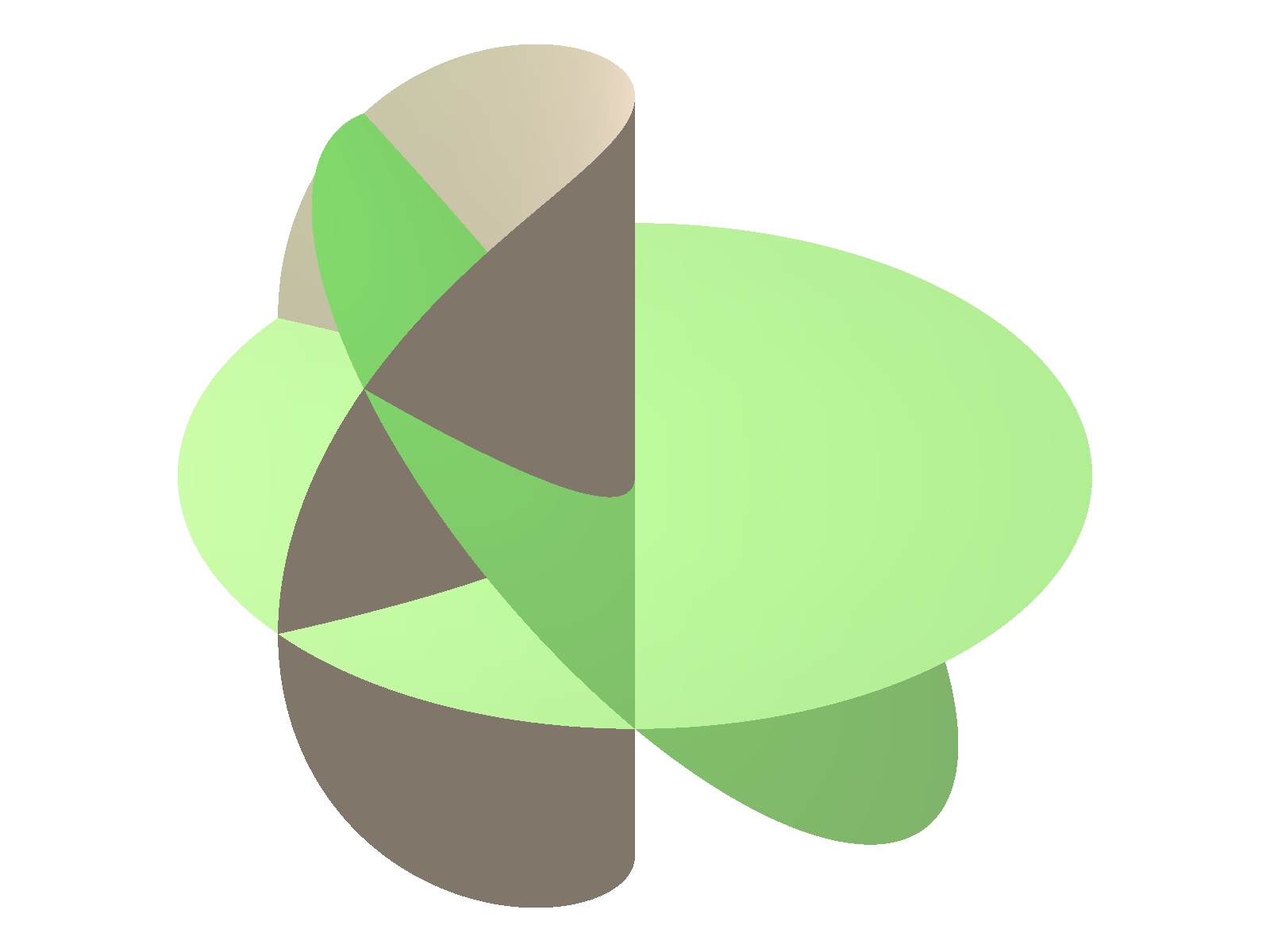}\\
 Figure 5: The zeroset of $(x-y^2)(x-z)z=0$ in $\AAA^3$.
\end{center}\medskip


{\sc Ex.÷ \ref{3.38}.} A point $a$ is a regular point of a variety if the local ring is regular, i.e., if its maximal ideal can be generated by as many elements as the Krull dimension indicates. For a cartesian product $X\times Y$, the local ring at a point $(a,b)$ is the tensor product of the local rings of the two factors,
$$\calo_{X\times Y,(a,b)}=\K[X\times Y]_{m_{(a,b)}}\isom \K[X]_{m_a}\otimes \K[Y]_{ m_b}=\calo_{X,a}\otimes\calo_{Y,b}$$
and the same holds for the respective maximal ideals,
$$ m_{(a,b)}/ m_{(a,b)}^2\isom m_a/m_a^2\otimes m_b/ m_b^2.$$
Therefore $(a,b)\in X\times Y$ is regular if and only if $a$ is regular in $X$ and $b$ is regular in $Y$. Consequently, $\Sing(X\times Y)=(\Sing(X)\times Y)\cup (X\times\Sing(Y)).$\\


{\sc Ex.÷ \ref{3.39}.} A rather simple example for such a variety is the surface {\sl T\"ulle} defined by $xz(x+z-y^2)=0$ in $\AAA^3$. The pairwise intersections of its three components are the $y$-axis, respectively the two parabolas defined by $x=z-y^2=0$ and $z=x-y^2=0$, and are therefore regular. The parabolas are tangent to the $y$-axis at $0$. Said differently, the intersection of all three components, which is set-theoretically just the origin $0\in\AAA^3$, is scheme-theoretically singular (i.e., the sum of the three ideals $(x)$, $(z)$ and $(x+z-y^2)$ is not the ideal $(x,y,z)$ defining the origin in $\AAA^3$). Because of this, the surface is not mikado at $0$. \\



{\sc Ex.÷ \ref{4.25}.} For $X=V(xy,x^2)$ and $a$ the origin, the local ring is $\calo_{X,a}=\K[x,y]/(xy,x^2)_{(x,y)}\isom\K[x,y]_{(x,y)}/(xy,x^2)$. As a scheme, $X$ equals the $y$-axis together with an embedded point at $0$, since $(xy,x^2)$ has the primary decomposition $(xy,x^2)=(x)\cap (x^2,y)$. As $Z=\{0\}$ cannot be defined  in $X$ by a single equation, it is not a Cartier divisor.\\


{\sc Ex.÷ \ref{4.26}.} The local ring $\calo_{X,a}$ of $X=\AAA^1$ at the origin is  $\calo_{X,a}=K[x]_{(x)}.$ The monomial $x^2$ is not a zero-divisor in $\calo_{X,a}$, hence $Z=V(x^2)$ is a Cartier divisor in $\AAA^1$. Similarly, $x^2y$ is not a zero divisor in $\calo_{\AAA^2,a} =K[x,y]_{(x-a_1,y-a_2)}$ for any $a=(a_1,a_2)\in\AAA^2$ and $Z=V(x^2y)$ is a Cartier divisor in $\AAA^2$.\\


{\sc Ex.÷ \ref{4.27}.} As $Z=V(x^2,y)$ is defined in $X=V(x^2,xy)$ by $\overline y=0$ where the residue class $\overline y$ of $y$ in $\calo_{X,0}=K[x,y]_{(x,y)}/(x^2,xy)$ is a zero divisor, it follows that $Z$ is not a Cartier divisor in $X$.\\ 


{\sc Ex.÷ \ref{4.38}.} The Rees algebra of the ideal $I=(x,2y)$ in $\ZZZ_{20}[x,y]$ is given by 
$$\wt{R}=\bigoplus_{k\geq 0} I^kt^k=\bigoplus_{k\geq 0}(xt,2yt)^k=\ZZZ_{20}[x,y,xt,2yt]$$
with $\deg t=1$. It is isomorphic to $\ZZZ_{20}[u,v,w,z]/(uz-2vw)$ with $\deg u=\deg v=0$ and $\deg w=\deg z=1$.\\



{\sc Ex.÷ \ref{4.46}.} The blowup of $\AAA^2$ in $Z=\{(0,1)\}$ is given by the closure $\overline{\Gamma}$ of the graph of
$\gamma:\AAA^2\setminus\{(0,1)\}\to\PPP^1, (a,b)\mapsto (a:b-1)$ in $\AAA^2\times\PPP^1$ together with the projection $\pi:\overline{\Gamma}\to X, (a,b,(c:d))\mapsto (a,b)$. More explicitly, 
$$\overline{\Gamma}=\{(a,b, (a:b-1)),\, (a,b)\in\AAA^2\}\cup\{(0,1)\}\times\PPP^1.$$
The line $L$ in $\AAA^2$ defined by $x+y=0$ does not contain the point $(0,1)$, therefore its preimage in $\overline \Gamma$ is $\pi\inv(L)=\{((a,-a),(a:-a-1)),\, a\in\AAA^1\}.$

The line $L'$ in $\AAA^2$ defined by $x+y=1$ contains the point $(0,1)$, therefore its preimage in $\overline \Gamma$ is $\pi\inv(L')=\{(a,1-a,(a:-a)),\, a\in\AAA^1\setminus\{0\}\}\cup\{(0,1)\}\times\PPP^1$.\\


{\sc Ex.÷ \ref{4.47}.} The chart expressions of the blowup map of $\AAA^3$ along the $z$-axis are given by $\pi_1(x,y,z)=(x,xy,z)$ and $\pi_2(x,y,z)=(xy,y,z)$. Their inverses are 
$\pi_1^{-1}(x,y,z)=(x,\frac{y}{x},z)$ and $\pi_2^{-1}(x,y,z)=(\frac{x}{y},y,z)$. The chart transition maps are given by the compositions
$$\pi_1^{-1}\circ\pi_2(x,y,z)=\left(xy,\frac{1}{x},z\right)$$
and
$$\pi_2^{-1}\circ\pi_1(x,y,z)=\left(\frac{1}{y},xy,z\right).$$


{\sc Ex.÷ \ref{4.48}.} Let the chosen line $Z$ in the cone $X$ be defined by $x=y-z=0$. It requires two equations, therefore it is not a Cartier divisor. The polynomials $x$ and $y-z$ form a regular sequence in $\K[X] =\K[x,y,z]/(x^2+y^2-z^2)$. Therefore the blowup $\wt{X}$ of $X$ along $Z$ is given in $X\times\PPP^1$ by the equation $xv-(y-z)u=$ where $(u:v)$ are projective coordinates in $\PPP^1$.

An alternate way to compute the blowup is by applying first the linear coordinate change $u=x, w=y+z$ and $t=y-z$. The equation $x^2+y^2=z^2$ of $X$ transforms into $Y: u^2+2wt=0$ and $x=y-z=0$ becomes $u=t=0$. The latter equations define the $w$-axis, which is now the center of blowup. The resulting chart expressions of the blowup map are given by $\pi_1(u,w,t)=(ut,w,t)$ and $\pi_2(u,w,t)=(u,w,wt)$. This gives the chart descriptions of the total transform $Y^*$ of $Y$ via
$$Y_1^*=V(u^2t^2+2wt)=V(t)\cup V(u^2t+2w),$$
$$Y_2^*=V(u^2+2uwt)=V(u)\cup V(u+2wt),$$
respectively. Substituting backwards gives
$$X_1^*=V(y-z)\cup V(x^2(y-z)+2(y+z)),$$
$$X_2^*=V(x)\cup V(x+2(y^2-z^2)),$$
where the second components denote the chart expressions of the strict transform $X^s$ of $X$ in the blowup $\wt\AAA^3$ of $\AAA^3$ along $Z$, i.e., of the blowup $\wt X$ of $X$ along $Z$, cf.÷ def.÷ \ref{stricttransform} as well as Prop.÷ \ref{basechange} together with its corollary.\\


{\sc Ex.÷ \ref{4.49}.} The polynomials $z$ and $x^2+(y+2)^2-1$ defining the circle in $\AAA^3$ form a regular sequence. Hence, if $(u:v)$ denote projective coordinates in $\PPP^1$, the blowup $\wt{X}$ of $X$ along the circle is given in $X\times \PPP^1$ by the equation $uz-v(x^2+(y+2)^2-1)=0$ together with the projection $\pi:\wt{X}\to X$ on the first factor.\\
The polynomials $z$ and $y^2-x^3-x$ defining the elliptic curve in $\AAA^3$ form a regular sequence, too. The blowup $\wt{X}$ of $X$ along this curve is given in $X\times \PPP^1$ by the equation $uz-v(y^2-x^3-x)=0$ together with the projection $\pi:\wt{X}\to X$ on the first factor.\\


{\sc Ex.÷ \ref{4.52}.} The ideals $I=(x_1,\dots, x_n)$ and $J =(x_1,\dots,x_n)^m=I^m$ induce isomorphic Rees algebras $\wt R=\oplus_{k\geq 0}\, I^k$ and $\wt S=\oplus_{k\geq 0}\, J^k =\oplus_{k\geq 0}\, I^{mk}$, hence define the same blowups of $\AAA^n$. See \cite{Moody} for more details on the characterization of ideals producing the same blowup. \\\goodbreak


{\sc Ex.÷ \ref{4.57}.} Since the centers of the blowups are given by coordinate subspaces, the definition of blowup via affine charts can be used.
For the blowup in the origin the three chart expressions of the blowup map are given by $\pi_x(x,y,z)=(x,xy,xz)$, $\pi_y(x,y,z)=(xy,y,yz)$ and $\pi_z(x,y,z)=(xz,yz,z)$.
This gives for the total transforms of $X$ the expressions
$$X^*_x=V(x^2-x^3y^2z)=V(x^2)\cup V(1-xy^2z),$$
$$X^*_y=V(x^2y^2-y^3z)=V(y^2)\cup V(x^2-yz),$$
$$X^*_z=V(x^2y^2-y^2z^3)=V(z^2)\cup V(1-xy^2z).$$
The blowup $\wt X$ of $X$ is given by gluing the three charts $V(1-xy^2z)$, $V(x^2-yz)$, and $V(1-xy^2z)$ of the strict transform $X^s$ of $X$ according to def.÷ \ref{blowup6}. Observe that the origin of the $y$-chart $V(x^2-yz)$ has the same singularity as $X$ at $0$. 

The blowup of $X$ along the $x$-axis yields for the blowup map the chart expressions $\pi_y(x,y,z)=(x,y,yz)$ and $\pi_z(x,y,z)=(x,yz,z)$. As the $x$-axis is not contained in $X$, the total transform $X^*$ and the strict transform $X^s=\wt X$ of $X$ coincide, cf.÷ def.÷ \ref{stricttransform}. This gives for $\wt X$ the chart expressions
$$\wt{X}_y=V(x^2-y^3z),$$
$$\wt{X}_z=V(x^2-y^2z^3).$$
For the blowup of $X$ along the $y$-axis, the total transform has chart expressions
$$X^*_x=V(x)\cup V(x-y^2z),$$
$$X^*_z=V(z)\cup V(x^2z-y^2),$$
where the secondly listed components are the charts of the strict transform. Again, the chart $V(x^2z-y^2)$ has, up to permutation of the variables, the same singularity as $X$ at $0$. The blowup of $X$ along the $z$-axis gives accordingly
$$X^*_x=V(x^2)\cup V(1-y^2z),$$
$$X^*_y=V(y^2)\cup V(x^2-z).$$


{\sc Ex.÷ \ref{4.64}.} The blowup of $X=\Spec(\ZZZ[x])$ along $I=(x,p)$ is covered by two charts with coordinate rings $\ZZZ[x, x/p]\isom \ZZZ[x,u]/(x-pu)\isom \ZZZ[u]$ and $\ZZZ[x,p/x]\isom\ZZZ[x,u]/(p-xu)$, respectively. Observe that the second chart is not equal to the affine line $\AAA^1_\ZZZ$ over $\ZZZ$, see also \cite{Eisenbud_Harris}.

For the blowup of $X$ along $I=(px,pq)$ one cannot use the equations of def.÷ \ref{blowup5} since $px$ and $pq$ do not form a regular sequence in $\ZZZ[x]$. Similarly as before, the affine charts have coordinate rings $\ZZZ[x, x/q]\isom \ZZZ[x,u]/(x-qu)\isom \ZZZ[u]$ and $\ZZZ[x,q/x]\isom\ZZZ[x,u]/(q-xu)$, respectively. \\

\ignore    

Let $Z$ be the scheme $\Spec(\ZZZ[x]/(x,p))$ and set $I=(x,p)$. The Rees algebra of $I$ equals 
$$\wt{R}=\bigoplus_{k\geq 0}\, (x,p)^kt^k.$$
The definition of blowups via Rees algebras yields $\wt{X}=\Proj(\wt{R})$ together with the map $\pi:\wt{X}\to X$ given via the homomorphism
$\alpha: R\to \wt{R}, r\mapsto r$ and $\pi(\mathfrak{p})=\alpha^{-1}(\mathfrak{p})$. The irrelevant ideal of $\wt{R}$ is 
$$\wt{R}_+=\bigoplus_{k\geq 1}(x,p)^kt^k=(xt,pt).$$
Therefore
$$\wt{X}=\Proj(\wt{R})=\{\mathfrak{p}\in\Spec(\ZZZ[x,t]), \mathfrak{p}\text{ prime, homogeneous}, \mathfrak{p}\not\supseteq (xt,pt)\}$$
$$=\{\mathfrak{p}\in\Spec(\ZZZ[x]), \mathfrak{p}\text{ prime, homogeneous}, \mathfrak{p}\not= (x,p)\}.$$

Similarly, the computation for $I=(x,pq)$ with $p$ and $q$ primes yields
$$\wt{X}=\{\mathfrak{p}\in\Spec(\ZZZ[x]), \mathfrak{p}\text{ prime, homogeneous}, \mathfrak{p}\not= (x,p) \text{ or } (x,q)\}.$$
\\
\recognize


{\sc Ex.÷ \ref{5.12}.} The equations $g_1=y^2-xz, g_2=yz-x^3$ and $g_3=z^2-x^2y$ do not form a regular sequence, since they admit the non-trivial linear relations $z\cdot g_1 -y\cdot g_2+ x\cdot g_3=0$ and $x^2\cdot g_1-z\cdot g_2+ y\cdot g_3=0$. Therefore one cannot use the equations of def.÷ \ref{blowup5} to describe the blowup of $\AAA^3$ along the curve $Z$ defined by $g_1=g_2=g_3=0$.\\


{\sc Ex.÷ \ref{5.13}.} Consider the map $\gamma: X\setminus Z\to\PPP^{1}, (x,y,z)\mapsto (x:y)$. For $x=y=0$ and $z\not=0$ there lies only one point in the closure $\overline\Gamma$ of the graph of $\gamma$ in $X\times \PPP^1$, while for $z=0$ the set of limit points forms a projective line $\PPP^1$. The blowup is a local isomorphism outside $0$ since $Z$ is locally a Cartier divisor in $X$ at these points (being locally a regular curve in a regular surface). Above $0\in X$ the blowup map $\pi:\wt X\too X$ is not a local isomorphism, since $\pi$ contracts all limit points to $0$, or, alternatively, because the blowup $\wt X$ is regular, while $X$ is singular at $0$.\\


{\sc Ex.÷ \ref{5.14}.} The blowup map $\pi:\wt \AAA^2\too\AAA^2$ with center the origin has the chart expressions $(x,y)\mapsto (x,xy)$ and $(x,y)\mapsto(xy,y)$. The total transform $X^*=\pi\inv(X)$ of $X=V(x^2)$ in $\wt \AAA^2$ has therefore charts defined in $\AAA^2$ by $x^2=0$, respectively $x^2y^2=0$, with exceptional divisors given by $x=0$ and $y=0$. Hence the strict transform $X^s=\wt X$ of $X$ lies only in the $y$-chart and is defined there by $x^2=0$.\\


{\sc Ex.÷ \ref{5.15}.} The same computation as in the preceding example applies and shows  that $\wt X$ lies entirely in the $y$ chart. It is defined there by the ideal $(x^2, x)=(x)$, equals hence the $y$-axis of this chart.\\


{\sc Ex.÷ \ref{5.17}.}  The blowup of $\AAA^3$ along the union of two coordinate axes is discussed in ex.÷ \ref{4.36} and \ref{4.66}.

The blowup $\wt \AAA^3$ of $\AAA^3$ along the cusp $(x^3-y^2,z)$ is defined in $\AAA^3\times \PPP^1$ by $uz-v(x^3-y^2)=0$, since $x^3-y^2$ and $z$ form a regular sequence. It follows that $\wt \AAA^3$ is singular at $0$.\\




{\sc Ex.÷ \ref{6.14}.} By Prop.÷ \ref{6.6}, the defining ideal of the strict transform of $X$ is generated by the strict transforms of the elements of a Macaulay basis of the ideal $(x^2-y^3,xy-z^3)$. Notice that the given generators are not a Macaulay basis since their initial forms are $x^2$ and $xy$, which do not generate the initial form of $y(x^2-y^3)-x(xy-z^3)=xz^3-y^4$. By adding this element, the Macaulay basis $x^2-y^3,xy-z^3,xz^3-y^4$ is obtained.

The chart expressions of the strict transform of $X$ can now be computed from this Macaulay basis,
$$X^s_x=V(1-xy^3,y-xz^3,z^3-y^4),$$
$$X^s_y=V(x^2-y,x-yz^3,xz^3-1),$$
$$X^s_z=V(x^2-y^3z,xy-z,x-y^4).$$\\


{\sc Ex.÷ \ref{6.18} and \ref{6.19}.} The transformation of flags under blowups is explicitly described in \cite{Ha_Power_Series} p.÷ 5--8.\\


{\sc Ex.÷ \ref{6.22}.} In the $y$-chart, the total transform of $I=(x^2,y^3)$ is given by the ideal $I^*=(x^2y^2,y^3)$. Factoring out the maximal power of the monomial defining the exceptional divisor, $I^*=(y^2)(x^2,y)$ is obtained. Thus, the weak transform of $I$ is given by the ideal $I^\curlyvee=(x^2,y)$. On the other hand, it is easy to see that $x^2,y^3$ is a Macaulay basis for $I$. Thus, by Prop.÷ \ref{6.6}, the strict transform of $I$ is generated by the strict transforms of these generators. Therefore, $I^s=(x^2,1)=\mathbb{K}[x,y]$.
\\


{\sc Ex.÷ \ref{6.23}.} This result is proved in \cite{Ha_BAMS_1}, p.÷ 345.
\\




{\sc Ex.÷ \ref{7.10}.} For the first two equations, the implicit function theorem shows that the zerosets are regular at $0$. This gives a parametrization by formal power series. The zeroset of the third equation is singular at $0$ and one cannot use the implicit function theorem to describe it at $0$. To give a parametrization requires to construct first a resolution, which, in the present case, is very tedious.\\


{\sc Ex.÷ \ref{7.14}.} Consider the case that the resolution of $Y$ is achieved by a sequence of blowups:
 \[Y'=Y_n\overset{\pi_{n-1}}{\longrightarrow}Y_{n-1}\overset{\pi_{n-2}}{\longrightarrow}\cdots \overset{\pi_{0}}{\longrightarrow}Y_0=Y.\]
Denote by $Z_i\subseteq Y_i$ the center of the blowup $\pi_i:Y_{i+1}\to Y_i$. Since $Z$ is regular, the singular locus of $Y\times Z$ is $\textnormal{Sing}(Y)\times Z$. By the base change property for blowups, Prop.÷ 5.1, the blowup of $Y_i\times Z$ along the center $Z_i\times Z$ equals $Y_{i+1}\times Z$. This gives a new sequence
 \[Y'\times Z=Y_n\times Z\overset{\wt\pi_{n-1}}{\longrightarrow}Y_{n-1}\times Z\overset{\wt\pi_{n-2}}{\longrightarrow}\cdots \overset{\wt\pi_{0}}{\longrightarrow}Y_0\times Z=Y\times Z\]
 where $\wt\pi_i:Y_{i+1}\times Z\to Y_i\times Z$ is the blowup along the center $Z_i\times Z\subseteq Y_i\times Z$. It is checked that the morphism $Y'\times Z\to Y\times Z$ is a resolution of the singularities of $Y\times Z$.\\


{\sc Ex.÷ \ref{8.25}.} Consider the variety $X=V((x^2-y^3)(z^2-w^3))\subseteq\mathbb{A}^4$ over a field of characteristic zero. The stratification of $X$ by the iterated singular loci is as follows.
$$\Sing(X)=V(x,y)\cup V(z,w)\cup V(x^2-y^3,z^2-w^3),$$
$$\Sing^2(X)=V(x,y,z^2-w^3)\cup V(z,w,x^2-y^3),$$
$$\Sing^3(X)=V(x,y,z,w).$$


{\sc Ex.÷ \ref{8.30}--\ref{8.32}.} Let $J\subseteq R$ be an ideal and $I\subseteq R$ a prime ideal. The order of $J$ along $I$ is defined as
 \[\ord_IJ=\max\{k\in\mathbb{N},\, JR_I\subseteq I^kR_I\}\]
 where $R_I$ is the localization of $R$ in $I$. If $I^{(k)}=I^kR_I\cap R$ denotes the $k$-th symbolic power of $I$, the order can also be expressed as
 \[\ord_IJ=\max\{k\in\mathbb{N},\, J\subseteq I^{(k)}\}\]
 without making explicit use of the localization \cite{Zariski_Samuel}.
 
Now consider the example $R=\mathbb{K}[x,y,z]$, $I=(y^2-xz,yz-x^3,z^2-x^2y)$. It can be checked that $I$ is a prime ideal of $R$ but not a complete intersection (i.e., not generated by a regular sequence, cf.÷ ex.÷ \ref{5.12}). Also consider the principal ideal $J$ in $R$ that is generated by $f=x^5+xy^3+z^3-3x^2yz$. The order of $J$ along $I$ can be determined as follows. First notice that $f\notin I^2$ since $f$ has order $3$ at the origin, but all elements in $I^2$ have at least order $4$ at the origin. On the other hand,
 \[xf=x^6+x^2y^3+xz^3-3x^3yz=(x^3-yz)^2-(y^2-xz)(z^2-x^2y)\in I^2.\]
 Since $x\notin I$, this implies that $J\cdot R_I\subseteq I^2R_I$, and in particular, $\ord_I J\geq 2$. Thus, $f$ is an example for an element in $R$ that is contained in the symbolic power $I^{(2)}$, but not in $I^2$. It remains to show that $\ord_I J=2$. By Thm.÷ \ref{order_under_localization} it suffices to find a point $a$ that lies on the curve $V(I)$ for which $\ord_a f=2$. Such a point is for instance $a=(1,1,1)\in V(I)$.  \\


{\sc Ex.÷ \ref{8.39}.} Consider the polynomial $g=x^c+\sum_{i=0}^{c-1}g_i(y)\cdot x^i$ at the origin $a=0$ of $\AAA^{1+n}$. The order of $g$ at $a$ equals the minimum of $c$ and all values $\ord_a g_i+i$, for $0\leq i<c$. Assume that $\ord_0g=c$ and also that $g_{c-1}=0$. If the characteristic of the ground field is zero, this can be achieved by a change of coordinates $x\mapsto x+\frac{1}{c}\cdot g_{c-1}(y)$, compare with rem.÷ \ref{9.4} and ex.÷ \ref{9.10}. The defining ideal $I$ of the top locus of $V(g)$ is generated by the derivatives $\frac{\partial^{i+|\alpha|}}{\partial x^i\partial y^\alpha}g$ where $i\in\mathbb{N}$, $\alpha\in\mathbb{N}^n$ and $i+|\alpha|<c$. In particular, if the characteristic of the ground field is zero, $\frac{\partial^{c-1}}{\partial x^{c-1}}g=c!\cdot x\in I$ and thus $x\in I$. This allows to express $I$ in the form:
 \[I=\bigg(x,\frac{\partial^{|\alpha|}}{\partial y^\alpha}g_i(y),\, \text{ for } \alpha\in\mathbb{N}^n, i<c, |\alpha|<c-i\bigg).\]


{\sc Ex.÷ \ref{8.43}.} Assume that $\ord\, g=d$ and $g=\sum_{i,j}c_{ij}y^iz^j$. The order will always be taken at the origin of the respective charts. Let $g'$ be the strict transform of $g$ under the first blowup. Then
 \[g'=\sum_{i,j}c_{ij}y^{i+j-d}z^j.\]
 Set $d'=\ord\, g'$ and notice that $d'=\min\{i+2j-d,\, c_{ij}\neq0\}$. Let $g''$ be the strict transform of $g'$ under the second blowup. Then
 \[g''=\sum_{i,j}c_{ij}y^{i+j-d}z^{i+2j-d-d'}.\]
 Set $d''=\ord\, g''$ and notice that $d''=\min\{2i+3j-2d-d',\,c_{ij}\neq0\}$. It is clear that $d''\leq d'\leq d$. If $d'\leq \frac{d}{2}$, then $d''\leq \frac{d}{2}$ follows. So assume that $d'>\frac{d}{2}$. By assumption, there exists a pair $(i,j)\in\mathbb{N}^2$ such that $i+j=d$ and $c_{ij}\neq0$. Thus,
 \[d''\leq \underbrace{2i+3j}_{=2d+j}-2d\underbrace{-d'}_{<-\frac{d}{2}}<\underbrace{j}_{\leq d}-\frac{d}{2}\leq\frac{d}{2}.\]
\\


{\sc Ex.÷ \ref{8.53}.} The situation is described in detail in \cite{Ha_Excellent}, Prop.÷ 4.5, p.÷ 354, and \cite{Zariski_1944}, Thm.÷ 1 and Lemma 3.2, p.÷ 479. If the center is a smooth curve, see \cite{Ha_Excellent}, Prop.÷ 4.6, p.÷ 354, and \cite{Zariski_1944}, Thm.÷ 2, p.÷ 484 and its corollary, p.÷ 485.\\ \goodbreak


{\sc Ex.÷ \ref{9.8}.} Let $a=0$ and pass to the completion $\wh \calo_{W,a}\isom\mathbb{K}[[x_1,\ldots,x_n]]$. Assume that $\ord_a f=c$ and that $f$ has the expansion $f=\sum_{\alpha\in\mathbb{N}^n}c_\alpha x^\alpha$. The initial form of $f$ is $\sum_{|\alpha|=c}c_\alpha x^\alpha$. It is assumed that the blowup is monomial; thus, $a'$ is the origin of the $x_i$-chart for some $i\leq n$. If $i>1$, then $a'$ is contained in the strict transform of $V^s$. It remains to show that, if $a'$ is the origin of the $x_1$-chart, the order of the strict transform of $f$ is smaller than $c$ .
 
For this, notice that the strict transform of $f$ at the origin of the $x_1$-chart is given by $f'=\sum_{\alpha\in\mathbb{N}^n}c_\alpha x_1^{|\alpha|-c-\alpha_1}x^\alpha$ where $\alpha_1$ is the first component of $\alpha$. The monomials of this expansion are distinct, so there can be no cancellation between them. By assumption, $x_1$ appears in the initial form of $f$; thus, there is an exponent $\alpha'\in\mathbb{N}^n$ such that $|\alpha'|=c$, $\alpha'_1>0$ and $c_{\alpha'}\neq0$. This implies that  the expansion of $f'$ contains the non-zero monomial $c_{\alpha'}x_1^{-\alpha_1'}x^\alpha$. Therefore $\ord_{a'}f'<c=\ord_af$.
\\


{\sc Ex.÷ \ref{9.10}.} Assume that the characteristic of the ground field is zero. The coordinate change $x_n\mapsto x_n-\frac{1}{d}\cdot a_{d-1}(y)$ transforms $f$ into
 \[f\mapsto \bigg(x_n-\frac{1}{d}\cdot a_{d-1}(y)\bigg)^d+\sum_{i=0}^{d-1}a_i(y)\bigg(x_n-\frac{1}{d}\cdot a_{d-1}(y)\bigg)^i.\]
 Notice that the coefficient of $x_n^d$ in this new expansion is $1$ and the coefficient of $x_n^{d-1}$ is $-\binom{d}{1}\frac{1}{d}a_{d-1}(y)+a_{d-1}(y)=0$. Thus, there are polynomials $\widetilde a_i(y)$ such that, in the new coordinates, 
 \[f=x_n^d+\sum_{i=0}^{d-2}\widetilde a_i(y)x_n^i.\]
 Consequently, $\frac{\partial^{d-1}}{\partial x_n^{d-1}}(f)=d!\cdot x_n$. Thus, the variety defined by $x_n=0$ defines an osculating hypersurface for $X$ at the origin. By Prop.÷ \ref{9.5}, every osculating hypersurface has maximal contact.\\


{\sc Ex.÷ \ref{9.11}.}
 The defining equation for the strict transform of $X$ in the $x$-chart is given by
 \[f'=y+y^2-z+y^2z-z^2-yz^2 =((y+1)-(z+1))(y+1)(z+1).\]
The order of $f'$ at the point $a'=(0,-1,-1)$ is $3$. Since the exceptional divisor is given by the equation $x=0$ in the $x$-chart, $a'$ lies on it. The strict transform $V^s$ of $V$ coincides with the complement of the $x$-chart, so $a'$ cannot lie on $V^s$.\\




{\sc Ex.÷ \ref{9.14}.} The relevant case appears when $c=p>0$ where $p$ is the characteristic of the ground field. Let $g(y)=\sum_{\alpha\in\mathbb{N}^m}g_\alpha y^\alpha$ be the expansion of $g$ with respect to the coordinates $y_1,\ldots,y_m$. It decomposes into
 \[g(y)=\underbrace{\sum_{\substack{\alpha\in p\cdot\mathbb{N}^m}}g_\alpha y^\alpha}_{g_1(y)}+\underbrace{\sum_{\alpha\in\mathbb{N}^m\setminus p\cdot\mathbb{N}^m}g_\alpha y^\alpha}_{g_2(y)}.\]
 If the ground field is assumed to be perfect, $g_1(y)$ is a $p$-th power. Thus, there is a formal power series $\widetilde g(y)$  with $\ord\, \widetilde g\geq1$ such that $\wt g^p=g_1$. Apply the change of coordinates $x\mapsto x-\widetilde g$. This transforms $f$  into 
 \[f=(x-\widetilde g)^p+g_1+g_2=x^p-\widetilde g^p+g_1+g_2=x^p+g_2.\]
 Notice that $\ord\, g_2\geq\ord\, g$. To show that this order is maximal, it suffices to consider an arbitrary change of coordinates $z\mapsto z+h(y)$ where $h\in \mathbb{K}[[y_1,\ldots,y_m]]$ is any power series with $\ord\, h\geq1$. This transforms the equation into
 \[f=x^p+h^p+g_2.\]
 Notice that there can be no cancellation between the terms in the expansions of $h^p$ and $g_2$. Thus, $\ord(h^p+g_2)\leq \ord\, g_2$. 
\\

{\sc Ex.÷ \ref{9.17}.} Denote the sequence of local blowups by 
 \[(W',a')=(W_m,a_m)\overset{\pi_{m-1}}{\longrightarrow}\cdots\overset{\pi_0}{\longrightarrow}(W_0,a_0)=(W,a).\]
Set $f^{(0)}=f$ and let $f^{(i)}$ be the strict transform of $f$ after $i$ blowups. Assume that the center of each local blowup $\pi_i:(W_{i+1},a_{i+1})\to(W_i,a_i)$ is contained in the top locus of $f^{(i)}$ and has normal crossings with the exceptional divisors produced by previous blowups. Let $Z$ be a local hypersurface of maximal contact for $f$ at $a$. Assume that $\ord_{a_i} f^{(i)}=\ord_a f$ for $i=1,\ldots,m$. Then, by definition of maximal contact, $a_i\in Z_i$ for $i=0,\ldots,m$, where $Z_i$ denotes the strict transform of $Z$ after $i$ blowups. Since all centers are contained in the top locus of $f^{(i)}$, they are also contained in $Z_i$. By repeatedly applying Prop.÷ \ref{199}, the union of all exceptional divisors produced by the blowups $\pi_{m-1},\ldots,\pi_0$ with $Z'=Z_m$ is a normal crossings divisor. Thus, $a'$ is contained in $n+1$ hypersurfaces that form a normal crossings divisor. But this is impossible in an ambient space of dimension $n$.\\


{\sc Ex.÷ \ref{10.20}.} Set $f=x^5+x^2y^4+y^k$. Denote by $V_x$ and $V_y$ the regular local hypersurfaces defined by $x=0$, respectively $y=0$. If $k\geq5$, then the order of $f$ at the origin is $5$ and the derivative $\partial^4_xf=5!\cdot x$ defines the hypersurface $V_x$. Thus, $V_x$ is osculating. If $k=4$, then the order of $f$ at the origin is $4$. The derivative $\partial^3_yf=4!\cdot y(1+x^2)$ defines a regular parameter. Since $1+x^2$ is a unit in the local ring $\mathcal{O}_{\mathbb{A}^2,0}$, the hypersurface $V_y$ is osculating. If $k<4$, taking the differential with respect to $y$ shows again that $V_y$ is an osculating hypersurface.
 
 If $k\geq5$, then $J_{V_x}(I)=(y^{160},y^{24k})$ and $J_{V_y}(I)=(x^{240},x^{120})=(x^{120})$. Thus, both $V_x$ and $V_y$ have weak maximal contact if $k=5$, while only $V_x$ has weak maximal contact if $k>5$. If $2\leq k\leq4$, then $J_{V_x}(I)=(y^{4\frac{k!}{k-2}},y^{k!})$ and $J_{V_y}(I)=(x^{5(k-1)!})$. Since $k!<5(k-1)!$, only $V_y$ has weak maximal contact in this case.\\


{\sc Ex.÷ \ref{10.21}.} Assume that the characteristic of the ground field is not equal to $2$. Then $\partial_xf=2x$ and the hypersurface $V$ has weak maximal contact with $f$ by Prop.÷ 10.11. The coefficient ideal $J_V(f)$ of $f$ with respect to $V$ is generated by $y^3z^3+y^7+z^7$ and has order $6$ (up to raising the generator to the required power).
 
Restrict to the case of the blowup of $\AAA^2$ at the origin and the study of $f$ at the origin of the $y$-chart. The strict transform of $f$ is given there by $f'=x^2+y^4(z^3+y+yz^7)$. The points where the order of $f'$ has remained constant are exactly the points of the $z$-axis. Let $V'$ denote the strict transform of $V$. Then $J_{V'}(f')$ is generated by $y^4(z^3+y+yz^7)$ where $y^4$ is the exceptional factor (again, up to taking powers). The order of the residual factor has dropped to $1$ at the origin and to $0$ at all other points of the $z$-axis.\\


{\sc Ex.÷ \ref{10.22}.} Assume that $a=0$, $\ord_af=c$ and that $J_V(f)=(y^\alpha)$ where $y=(x_1,\ldots,x_{n-1})$ and $\alpha\in\mathbb{N}^{n-1}$. Now let $L\subseteq\{1,\ldots,n-1\}$ be a subset such that $\sum_{i\in L}\alpha_i\geq c!$ but  $\sum_{i\in L\setminus\{j\}}\alpha_i<c!$ for all $j\in L$. This is possible since $\ord_aJ_V(f)=|\alpha|\geq c!$ by the definition of coefficient ideals.
 
Blow up $\AAA^n$ in the center given by $x_n=0$ and $x_i=0$ for all $i\in L$. Let $a'$ be a point over $a$ at which the order of the strict transform of $f$ has remained constant. Since $V$ is osculating, $a'$ is not contained in the $x_n$-chart. So let $j\in L$ be such that $a'$ is contained in the $x_j$-chart. The transform of the coefficient ideal in this chart is 
 \[J_{V'}(f')=J_V(f)^!=(x_j^{\sum_{i\in L\setminus\{j\}}\alpha_i-c!}y^\alpha).\]
 But since $\sum_{i\in L\setminus\{j\}}\alpha_i-c!<0$ by assumption, it follows that $\ord_{a'} J_{V'}(f')<\ord_a J_V(f)$. But the order of the coefficient ideal is at least $c!$ as long as the order of the strict transform of $f$ remains equal to $c$. Thus, by iterating this procedure, after finitely many steps the order of the strict transform of $f$ has to drop. \\


{\sc Ex.÷ \ref{10.27}.} Assume that $\ord\, I=c$. Let $x_1,\ldots,x_n$ be a regular system of parameters for $\widehat{\mathcal{O}}_{\mathbb{A}^n,0}=\mathbb{K}[[x_1,\ldots,x_n]]$. Define a map $\pi:\mathbb{N}^n\to\mathbb{N}\cup\{\infty\}$ in the following way:
 \[\pi(\alpha_1,\ldots,\alpha_n)=\begin{cases}
                                  \frac{c!}{c-\alpha_n}\sum_{i=1}^{n-1}\alpha_i & \text{ if $\alpha_n<c,$} \\
                                  \infty & \text{ if $\alpha_n\geq c.$}
                                 \end{cases}
\]
Let $V$ be the regular local hypersurface defined by $x_n=0$. Elements $f\in I$ have expansions $f=\sum_{\alpha\in\mathbb{N}^n}c_{f,\alpha}x^\alpha$. Then, by the definition of coefficient ideals,
 \[\ord(J_V(I))=\min\{\pi(\alpha),\alpha\in\mathbb{N}^n:\text{there is an $f\in I$ such that $c_{f,\alpha}\neq0$}\}.\]
Now define a monomial order $<_\varepsilon$ on $\mathbb{K}[x_1,\ldots,x_n]$ in the following way: Set $x^\alpha<_\varepsilon x^\beta$ if and only if $\pi(\alpha)<\pi(\beta)$ or $\pi(\alpha)=\pi(\beta)$ and $\alpha<_{lex}\beta$ where $<_{lex}$ denotes the lexicographic order on $\mathbb{N}^n$.  

By \cite{Ha_Power_Series}, Thm.÷ 3, p.÷ 10, there exists a regular system of parameters $x_1,\ldots,x_n$ for $\widehat{\mathcal{O}}_{\mathbb{A}^n,0}$ such that the initial ideal of $I$ with respect to $<_\varepsilon$ is maximal (again, with respect to $<_\varepsilon$) over all choices of regular systems of parameters for $\widehat{\mathcal{O}}_{\mathbb{A}^n,0}$. In particular, let $y_1,\ldots,y_n$ be another regular system of parameters and let each $f\in I$ have the expansion $f=\sum_{\alpha\in\mathbb{N}^n}c'_{f,\alpha}y^\alpha$ with respect to these parameters. Then there exist elements $g\in I$ and  $\wt\alpha\in\mathbb{N}^n$ such that $c'_{g,\wt\alpha}\neq0$ and
\[\pi(\wt\alpha)\leq \min\{\pi(\alpha),\alpha\in\mathbb{N}^n,\, \text{there is an $f\in I$ such that $c_{f,\alpha}\neq0$}\}.\]
Let $V'$ be the regular local hypersurface defined by $y_n=0$. Then the last statement implies that $\ord(J_{V'}(I))\leq\ord(J_V(I))$. Thus, the regular system of parameters $x_1,\ldots,x_n$ maximizes the order of the coefficient ideal.\\


{\sc Ex.÷ \ref{10.30}.} Let $I$ be an ideal with $\ord_aI=\ord_{a'}I^\curlyvee$ where $a'$ is the origin of the $x_j$-chart for some $j<n$ and $I^\curlyvee$ denotes the weak transform of $I$. Elements $f\in I$ have  expansions $f=\sum_{i\geq0}f_i(y)x_n^i$ where $y=(x_1,\ldots,x_{n-1})$. Then 
\[J_V(I)=(f_i^\frac{o!}{o-i},\, i<o,f\in I).\]
The weak transform of $I$ is given by $I^\curlyvee=(f^\curlyvee,f\in I)$ where 
\[f^\curlyvee=x_j^{-o}f^*=x_j^{-o}\sum_{i\geq0}f_i^*x_n^ix_j^i=\sum_{i\geq0}\underbrace{(f_i^*x_j^{i-o})}_{=:f_i'}x_n^i.\]
Here, $f^*$ and $f_i^*$ denote the total transforms. Thus, the coefficient ideal of $I^\curlyvee$ with respect to the strict transform $V'$ of $V$ is
\[J_{V'}(I^\curlyvee)=(f_i'^{\frac{o!}{o-i}}:i<o,f\in I)=({f_i^*}^{\frac{o!}{o-i}}x_j^{(i-o)\frac{o!}{o-i}}:i<o,f\in I)\]
\[x_j^{-o!}({f_i^*}^{\frac{o!}{o-i}}:i<o,f\in I)=x_j^{-o!}(J_V(I))^*=J_V(I)^!\]
where $J_V(I)^!$ denotes the controlled transform with control $o!$.
\\


{\sc Ex.÷ \ref{11.13} and \ref{11.14}.}
 These examples are discussed in detail in Lecture XII.
\\


{\sc Ex.÷ \ref{12.17}--\ref{12.22}.}
 Let $\mathbb{K}$ be a field of characteristic $2$. Consider the ring homomorphism $\phi:\mathbb{K}[x,y,z,w]\to \mathbb{K}[t]$ given by $\phi(x)=t^{32}$, $\phi(y)=t^7$, $\phi(z)=t^{19}$, $\phi(w)=t^{15}$. Let $I\subseteq\mathbb{K}[x,y,z,w]$ be its kernel. Its zeroset is an irreducible curve $C$ in $\mathbb{A}^4$.
 
 Now let $f$ be the polynomial $f=x^2+yz^3+zw^3+y^7w$. The partial derivatives of $f$ have the form:
 \[\frac{\partial}{\partial x}f=0,\]
 \[\frac{\partial}{\partial y}f=z^3+y^6w,\]
 \[\frac{\partial}{\partial z}f=yz^2+w^3,\]
 \[\frac{\partial}{\partial w}f=zw^2+y^7.\]
 It is easy to check that $f$ and all of its first derivatives are contained in $I$. Thus, by the Jacobian criterion, $\ord_If\geq 2$. Since the order of $f$ in the origin is $2$, it is possible to conclude that $\ord_If=2$. In other words, the top locus of $f$ contains the curve $C$ that is parametrized by $t\mapsto(t^{32},t^7,t^{19},t^{15})$. It can be checked, e.g. via any computer algebra system, that the top locus of $f$ is itself an irreducible curve. Thus, it coincides with the curve $C$.
 
Now suppose that there is a regular local hypersurface at the origin that contains the top locus of $f$ and hence the curve $C$. This hypersurface has an equation $h=0$ in which at least one variable must appear linearly. But since $h\in I$, this is only possible if one of the numbers $32,7,19,15$ can be written as an $\mathbb{N}$-linear combination of the others. This is not the case so that there is no regular local hypersurface containing the top locus of $f$.
 
Let $V$ be any regular local hypersurface at the origin. Since it does not contain the curve $C$, there is a sequence of point blowups that separates the strict transforms of $V$ and $C$. Since the variety $X$ defined by $f=0$ has order $2$ along the curve $C$ and the point blowups are isomorphisms over all but one point of $C$, the order of the strict transform of $X$ under these blowups will again be $2$ along the strict transform of $C$. So there eventually is a point $a'$ at which the strict transform of $X$ has order $2$, but which is not contained in the strict transform of $V$. Thus, $V$ does not have maximal contact with $X$. Since $V$ was chosen arbitrarily, there can be no hypersurface that has maximal contact with $X$. \\


{\sc Ex.÷ \ref{12.24}.} Let $\alpha\in\mathbb{N}^n$. Define the differential operator $\partial_{x^\alpha}$ on $\mathbb{K}[x_1,\ldots,x_n]$ as the linear extension of $\partial_{x^\alpha} x^\beta=\binom{\beta}{\alpha}x^{\beta-\alpha}$.
 
 In particular, consider the differential operator $\partial_{x_i^p}$ for $i\in\{1,\ldots,n\}$. If $n\in\mathbb{N}$ has the $p$-adic expansion $n=\sum_{i\geq0}n_i p^i$ with $0\leq n_i<p$, then $\partial_{x_i^p}x_i^n=\binom{n}{p}x_i^{n-p}=n_1x_i^{n-p}$. Similarly, $\partial_{x_i^{p^k}}x_i^n=n_kx_i^{n-p^k}$ for any $k\in\mathbb{N}$.
 
 Notice that $\partial_{x_i^p}$ is not a derivation since it does not fulfill the Leibniz rule: $\partial_{x_i^p}(x_i^p)=1$, but
 \[x_i\underbrace{\partial_{x_i^p}x_i^{p-1}}_0+x_i^{p-1}\underbrace{\partial_{x_i^p}x_i}_0=0.\]
 
For more details on differential operators in positive characteristic, see \cite{Kawanoue_IF_1} Chap.÷ 1, pp.÷ 838--851.
\\


{\sc Ex.÷ \ref{12.25}.} Consider $f=z^{3^5}+x^{4\cdot3^4}y^{3^4}(x^{3^4}+y^{3^4}+x^{300})$ as a polynomial over a field of characteristic $3$. Assume that the exceptional locus is given locally by $V(xy)$. Then its residual order with respect to the local hypersurface defined by $z=0$ is $3^4$. Now consider the blowup of $\AAA^2$ at the origin and let $a'$ be the point $(0,1,0)$ in the $x$-chart. The strict transform of $f$ has the equation
 \[f'=z^{3^5}+x^{3^5}(y^{3^4}+1)(y^{3^4}-1+x^{219})=z^{3^5}+x^{3^5}(y^{2\cdot 3^4}-1+x^{219}(1+y^{3^4})).\]
 By making the change of coordinates $z\mapsto z+x$, this transforms into
 \[f'=z^{3^5}+x^{3^5}(y^{2\cdot 3^4}+x^{219}(1+y^{3^4})).\]
 The exceptional divisor is given locally by $V(x)$. Thus, the residual order of $f'$ with respect to the hypersurface defined by $z=0$ is $2\cdot 3^4>3^4$.\\


{\sc Ex.÷ \ref{12.26}.} The equation of $f$ in local coordinates $x,y,z$ at a point $(0,0,t)$ on the $z$-axis is obtained by making a translation $z\mapsto z+t$. Thus, the equation $x^p+y^p(z+t)=x^p+y^pz+y^pt$ is obtained. Assume that the ground field is perfect and let $\lambda$ be the $p$-th root of $t$. Then one can write $x^p+y^pz+y^pt=(x+\lambda y)^p+y^pz$. Thus, the residual order of $f$ at every point of the $z$-axis is $p+1$.
 
Now consider the generic point $P=(x,y)$ of the $z$-axis. The order of the coefficient ideal along $P$ is $p$.\\


{\sc Ex.÷ \ref{12.27}.} Observe that
 \[G^+(y)=\prod_{i=1}^{m-1}(y_i+y_m)^{r_i}y_m^{r_m}g^+(y).\]
Let $I$ denote the ideal $(y_1,\ldots,y_{m-1})$. Notice that $I$ defines a complete intersection. By \cite{Hochster_73} 2.1, p.÷ 57, \cite{Pellikaan}, Prop.÷ 1.8, p.÷ 359, the order of a polynomial $f$ along $I$ can be expressed as the largest power of $I$ that contains $f$. In particular, it fulfills that $\ord_I(f\cdot g)=\ord_If+\ord_Ig$ for polynomials $f,g$. Now compution yields
 \[\ord_IG^+(y)=\sum_{i=1}^{m-1}r_i\cdot\underbrace{\ord_I(y_i+y_m)}_0+r_m\cdot\underbrace{\ord_Iy_m}_0+\ord_Ig^+(y)\]
 \[\leq\ord\, g^+(y)=\ord\, g(y)=k.\]


{\sc Ex.÷ \ref{12.29} and  \ref{12.30}.}
 This is explained in detail in \cite{Ha_BAMS_2}.
\\
\goodbreak


\comment{[[D] E. C. Dade, Multiplicity and monoidal transformations, Thesis, Princeton University, 1960. \cite{Abhyankar_Book_59,
Abhyankar_Book_98,
Abhyankar_56,
Abhyankar_64,
Abhyankar_66,
Abhyankar_88,
Abhyankar_Moh,
Abhyankar_66b,
Abhyankar_67,
Artin,
Benito_Villamayor,
BM_Canonical_Desing,
BM_Functoriality,
Bravo_Villamayor_10,
Bravo_Villamayor_11,
Cossart_Polyedre,
Cossart_Excellent,
Cossart_WMC,
CJS,
Cossart_Piltant_1,
Cossart_Piltant_2,
Cutkosky_3-Folds,
EV_Good_Points,
Giraud,
Ha_Wild,
HS_Game,
HW,
Hironaka_Bowdoin,
Kawanoue_IF_1,
Kawanoue_Matsuki_Surfaces,
Kuhlmann,
Levine,
Lipman_Surfaces,
Mulay,
Spivakovsky,
Szabo,
Teissier,
Temkin,
Villamayor_Constructiveness,
Villamayor_Hypersurface,
Wagner_Thesis,
Wlodarczyk,
Zariski_Local_Uniformization,
Zeillinger_Thesis,
Zeillinger_Polyhedra_Game} ]}


\nocite{*}

\bibliographystyle{amsalpha}
\bibliography{bibliography}


\vskip .6cm
\noindent  Fakult\"at f\"ur Mathematik, 

\noindent  University of Vienna, and

\noindent Institut f\"ur Mathematik, 

\noindent  University of Innsbruck, 

\noindent Austria

\noindent herwig.hauser@univie.ac.at
\end{document}